\documentclass[12pt]{aims}
\usepackage{amssymb,color}

\usepackage{graphicx}
\usepackage{appendix}
\usepackage{url}
\usepackage{epstopdf}
\usepackage[colorlinks,linkcolor=blue,breaklinks=true,citecolor=blue,pdfstartview=FitH]{hyperref}
\usepackage{subfig}
\def\currentvolume{X}
 \def\currentissue{X}
  \def\currentyear{200X}
   \def\currentmonth{XX}
    \def\ppages{X--XX}
     \def\DOI{xx.xx.xx.xx}
\newtheoremstyle{mystyle}
  {}
  {}
  {\itshape}
  {}
  {\bfseries}
  {.}
  { }
  {}
  
\newtheoremstyle{myremstyle}
  {}
  {}
  {}
  {}
  {\bfseries}
  {.}
  { }
  {}

\theoremstyle{mystyle}
\newtheorem{theorem}{Theorem}[section]

\newtheorem{definition}[theorem]{Definition}

\theoremstyle{definition}

\theoremstyle{remark}
\theoremstyle{myremstyle}
\newtheorem{remark}[theorem]{Remark}
\usepackage{fullpage}
\graphicspath{{./figures/}}

\newcommand\PG[1]{{#1}}
\newcommand\PGn[1]{{#1}}
\newcommand\KE[1]{{#1}}

\title[Optimal transport over nonlinear systems]{Optimal transport over nonlinear systems via infinitesimal generators on graphs}

\author[Karthik Elamvazhuthi and Piyush Grover]{}
\thanks{This work was funded by Mitsubishi Electric Research Labs.}
\thanks{$^*$ Corresponding author}
\begin{document}
\maketitle

\centerline{\scshape Karthik Elamvazhuthi}
\medskip
{\footnotesize
 \centerline{Arizona State University}
   \centerline{ Tempe, AZ, USA}
} 

\medskip

\centerline{\scshape Piyush Grover $^*$}
\medskip
{\footnotesize
 \centerline{ Mitsubishi Electric Research Labs}
   \centerline{Cambridge, MA, USA}
}

\bigskip

 \centerline{(Communicated by the associate editor name)}

\begin{abstract}
We present a set-oriented graph-based computational framework for continuous-time optimal transport over nonlinear dynamical systems.  We recover provably optimal control laws for steering a given initial distribution in phase space to a final distribution in prescribed finite time for the case of non-autonomous nonlinear control-affine systems, while minimizing a quadratic control cost. The resulting control law can be used to obtain approximate feedback laws for individual agents in a swarm control application. Using infinitesimal generators, the optimal control problem is reduced to a modified Monge-Kantorovich optimal transport problem, resulting in a convex Benamou-Brenier type fluid dynamics formulation on a graph. The well-posedness of this problem is shown to be a consequence of the graph being strongly-connected, which in turn is shown to result from controllability of the underlying dynamical system. Using our computational framework, \PGn{we study optimal transport of distributions where the underlying dynamical systems are chaotic, and non-holonomic. The solutions to the optimal transport problem elucidate the role played by invariant manifolds, lobe-dynamics and almost-invariant sets in efficient transport of distributions in finite time.} Our work connects set-oriented operator-theoretic methods in dynamical systems with optimal mass transportation theory, and opens up new directions in design of efficient feedback control strategies for nonlinear multi-agent and swarm systems operating in nonlinear ambient flow fields. 
\end{abstract}

\section{Introduction} \label{section:Introduction}
Understanding, computing and controlling phase space transport is of utmost importance in the study of nonlinear dynamical systems. For computation of phase space transport in dynamical systems, the available techniques can be divided into roughly three classes: geometric, topological and statistical (or operator theoretic) methods. 

The geometric methods, originating in Poincar\'{e}'s \cite{poincare1893methodes} work in celestial mechanics, aim at extracting structures in phase space that organize transport. In recent years, the focus in this field has been on extracting the Lagrangian coherent structures in autonomous and non-autonomous systems, which are often the stable and unstable manifolds \cite{wiggins2013chaotic} of fixed points or periodic orbits, or their time-dependent analogues \cite{haller2015lagrangian}.  Related techniques based on lobe-dynamics \cite{wiggins2013chaotic} allow for quantifying transport between different weakly mixing regions in the phase space. Application to low-dimensional systems arising in fluid kinematics \cite{ottino2004introduction, wiggins2004foundations}, celestial mechanics \cite{RossBook}, and plasma physics \cite{Meiss1992} have been developed over the years. Once these Lagrangian structures have been identified, intelligent control strategies can be formulated to obtain efficient phase space transport between the desired regions in the phase space; see Refs. \cite{RossBook,senatore2008fuel,sabuco2012dynamics,vaidya2004controllability,vainchtein2006passage} for some recent work in this area. The topological methods \cite{gilmore1998topological} provide rigorous bounds and sharp estimates of certain transport related quantities. For example, such methods have been applied in the the study of passive scalar mixing in laminar fluid flows \cite{BoArSt2000,thiffeault2006topology}. Some topological optimal control problems have also been studied \cite{finn2011topological}. 

The operator-theoretic statistical techniques \cite{LasotaMackey} are based on lifting the evolution from the state space to the space of measures, in case of the Perron-Frobenius (or transfer) operator, and to the space of observables, in case of the Koopman operator. In both cases, the lifted dynamical system is linear, albeit in infinite dimensions. The linearity allows for immediate application of techniques from linear algebra and linear systems theory. Numerical methods based on operator theory have been developed, and applied to several problems of contemporary interest \cite{Dellnitz98onthe, DeFrJu2001,budivsic2012applied,bollt2013applied,mezic2013analysis}. Using efficient phase space discretization techniques, these methods enable discovery of `coherent sets' in autonomous \cite{FrDe2002} and non-autonomous \cite{FrSaMo2010} dynamical systems. Recent work has also shown that combining the statistical methods with geometric \cite{DeJuKoLeLoMaPaPrRoTh2005,FrPa09,tallapragada2013set}, or topological methods \cite{grover2012topological,stremler2011topological} can give further qualitative and quantitative information about phase space transport. 

The set-oriented numerical methods for computing transfer operators \cite{DeFrJu2001, bollt2013applied} are especially well-suited for developing rigorous methods for control, e.g. see Refs. \cite{junge2004set,ross2009optimal,jerg2013global} for applications to control of individual trajectories. In Refs. \cite{vaidya2008lyapunov,vaidya2010nonlinear, raghunathan2014optimal}, an optimal control framework for asymptotic feedback stabilization of arbitrary initial measure to an attractor is presented. Also relevant is the work in the area of occupation measures, see Ref. \cite{lasserre2008nonlinear}. In fluid mechanics, such techniques have been exploited for optimal control of mixing passive scalars in fluids, see Refs. \cite{froyland2015optimal,froyland2016optimal}. 

\PGn{For control problems, the transfer operator based framework has a dual interpretation. The density control problem of a passive scalar (as in fluid mechanics), and the control of distribution of a continuum of dynamic agents (as in swarm robotics etc.) can both be studied in this setting. The difference lies in the definition of admissible controlled vector fields. In the former case, the controlled vector field is chosen to be among a set of physically meaningful vector fields (e.g. satisfying the incompressibility condition), and/or is governed via a dynamic equation (e.g. Navier-Stokes). 

In the latter case that is the focus of our work, the individual agents are being governed by a nonlinear control system. The distribution of agents' states is described by a time-varying measure over the phase space of a single agent. The controlled velocity field governing the evolution of this measure is restricted to the set of all vector fields that result from controlled motion of individual agents. In this paper, we develop a set-oriented framework for the problem of continuous-time `optimal transport' \cite{villani2003topics} of measure under such controlled nonlinear dynamics. This problem involves optimally steering an initial measure on a phase space $X$, to a final measure in given finite time. Specifically, we consider nonlinear control-affine systems of the form,
\begin{align}
\dot{x}(t) = g_0(x(t),t) +\sum_{i=1}^n u_i(t)g_i(x(t)), \label{eq:sys1}
\end{align}
where $X$ is $d-$dimensional, and $n$ is the number of control inputs.
Our aim is to compute controls $u_i$ such that a cost of transporting a measure $\mu_{t_0}$ to $\mu_{t_f}$ over the time-horizon $[t_0,t_f]$ is minimized. This cost is given by the integral over phase-space and time,
\begin{align}
C=\int_X\int_{t_0}^{t_f}\sum_{i=1}^n |u_i(x,t)|^2dt\:d\mu(x).
\end{align}

A major motivation for studying this problem comes from the field of multi-agent systems or swarm control. The problem of path planning and control of a swarm of homogenous agents in an ambient nonlinear flow field can be formulated as optimal transport problem in the presence of nonlinear dynamics. For instance, the control of magnetic particles in blood stream \cite{ghosh2009controlled,cheang2014multiple,peyer2013bio}, robotic bees in air \cite{elamvazhuthi2015optimal,wood2013flight}, and swarms of autonomous underwater vehicles (AUVs) in the ocean \cite{lermusiaux2015science} can all be studied in this setting.  This problem also arises naturally in the realm of nonlinear control systems of a single `agent', where its initial and final states can only be specified as probability distributions. In this case, the measures involved are probability measures, and hence, the optimal transport cost $C$ is the \emph{expectation} of control cost over all possible initial and final states. 

The field of optimal mass transportation \cite{villani2003topics} is concerned with optimal mapping of measures in different settings, including on graphs \cite{maas2011gradient}, and has deep connections with phase space transport in dynamical systems \cite{figalli2010mass,bernard2004optimal,ambrosio2008gradient}. The problem of optimal transport in linear dynamical systems has been studied recently \cite{hindawi2011mass, chen2017optimal,chen2016relation}, resulting in several theoretical and computational advances. 

In our previous work \cite{grover2018optimal}, we studied the problem of obtaining optimal perturbations in discrete time for systems with nonlinear dynamics, that result in desired measure transport. The perturbations were modeled as instantaneous, and computed by solving a Monge-Kantorovich optimal transport problem on a graph in \emph{pseudo-time}. Furthermore, full controllability was assumed in computation of the pseudo-time optimal transport. Hence, the evolution of the measure resulted from switching between the uncontrolled dynamics in (real) time, and the continuous pseudo-time optimal transport. 

\paragraph{\textbf{Contributions:}}In this work, we develop a set-oriented graph-based computational framework for continuous-time optimal transport over nonlinear dynamical systems of the form given in Eq. (\ref{eq:sys1}). The solution of this problem provides an open-loop control for the measure, and an approximate feedback control for individual agents, to move from an initial to final measure in finite time. Using infinitesimal generators, the optimal control problem is reduced to a modified Monge-Kantorovich optimal transport problem, resulting in a convex Benamou-Brenier type fluid dynamics formulation on a graph. We show that the well-posedness of the  resulting optimal transport problem on this graph is related to controllability of the underlying dynamical system. We prove that if the underlying dynamical system is controllable, the graph obtained in our formulation is strongly-connected. It is then proved that arbitrary final measures not lying on the boundary of the probability simplex can be reached in finite-time.

This work extends our previous work \cite{grover2018optimal} in several directions. First, we work in continuous-time, and as a result, the (passive) dynamics and the control act on the measure concurrently (rather than in a switching fashion). This also removes the requirement in Ref. \cite{grover2018optimal} that the control acts on faster time scales than the passive dynamics. Second, rather than assuming full-controllability of measures, we rigorously relate the controllability in the space of measures on a graph to the controllability properties of the underlying (single-agent) dynamical system in continuous phase space. Third, we obtain an algorithm to obtain approximate feedback laws for individual agents from the open-loop solutions of the optimal transport problem.

Using this framework, we compute optimal transport of distributions where the underlying dynamical systems are chaotic (periodic double-gyre), and non-holonomic (unicycle). The application to periodically driven double-gyre rigorously elucidates the role of invariant manifolds, lobe-dynamics and almost-invariants sets in efficient finite-time transport of distributions in the phase space. }

%
%
%

\section{Background and Mathematical Preliminaries}\label{sec:OT}
We briefly review concepts from control systems theory, optimal transport and set-oriented numerical methods relevant to the discussion in Section \ref{sec:Algo}. Specifically, we motivate the developments of Section \ref{sec:Algo} by relating the continuous and discrete (graph-based) concepts of optimal transport in controlled dynamical systems.

\subsection{Optimal Transport in Controlled Dynamical Systems}
The Monge-Kantorovich optimal transport (OT) problem \cite{villani2003topics} is concerned with mapping of an initial measure $\mu_0$ on a space $X$ to a final measure $\mu_1$ on a space $Y$. In the original formulation, it involves solving for a measurable transport map $T:X\rightarrow Y$, which pushes forward $\mu_0$ to $\mu_1$ in an optimal manner. The cost of transport per unit mass is prescribed by a function $c(x,T(x))$. Hence, the optimization problem is 
\begin{align}
\inf_T \int c(x,T(x))d\mu_0(x),\label{eq:OT1}\\
\text{ s.t. } T_{\#}\mu_0=\mu_1, \nonumber 
\end{align}
where $T_{\#}$ is the pushforward of $T$, i.e. $(T_{\#}\mu)(A)=\mu(T^{-1}(A))$ for every $A$.
 In a `relaxed' version of this problem, due to Kantorovich, the optimization problem is to obtain an optimal joint distribution $\pi(X\times Y)$ on the product space $X \times Y$, where the marginal of $\pi$ on $X$ is $\mu_0$ and on $Y$ is $\mu_1$. We denote by $\prod(\mu_0,\mu_1)$ the set of all measures on product space with the marginals $\mu_0$ and $\mu_1$ on $X$ and $Y$ respectively. Hence, the relaxed problem is 
\begin{align}
\inf_{\pi(X\times Y)\in\prod(\mu_0,\mu_1)} \int c(x,y)d\pi(x,y). \label{eq:OT2}
\end{align}

For the case of quadratic costs, i.e., $c(x,y)=\|x-y\|^2$, the support of the optimal distribution $\pi(X\times Y)$ is the graph of the optimal map $T$ obtained from the solution of problem \ref{eq:OT1}. The square-root of the optimal cost obtained as solution of this problem is called the $2-$Wasserstein distance, and we denote it by $W_2(\mu_0,\mu_1)$. We concern ourselves with only quadratic cost in this paper.

An alternative fluid dynamical interpretation of OT problem was provided by Brenier-Benamou \cite{benamou2000computational}. 
In this approach, the optimization problem is formulated in terms of an advection field $u(x,t)$, and initial and final \emph{densities} $(\rho_0(x),\rho_1(x))$ of a passive scalar. The core idea is to obtain the optimal map $T$ as a result of advection over a `time' period $(t_0,t_f)$ by an optimal advection field $u(x,t)$. It can be shown that the optimization problem given by Eq. (\ref{eq:OT1}) (with $X=Y=\mathbb{R}^d$) with quadratic cost is equivalent to the following problem:
\begin{align}
W_2^2(\mu_0,\mu_1)=\inf_{u(x,t),\rho(x,t)} \int_{\mathbb{R}^d}\int_{t_0}^{t_f} \rho(x,t)|u(x,t)|^2dt dx,\label{eq:OT_BB}\\
\text{ s.t. } \frac{\partial\rho(x,t)}{\partial t}+\nabla \cdot(\rho(x,t)u(x,t))=0,\label{eq:OT_adv} \\
\rho(x,t_0)=\rho_0(x), \rho(x,t_f)=\rho_1(x).\nonumber
\end{align}

The motion of a passive scalar is governed by the ordinary differential equation of the single integrator,
\begin{align}
\dot{x}(t)=u(x,t)\label{eq:passive}.
\end{align}
%

By a change of variables from $(\rho,u)$ to $(\rho,m\overset{\Delta}{=}\rho u)$, the optimization problem in Eq. (\ref{eq:OT_BB}) can be put into a form where its convexity can be proved easily. The transformed convex optimization problem is 

\begin{align}
\inf_{\rho(x,t) \KE{\geq 0},m(x,t)} \int_{\mathbb{R}^d}\int_{t_0}^{t_f} \frac{|m(x,t)|^2}{\rho(x,t)}dt dx,\label{eq:OT3}\\
\text{ s.t.  } \frac{\partial\rho(x,t)}{\partial t}+\nabla \cdot (m(x,t))=0,\PG{\:\: t_0\leq t\leq t_f,}\nonumber\\
\rho(x,t_0)=\rho_0(x), \rho(x,t_f)=\rho_1(x).\nonumber
\end{align}
%

\PGn{The basic theory of generalization to general nonlinear controlled dynamical systems $\dot{x}(t)=f(x(t),u(t))$, has been developed in Refs. \cite{agrachev2009optimal, rifford2014sub}. This problem can be interpreted as finding optimal control which steers an initial scalar density to a final density, where the scalar transport occurs according to a controlled dynamical system $f(x(t),u(t))$}.

For the special case of linear dynamical systems with quadratic cost, mirroring the optimal control case, further analytical development and computational simplification has been made \cite{hindawi2011mass, chen2017optimal}. As described in Ref. \cite{chen2017optimal}, consider the following setup:
\begin{align}
c(x_1,x_2)= \inf_{\mathbb{U}_{x_1}^{x_2}}\int_{t_0}^{t_f} \frac{1}{2}\|u\|^2dt,\label{eq:OT_LQ1}\\
\dot{x}(t)=A(t)x(t)+B(t)u(t), \PG{\:\: t_0\leq t\leq t_f,}\label{eq:OT_LQ2}\\
x(t_0)=x_1,x(t_f)=x_2\label{eq:OT_LQ3}.
\end{align}

The generalization of Benamou-Brenier approach to the corresponding optimal transport problem can be seen to be the following:
\begin{align}
\inf_{u(x,t),\rho(x,t)} \int_{\mathbb{R}^d}\int_{t_0}^{t_f} \rho(x,t)|u(x,t)|^2dt dx\label{eq:OT_BB1},\\
\text{ s.t. } \frac{\partial\rho(x,t)}{\partial t}+\nabla \cdot ((A(t)x(t)+B(t)u(x,t))\rho(x,t))=0,\PG{\:\: t_0\leq t\leq t_f,}\nonumber \\
\rho(x,t_0)=\rho_0(x), \rho(x,t_f)=\rho_1(x)\nonumber.
\end{align}

%

We note that the optimal transport problem given by Eq. (\ref{eq:OT_BB1}) can also be interpreted as the problem of optimally steering a dynamical system from a probabilistic initial state to a probabilistic final state. Note that the dynamics of the system are still taken to be deterministic; however see Ref. \cite{chen2016relation} for connections with stochastic dynamical systems. For the purpose of studying controlled measure transport in nonlinear systems, we use tools from operator theory, which are discussed next.

\subsection{Transfer Operator and Infinitesimal Generator} \label{sec:TOaIG}
Consider the flow-map $\phi_{t_0}^{t_0+T}:X\rightarrow X$ on a $d$-dimensional phase space $X$. This map may be obtained as a time-$T$ map of the flow of a possibly time-dependent dynamical system,
\begin{align}
\dot{x}=f(x,t).\label{eq:ode1}
\end{align}

The corresponding Perron-Frobenius transfer operator \cite{LasotaMackey} $P_{t_0}^{t_0+T}$  is a linear operator which pushes forward measures in phase space according to the dynamics of the trajectories under $\phi_{t_0}^{t_0+T}$. Let $\mathbf{B}(X)$ denote $\sigma-$algebra of Borel sets in $X$. Then, for any measure $\mu$,
\begin{align}
P_{t_0}^{t_0+T}\mu(A)=\mu((\phi_{t_0}^{t_0+T})^{-1}(A)) \:\:\:\: \forall A\in \mathbf{B}(X).
\end{align}
The transfer operator lifts the evolution of the dynamical systems from phase space $X$ to the space of measures $\mathbf{M}(X)$. Numerical approximation of $P$, denoted by $\hat{P}$, may be viewed as a transition matrix of an $N$-state Markov chain \cite{bollt2013applied}. For computation, we partition the phase space volume of interest into $N$ $d-$dimensional connected, positive volume subsets, $B_1,B_2,\dots,B_N$ with piecewise smooth boundaries $\partial B_i$. Usually, these subsets are hyperrectangles. The matrix $\hat{P}=\{\hat{p}_{ij}\}$ is numerically computed via the Ulam-Galerkin method \cite{ulam2004problems,bollt2013applied}, as 
follows
\begin{align}\label{matrix_entries}
\hat{p}_{ij}=\frac{\bar{m}\left((\phi_{t_0}^{t_0+T})^{-1}(B_{i})\cap B_{j}\right)}{\bar{m}(B_{j})},
\end{align}
where $\bar{m}$ is the Lebesgue measure. 
The action of the transfer operator over a finite time $T$ can also be defined naturally on densities in the case of Lebesgue absolutely continuous measures. However, we are more interested in capturing the continuous-time behavior of the dynamical system in Eq. (\ref{eq:ode1}) in the space of densities. The continuity equation for system in Eq. (\ref{eq:ode1}), is given by

\begin{align}\label{eq:cont}
\frac{d\mu}{dt}=-\nabla\cdot(f(x,t)\mu).
\end{align}

For the numerical approach used in this paper we briefly consider the Eq. (\ref{eq:cont}), in a operator theoretic framework,  as an abstract ordinary differential equation in the space of measures, formally. Eq. (\ref{eq:cont}) can be expressed as
\begin{equation}\label{eq}
\dot{\mu}(t) = \mathcal{A}(t)\mu  \hspace{2mm} ; \hspace{2mm} \mu(s) = \mu_s \in \mathbf{M}(X),
\end{equation}
where $\mathcal{A}(t) : D(\mathcal{A}(t) \rightarrow \mathbf{M}(X))$, $D(\mathcal{A}(t)) \subset \mathbf{M}(X)$ and the solution, $\mu(t),$ of Eq. (\ref{eq}) can be expressed using a two-parameter semigroup of operators $(\mathcal{U}(t,s)_{s,t \in \mathbb{R}, t \geq s}$ as $\mu(t) = \mathcal{U}(t,s)\mu_s$. 
The divergence operation is to be understood in the sense of duality of $M(X)$ with $C(X)$ (assuming $X$ is compact). Here $C(X)$ refers to the space of continuous functions on $X$. The Perron-Frobenius operator is related to this two-parameter semigroup of operators as $\mathcal{U}(T,t_0) = P^{t_0+T}_{t_0}$ for given parameters $t_0$ and $T$. In general, guaranteeing the existence of a strongly continuous two-parameter semigroup based on the time-dependent generator $\mathcal{A}(t)$ is quite involved. See for example Refs. \cite{engel1999one,fattorini1984cauchy}. In contrast,  the theory is more well-developed for the case when $\mathcal{A}(t) \equiv \mathcal{A}$, (the vector field $f(\mathbf{x})$ is time independent). In this case, the solution, $\mu(t)$, can be expressed by a one-parameter semigroup of bounded operators, $(\mathcal{T}(t))_{t \geq 0}$, as $\mu(t)=\mathcal{T}(t-s)\mu_s$.  Here, the generator $\mathcal{A}$ and $\mathcal{T}(t)$ are related by the formula

\begin{equation}
A \mu = \lim_{h \rightarrow 0^+} \frac{ \mathcal{T}(h)\mu- \mu}{h} \hspace{2mm} \text{for} \hspace{2mm} \text{each} \hspace{2mm} \mu \in D(\mathcal{A}).
\end{equation}

As in the case of the Perron-Frobenius operator, one can also consider the semigroup and its generator on a space of densities (or equivalently on a space of measures absolutely continuous with respect to a reference measure with additional regularity restrictions).

\begin{figure}[h!]
\centering
\includegraphics[width=2.5in]{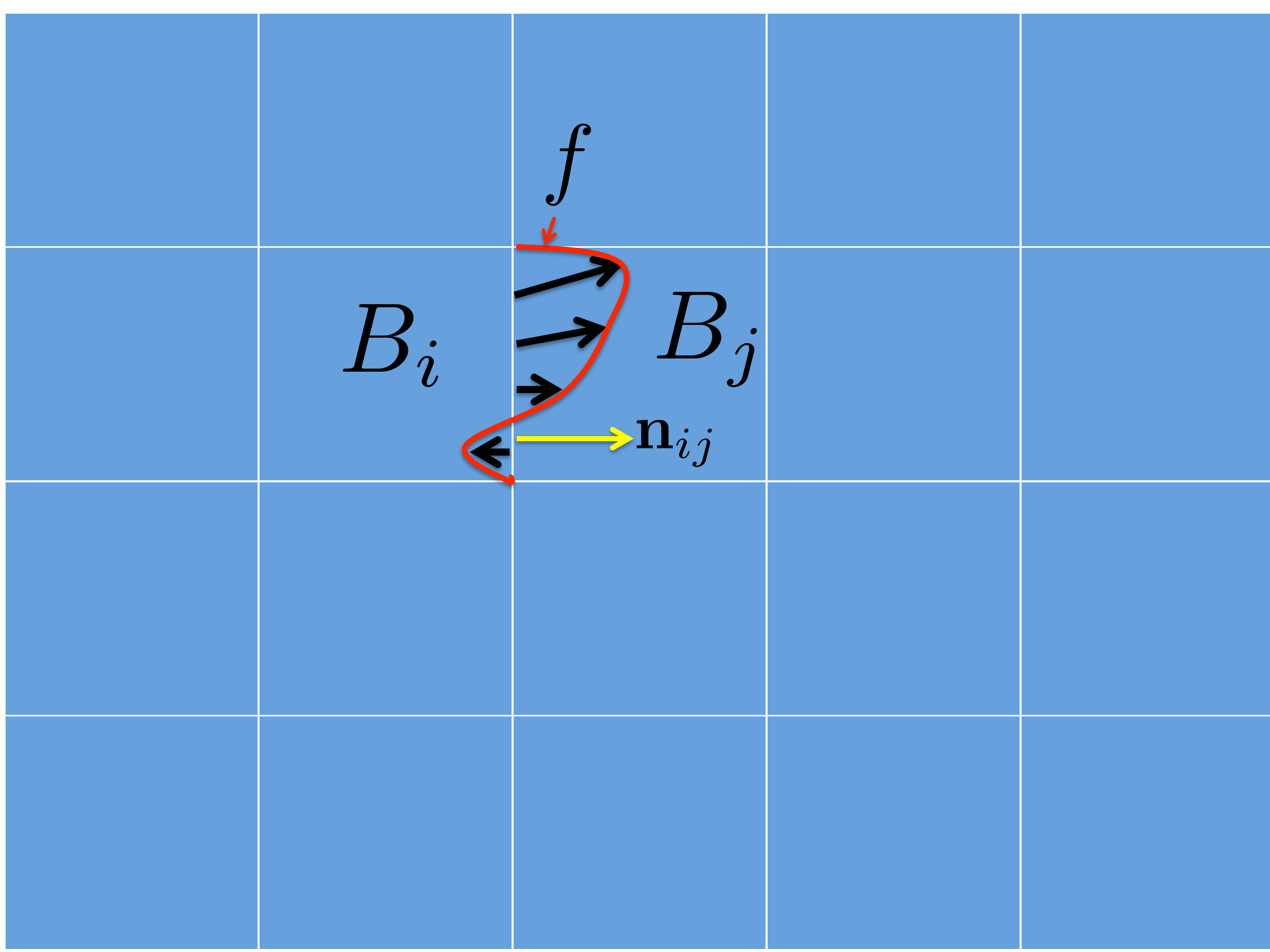}
\caption{\footnotesize{
Computation of infinitesimal generator $F$. The entry $F_{ij}$ is proportional to flux across $B_i\cap B_j$ from $B_i$ to $B_j$, due to vector field $f$.  
}}
\label{fig:Ulam}
\end{figure}
Ulam's method for approximating Perron-Frobenius operators using Markov matrices extends to numerical approximations of semigroups corresponding to the continuity equation. Analogously, one approximates the generator of the semigroup using transition rate matrices, which generate approximating semigroups on a finite state space. We recall this method as shown in \cite{froyland2013estimating}. \PG{We denote by $\bar{B_i}$ the closure of $B_i$}. The operator $\mathcal{A}(t)$ is approximated by defining elements of time-varying transition rate matrix  $\{A_{ij}(t)\}$, which are computed as follows,
\begin{align}\label{eq:gen1}
A_{ij}(t)=\begin{cases} \dfrac{1}{\bar{m}(B_j)}\int_{\bar{B}_i\cap \bar{B}_j} max\{f(x,t)\cdot \mathbf{n}_{ij},0\} & i \neq j, \\[7pt]
-\sum_{k\neq i} \dfrac{\bar{m}(\bar{B}_k)}{\bar{m}(\bar{B}_i)} A_{ik}(t) \:\:\:\mbox{otherwise, } \end{cases}
\end{align}
where $\mathbf{n}_{ij}$ is the unit normal vector pointing out of $B_i$ into $B_j$ if $\bar{B}_i\cap \bar{B}_j$ is a $(d-1)$ dimensional face, and zero vector otherwise. See Fig. \ref{fig:Ulam}. Note that in \cite{froyland2013estimating}, the authors also considered the perturbed version of the operator, $-\nabla \cdot (f(x,t)\cdot)$ :  $-\nabla \cdot (f(x,t) \cdot) + \frac{\epsilon^2}{2} \Delta $. This was mainly to exploit the spectral properties of the perturbed operator and the corresponding semigroup. However, in this work the perturbed operator does not offer any visible advantages. Hence, we work with approximations of the operator, $-\nabla \cdot (f(x,t) \cdot )$, alone. Nevertheless, we note that the discretization will introduce some numerical diffusion.

\subsection{Monge-Kantorovich Transport on Graphs}\label{subsec:MKG}
Now consider a directed graph \KE{$\mathcal{G}=(\mathcal{V},\mathcal{E})$} on $X$, where the set of vertices \KE{$\mathcal{V}$} represent the subsets $B_i$ as before, and the set of directed edges $\mathcal{E}$ are obtained from the topology of $X$. For each pair of neighboring vertices, two edges are constructed, one in each direction. 
%

A continuous-time advection on such a graph can be described \cite{berman2009optimized,chapman2011advection} as,
\begin{align}
\frac{d}{dt}\mu(t,v)=\sum_{e=(w\rightarrow v)}U(t,e)\mu(t,w)-\sum_{e=(v\rightarrow w)}U(t,e)\mu(t,v)\label{eq:adv_graph}, 
\end{align}

where $\mu(t,v)$ is the time-varying measure on a vertex $v$, and $U(t,e)$ is the flow on an edge $e$. Here we use the notation $e=(v\rightarrow w)$ to represent the edge $e$ directed from a vertex $v$ to $w$.
The notion of optimal transport has been extended to such a continuous-time discrete-space setting recently \cite{maas2011gradient,gigli2013gromov,mielke2013geodesic,solomon2016continuous}. Following \cite{solomon2016continuous}, one can formulate a quadratic cost optimal transport problem on \KE{$\mathcal{G}$} as follows. First, define an advective inner product between two flows $U_1,U_2$ as 
\begin{align}
\langle U_1,U_2\rangle_{\mu}=\sum_{e=(v\rightarrow w)}\left( \frac{\mu(v)}{\mu(w)}. \frac{\mu(v)+\mu(w)}{2}\right) U_1(e)U_2(e).
\end{align}
 
Then the corresponding optimal transport distance between a set of measures ($\mu_0,\mu_1$) supported on $\mathcal{V}$ can be written as 
\begin{align}
\tilde{W}_N(\mu_0,\mu_1)=\inf_{U(t,e)\geq 0, \mu(t,v)\geq 0} \int_0^1\|U(t,.)\|_{\mu(t,.)}dt,\\
\text{ such that Eq (\ref{eq:adv_graph}) holds, and }\nonumber \\
\mu(0,v)=\mu_0(v),\mu(1,v)=\mu_1(v) \:\:\:\:\forall v\in V.\nonumber 
\end{align}

Here $\|U(t,.)\|_{\mu(t,.)}\triangleq\sqrt{\langle U,U\rangle}_{\mu}$. This approach is motivated by the previously discussed Benamou-Brenier approach for optimal transport on continuous spaces, and results in the following advection based convex optimization problem. \begin{align}
\tilde{W}_N(\mu_0,\mu_1)^2=\inf_{J(t,e)\geq 0, \mu(t,v)\geq 0} \int_0^1 \sum_{e=(v\rightarrow w)}\frac{J(t,e)^2}{2} \left(\frac{1}{\mu(t,v)}+\frac{1}{\mu(t,w)}\right)dt,\label{eq:OT_ADV1}\\
\mu(0,v)=\mu_0(v),\mu(1,v)=\mu_1(v) \:\:\:\:\forall v\in V,\label{eq:OT_ADV2}\\
\frac{d}{dt}\mu(t,.)=D^TJ(t,.),\label{eq:OT_ADV3}
\end{align}
where $J(t,e)\triangleq \mu(t,v)U(t,e)$ for $e=(v\rightarrow w)$, and $D\in \mathbb{R}^{|\mathcal{E}|\times|\mathcal{V}|}$ is the linear flow operator computing $\mu(w)-\mu(v)$ for each $e=(v\rightarrow w)\in \mathcal{E}$. \PG{Specifically, $D^T(i,j)$ equals $+1$ if $j$th edge points into $i$th vertex, $-1$ if $j$th edge points out of $i$th vertex, and $0$ if $j$th edge is not connected to $i$th vertex. Hence Eq. (\ref{eq:OT_ADV3}) is a rewriting of Eq. (\ref{eq:adv_graph}) in terms of $J(t,.).$ The change of variables from $U$ to $J$ is analogous to the change of variables in Brenier-Benamou formulations, as discussed earlier in this section. }The convergence of $\tilde{W}_N$ to $W_2$, the $2$-Wasserstein distance on a continuous phase space (a  $d$-torus), as $N\rightarrow \infty$ is studied in Ref. \cite{gigli2013gromov}. 
\paragraph{\bf{Continuous-Discrete Analogy:}} Conceptually, one can regard the problem described by Eqs. (\ref{eq:OT_ADV1}-\ref{eq:OT_ADV3}) as the graph-based analogue of the optimal transport problem given in Eq. (\ref{eq:OT3}). Recall that this corresponds to \emph{single-integrator dynamics} $\dot{x}=u(t)$. In the next section, we use this interpretation, and generalize this graph-based framework to nonlinear dynamical systems of the form given in Eq. (\ref{eq:sys1}).

\section{Problem Setup and Computational Approach } \label{sec:Algo}
\subsection{Formulation of Optimal Transport Problem on Graphs}
Let $M \subset \mathbb{R}^d$ be an open bounded connected subset of an Euclidean space with piecewise smooth boundary. For a collection of analytic time-invariant vector fields $\lbrace g_i \rbrace_{i=1}^n $ and possibly time-varying vector field $g_0$ on $M$, consider the control affine system of the form
\begin{eqnarray}
\label{eq:ctrl_af}
\dot{x}(t) = g_0(x(t),t) +\sum_{i=1}^n u_i(t)g_i(x(t)), \nonumber \\
x(0) =x_0.
\end{eqnarray}
Then given the densities $\rho_0$ and $\rho_1$ on $M$, the corresponding optimal transport problem of interest is 
the following
\begin{gather}
\inf_{u(x,t),\rho(x,t)} \int_{\mathbb{R}^n}\int_{t_0}^{t_f} \sum_{i=1}^n \rho(x,t)|u_i(x,t)|^2dt dx ,\label{eq:OT_CF}\\
\text{ s.t. } \frac{\partial\rho(x,t)}{\partial t}+\nabla \cdot(\rho(x,t)g_0(x,t))+ \sum_{i=1}^n \nabla \cdot (\rho(x,t)u_i(x,t)g_i(x))=0 \hspace{2mm} ,x \in M, \label{eq:SG_CF} \\
\vec{n} \cdot (g_0(x,t)\rho(x,t) + \sum_i^n u_i(x,t)g_i(x)\rho(x,t) ) = 0 \hspace{2mm} a.e. \hspace{2mm} x \in \partial M, \nonumber \\
\rho(x,t_0)=\rho_0(x), \rho(x,t_f)=\rho_1(x).\nonumber
\end{gather}
\PG{Here, $\vec{n}$ is the outward normal vector at the boundary of $M$, and we have assumed zero mass flux boundary conditions.}

We approximate the optimal transport problem using a sequence of optimal transport problems on graphs. A key tool is to approximate the (time-varying) generator of the semigroup corresponding to the Eq. (\ref{eq:SG_CF}) using generator approximations on a finite state space \cite{froyland2013estimating}, as discussed in Section \ref{sec:TOaIG}. Hence, we approximate solutions of optimal transport problems on an Euclidean space using solutions of optimal transport problems on graphs. \\

\paragraph{\textbf{\PGn{Construction of Graph $\mathcal{G}$} :} }Towards this end, we partition M into $m$ d-dimensional connected, positive volume subsets $P_m = \lbrace  B_1,B_2 . . . , B_m \rbrace$. Additionally, we assume that the boundaries $\partial B_i$ are piecewise smooth. Then we can consider the optimal transport problem on a graph $\mathcal{G}=(\mathcal{V},\mathcal{E})$ where the the cardinality of $\mathcal{V}$ is $m$ and the connectivity of the graph is determined by the topology of $M$ and the partition $P_m$. More specifically, $\mathcal{V} = \lbrace 1,2.....m \rbrace$ and an element $e = (v \rightarrow w) \in \mathcal{E}$ for $v,w \in \mathcal{G}$ and $v \neq w$ if $\bar{B}_v \cap \bar{B}_w$ has non-zero $d-1$-dimensional measure. The graph $\mathcal{G}$ is {\it strongly connected}, i.e., for any two vertices, $v_0,v_T \in \mathcal{V}$ there exists a directed path of $r$ vertices, $(v_1,v_2....v_r)$ in $\mathcal{V}$, such that $(v_i \rightarrow v_{i+1}) \in \mathcal{E}$ for each $i \in \lbrace 1,2....r-1 \rbrace $. Moreover, this graph is also {\it symmetric}, that is, $e = (v \rightarrow w) \in \mathcal{E}$ implies $\bar{e}$ defined by $\bar{e} = (w \rightarrow v)$ is also in $\mathcal{E}$. 

In order to apply the approximation procedure from Ref. \cite{froyland2013estimating}, we express the continuity Eq. (\ref{eq:SG_CF}), as a bilinear control system,
\begin{equation}
\dot{y}(t) = \mathcal{A}_0(t)y + \sum_{i=1}^n \mathcal{A}_i(\hat{u}_i(t)y(t)),
\label{eq:BLsys}
\end{equation}

where $\mathcal{A}_0(t) = -\nabla \cdot (g_0(x,t) \hspace{1mm} \cdot \hspace{1mm})$ for each $t \in [0,1]$, $\hat{u}_i(t) = u_i(\cdot,t) $, $y(t) = \rho(\cdot,t)$,  $\mathcal{A}_i = -\nabla \cdot (g_i(x) \hspace{1mm} \cdot \hspace{1mm})$. Note that the right hand side of a bilinear system is traditionally expressed in the form  $A(t)\rho(t) + u(t)B\rho(t)$ in control theory literature \cite{elliott2009bilinear}. The form in Eq. (\ref{eq:BLsys}) is equivalent for systems on finite-dimensional state spaces, but not for general infinite dimensional bilinear systems if $\hat{u}(t)$ is not a scalar for each $t \in [t_0,t_f]$. For example, in the continuity equation, one can see that $u(x,t) \nabla \cdot (\rho(x,t)) \neq \nabla \cdot ( u(x,t)\rho(x,t))$ in general. Hence, the form Eq. (\ref{eq:BLsys}) is more appropriate for expressing the system in Eq. (\ref{eq:SG_CF}). 

In Section \ref{sec:TOaIG}, it was discussed how generators of  semigroups corresponding to the continuity equation can be used to define a approximating semigroup on a graph generated by appropriately constructed transition rate matrices. This method can be generalized to the controlled continuity equation, Eq. (\ref{eq:SG_CF}). A natural extension is to consider approximations of the control operators $\mathcal{A}_i$ using corresponding transition rate matrices, and analogously construct a controlled Markov chain on the space $\mathcal{V}$. However, we note that typically for a controlled Markov chain, the control parameters are constrained to be non-negative. Hence, a direct approximation of $\mathcal{A}_i$ using transition rate matrices and constraining $\hat{u}_i(t)$ to be positive would negate the possibility that one can flow both backward and forward along the control vector fields, which is critical for controllability of the system. Hence, to account for this in the approximation procedure, we define a bilinear control system equivalent to the one in Eq. (\ref{eq:BLsys}), but with positivity constraints on the control:
\begin{equation}
\dot{y}(t) = \mathcal{A}_0(t)y + \sum_{s \in \lbrace +,- \rbrace}\sum_{i=1}^n \mathcal{A}^s_i(\hat{u}^s_i(t)y(t)) \hspace{2mm}; \hspace{2mm} \hat{u}^s_i(t) \geq 0
\label{eq:BLsys2}
\end{equation}
where $\mathcal{A}^+_i = -\mathcal{A}^-_i = \mathcal{A}_i$ for each $i \in \lbrace 1,2....n \rbrace$.

\PG{Using the methodology introduced in Section \ref{sec:TOaIG}, for each of the operators $\mathcal{A}_0$, $\mathcal{A}^s_i$, we construct the control operators on the graph $\mathcal{G}$, which are denoted by $A_0 : [0,T] \times \mathcal{E} \rightarrow \mathbb{R}^+$ and  $A^s_i : \mathcal{E} \rightarrow \mathbb{R}^+$ \PG{(Recall that only $g_0$ is possibly time-varying, while $g_i$, $i>0$, are all time-invariant)}. The difference is that while generators in Section \ref{sec:TOaIG} were defined as \emph{vertex-based} $|\mathcal{V}|\times |\mathcal{V}|$ transition rate matrices, here we construct \emph{edge-based} vectors of size $\mathbb{|\mathcal{E}|}$ in a natural way.  Hence, $A_0$ is the edge-based version of the generator constructed from the vector field $g_0(\mathbf{x},t)$ using the formula in Eq. (\ref{eq:gen1}). For $A^s_i$, the corresponding transition rates are defined as}

\begin{equation}\label{eq:Aplus}
   A^+_i(e) = A^+_i(v \rightarrow w)=
                  \frac{1}{m(B_w)} \int_{\bar{B}_v \cap                               \bar{B}_w} max \lbrace g_i(x) \cdot \mathbf{n}_{v w},0)dm_{d-1}(x),
\end{equation}    
\begin{equation}\label{eq:Aminus}
   A^-_i(e) = A^-_i(v \rightarrow w)=
                  \frac{1}{m(B_w)} \int_{\bar{B}_v \cap                               \bar{B}_w} max \lbrace -g_i(x) \cdot \mathbf{n}_{v w},0)dm_{d-1}(x),
\end{equation} for \PG{$(i=1,\hdots,n)$}                
and where $\mathbf{n}_{vw}$ is the unit normal vector pointing out of $B_v$ into $B_w$ at $x$. \\

\paragraph{\textbf{\PGn{Construction of Control Graph $\mathcal{G}_c$, and Drift Graph $\mathcal{G}_0$ } :}} 
Let $\mathcal{P}(\mathcal{V})$ be the space of probability densities on the finite state space, $\mathcal{V}$. Then using the above parameter definitions, we consider the following flows on the graph $\mathcal{G}$,

\begin{align}
\frac{d}{dt}\mu(t,v)= & \sum_{e=(w\rightarrow v)}A_0(t,e)\mu(t,w)-\sum_{e=(v\rightarrow w)}A_0(t,e)\mu(t,v) \nonumber \\
& + \sum_{s \in \lbrace +,- \rbrace} \sum_{i=1}^n \sum_{e=(w\rightarrow v)}A^s_i(e)U^s_i(t,e)\mu(t,w)-\sum_{s \in \lbrace +,- \rbrace} \sum_{i=1}^n \sum_{e=(v\rightarrow w)}A^s_i(e)U^s_i(t,e)\mu(t,v)\label{eq:adv_graph_CF}, 
\end{align}
where $\mu(t,\cdot) \in \mathcal{P}(\mathcal{V})$ for each $t \in [0,T]$, and $U^s_i(t,\cdot)$ are the edge-dependent non-negative `control' parameters that scale the transition rates, $A^s_i(e)$. \KE{We associate a set of edges $\mathcal{E}_i^s$ with the above controlled flow. For each $s \in \lbrace +,- \rbrace$ and $i \in \lbrace 1,2,...n \rbrace$ we set $e \in \mathcal{E}_i^s$ if $A^s_i(e) \neq 0$. Similarly, we define $\mathcal{E}_0$ by setting $e \in \mathcal{E}_0$ if $A_0(t,e) \neq 0$ for some $t \in [0,1]$. 
\PG{Using these definitions we define the {\it control graph} $\mathcal{G}_c = (\mathcal{V},\mathcal{E}_c)$ by setting $\mathcal{E}_c = \cup_{s \in \lbrace +,-\rbrace}\cup_{i=1}^n\mathcal{E}_i^s$, and the {\it drift graph} $\mathcal{G}_0=(\mathcal{V},\mathcal{E}_0)$.} These definitions will be used in Section \ref{sec:GraphFlCtr}.}

The above defined flows can be shown to correspond to the evolution of a time-inhomogeneous continuous-time Markov chain on the finite state space, $\mathcal{V}$. The evolution of the corresponding stochastic process $X(t) \in \mathcal{V}$ over an edge, $e = (w \rightarrow v) \in \mathcal{E}$, is defined by the conditional probabilities:  
\begin{equation}
\mathbb{P}(X(t+h) = v | X(t) = w) = A_0(t,e) + \sum_{s \in \lbrace +,- \rbrace} \sum_{i=1}^n \sum_{e=(w\rightarrow v)}A^s_i(e)U^s_i(t,e) + o(h).
\end{equation}
%
%

\vspace{2mm}
This leads us to the approximating optimal transport problem on a graph, motivated by the formulation in Section \ref{subsec:MKG}:
\begin{align}
\label{eq:Caf_OT_Gr}
\tilde{W}(\mu_0,\mu_1)=\inf_{U^s_i(t,e)\geq 0, \mu(t,v)\geq 0} \sum_{s \in \lbrace +,- \rbrace} \sum_{i=1}^n \int_0^1\|U^s_i(t,.)\|_{\mu(t,.)}dt\\
\text{ such that Eq. (\ref{eq:adv_graph_CF}) holds, and }\nonumber \\
\mu(0,v)=\mu_0(v),\mu(1,v)=\mu_1(v) \:\:\:\:\forall v\in \mathcal{V}\nonumber 
\end{align}

\vspace{2mm}

Again, the formulation in Section \ref{subsec:MKG} motivates the following convex formulation of the above problem 
\begin{gather}
\label{eq:convex_OT_G}
\tilde{W}(\mu_0,\mu_1)^2=\inf_{J^s_i(t,e)\geq 0, \mu(t,v)\geq 0} \sum_{s \in \lbrace +,- \rbrace} \sum_{i=1}^n \int_0^1 \sum_{e=(v\rightarrow w)}\frac{J^s_i(t,e)^2}{2} \left(\frac{1}{\mu(t,v)}+\frac{1}{\mu(t,w)}\right)dt,\\
\mu(0,v)=\mu_0(v),\mu(1,v)=\mu_1(v) \:\:\:\:\forall v\in \mathcal{V},\nonumber \\
\frac{d}{dt}\mu(t,.)= \sum_{e=(w\rightarrow v)}A_0(t,e)\mu(t,w)-\sum_{e=(v\rightarrow w)}A_0(t,e)\mu(t,v) + \sum_{s \in \lbrace +,- \rbrace} \sum_{i=1}^n (D^s_i)^\intercal J^s_i(t,.),
\end{gather}

where $J^s_i(t,e)\triangleq \mu(t,v)U^s_i(t,e)$ for $e=(v\rightarrow w)$, $i= \lbrace 1,2...n \rbrace$, and $D^s_i\in \mathbb{R}^{|\mathcal{E}^s_i|\times|\mathcal{V}|}$ is the linear flow operator computing $\mu(w)-\mu(v)$ for each $e=(v\rightarrow w)\in \mathcal{E}^s_i$. 
\remark{}
We note that the controlled advection equation Eq. (\ref{eq:adv_graph_CF}), and the corresponding convex optimal transport problem in Eq. (\ref{eq:convex_OT_G}) can be simplified if control vector fields are uni-directional across all boundaries $\partial B_i$. This can often be achieved by choosing the grid carefully, and making the subvolumes $B_i$ small enough. If this condition holds, then we immediately see from Eqs. (\ref{eq:Aplus}-\ref{eq:Aminus}) that for each edge $e=(v\rightarrow w)$, only one of $A_i^+(e)$ and $A_i^-(e)$ is non-zero. Denote the non-zero matrix by $A_i(e)$. It also follows that $A_i(e)=A_i(\bar{e})$, where $\bar{e}=(w\rightarrow v)$. Then the simplified version of Eq. (\ref{eq:adv_graph_CF}) is
\begin{align}
\frac{d}{dt}\mu(t,v)= & \sum_{e=(w\rightarrow v)}A_0(t,e)\mu(t,w)-\sum_{e=(v\rightarrow w)}A_0(t,e)\mu(t,v) \nonumber \\
& + \sum_{i=1}^n \sum_{e=(w\rightarrow v)}A_i(e)U_i(t,e)\mu(t,w)- \sum_{i=1}^n \sum_{e=(v\rightarrow w)}A_i(e)U_i(t,e)\mu(t,v)\label{eq:adv_graph_CF1}.
\end{align}

This results in the following convex optimal transport problem,

\begin{align}\label{eq:Caf_OT_Gr1}
\tilde{W}(\mu_0,\mu_1)^2=\inf_{J_i(t,e)\geq 0, \mu(t,v)\geq 0} \sum_{i=1}^n \int_0^1 \sum_{e=(v\rightarrow w)}\frac{J_i(t,e)^2}{2} \left(\frac{1}{\mu(t,v)}+\frac{1}{\mu(t,w)}\right)dt,\\
\mu(0,v)=\mu_0(v),\mu(1,v)=\mu_1(v), \:\:\:\:\forall v\in \mathcal{V}\nonumber \\
\frac{d}{dt}\mu(t,.)= \sum_{e=(w\rightarrow v)}A_0(t,e)\mu(t,w)-\sum_{e=(v\rightarrow w)}A_0(t,e)\mu(t,v) +  \sum_{i=1}^n (D_i)^\intercal J_i(t,.).
\end{align}

\remark{}
 We note that the Eq. (\ref{eq:adv_graph}) discussed in Section \ref{subsec:MKG} can be seen as the special case of Eq. (\ref{eq:adv_graph_CF1}) with $g_0\equiv 0$ and $g_i=\hat{i}$ (the $i$th unit vector). Hence, our formulation generalizes optimal transport on graphs from a single-integrator system to general nonlinear control-affine systems. 
 
 While a rigorous proof of convergence of $\tilde{W}$ as defined in Eq. (\ref{eq:Caf_OT_Gr}) or Eq. (\ref{eq:Caf_OT_Gr1}) to $W_2$ is not provided here, the connection to Eq. (\ref{eq:adv_graph}) provides a heuristic argument in this direction. As discussed in Section \ref{subsec:MKG}, Ref. \cite{gigli2013gromov} provides such a convergence proof for an advection equation on graphs. The advection is modeled using anti-symmetric discrete `momentum vector fields' $V$ on the edges, and the optimal transport problem minimizes a discrete action. For the driftless case, Eq. (\ref{eq:adv_graph_CF1}) satisifies those conditions due to the way the transition matrices $A_i(e)$ (which give edge-weights) are defined, and our definition of $\tilde{W}$ agrees with the one given in Ref. \cite{gigli2013gromov}. We also note in general when $A_0 \neq 0$, the solution of the optimization problem $\tilde{W}$ does not necessarily define a metric on $\mathcal{P}(\mathcal{V})$ due to the asymmetry that is possibly induced by the drift vector fields. 

\subsection{Controllability Analysis of Flow over Graphs}\label{sec:GraphFlCtr}\PGn{
In this section, we establish that the controlled Markov chain approximations Eq. (\ref{eq:adv_graph_CF}) preserve the controllability properties of the system in Eq. (\ref{eq:ctrl_af}). In other words, we will show that if the underlying dynamical system Eq. (\ref{eq:ctrl_af}) satisfies some controllability conditions, then the dynamical system Eq. (\ref{eq:adv_graph_CF}) governing the evolution of measure on the graph $\mathcal{G}$ is also controllable in some precise sense. This will ensure the well-posedness of the graph optimal transport problem, Eq. (\ref{eq:Caf_OT_Gr}), since optimal transport is meaningful only if the set of possible transports between a pair of measures is non-empty. 

Our plan is as follows. First, in Theorem \ref{thm:strngcon}, we prove that controllability of Eq. (\ref{eq:ctrl_af}) results in the control graph $\mathcal{G}_c$ being strongly connected, and equal to $\mathcal{G}$. In the subsequent theorems, we show that the strongly connected property of $\mathcal{G}_c=\mathcal{G}_c$ implies that the system defined by Eq. (\ref{eq:adv_graph_CF}) is controllable between any pair of measures in the interior of $\mathcal{P}(\mathcal{V})$. This is first shown for the case of driftless systems (i.e., $g_0\equiv 0$) in Theorem \ref{thm:g0}, and then for systems with drift  (i.e., $g_0\not\equiv 0$) in Theorem \ref{drfl_ctrb}. Here, the interior of $\mathcal{P}(\mathcal{V})$ is defined as the set $int (\mathcal{P}(\mathcal{V}))= \lbrace \mu \in \mathcal{P}(\mathcal{V}); \mu(v)>0 ~~ \text{for each} ~~  v \in \mathcal{V}  \rbrace$.\\
}
Without loss of generality, we consider the case when $t_0=0 \:\:\& \:\: t_f=1$. First, we recall a few standard notions from geometric control theory \cite{bloch2003nonholonomic}.

\begin{definition}
Given $x_0  \in M$ we define $R(x_0,t)$ to be the set of all $x \in M$ for which there exists an admissible control $\mathbf{u} = (u_1, u_2....u_n)$ such that there exists a trajectory of  system in Eq. (\ref{eq:ctrl_af}) with $x(0) = x_0$, $x(t) = x$. The \textbf{reachable set from $x_0$ at time $T$} is defined to be
\begin{equation}
R_T(x_0) = \cup_{0 \leq t \leq T} R(x_0,t)
\end{equation}
\end{definition}

\begin{definition}
We say the system in Eq. (\ref{eq:ctrl_af}) is \textbf{small-time locally
controllable} from $x_0$ if $x_0$ is an interior point of $R_T (x_0)$ for any $T > 0$.
\end{definition}

\begin{definition}
Let $f = (f_1, ... f_d)$ and $g = (g_1, ... g_d)$ be two smooth vector fields on $M$. Then the \textbf{Lie bracket} $[f,g]$ is defined to be the vector field with components
\begin{equation}
[f,g]^i= \sum_{j=1}^d \bigg ( f^j \frac{\partial g^i}{\partial x^j} - g^j \frac{\partial f^i}{\partial x^j} \bigg )
\end{equation} 
\end{definition}

\begin{definition}
For a collection of vector fields $\lbrace g_i \rbrace$,  $\mathbf{Lie \lbrace g_i \rbrace}$  refers to the smallest Lie sub-algebra of set of smooth vector fields  on $M$ that contains $\lbrace g_i \rbrace$. $\mathbf{Lie_x \lbrace g_i \rbrace}$ refers to the span of all vector fields in $Lie \lbrace g_i \rbrace$ at $x \in M$
\end{definition}

\KE{We will use the notation $int(S)$ to refer to the interior of a set $S$.} Using these definitions we have the following result
\begin{theorem}\label{thm:strngcon}
Suppose one of the following statements is true:
\begin{enumerate}
  \item $g_0 \equiv 0$ and $Lie_x \bigg \lbrace g_i: i \in \lbrace 1,2....n \rbrace \bigg \rbrace = T_xM$ at each $x \in int(M)$.
  \item $span \bigg \lbrace g_i(x): i \in \lbrace 1,2....n \rbrace \bigg \rbrace = T_xM$ at each $x \in int(M)$.
\end{enumerate}
Then the graph $\mathcal{G}_c$ associated with the system in Eq. (\ref{eq:adv_graph_CF}) is strongly connected and $\mathcal{G}_c = \mathcal{G}$. 
\end{theorem}
\PGn{The proof of Theorem \ref{thm:strngcon} is provided in the appendix.}

\vspace{2mm}

\begin{remark}
The above result can also be seen to, equivalently, follow from the Orbit theorem \cite{agrachev2013control}[Theorem 5.1]. $\mathcal{G}_c \neq \mathcal{G}$ would imply the existence of a lower dimensional {\it immersed-submanifold}, $K$,  of $M$ such that $Lie_x \bigg \lbrace g_i: i \in \lbrace 1,2....n \rbrace \bigg \rbrace \subseteq T_xK$ for all $x$ in a neighborhood of a point in the boundary of one of the elements in $P_m$.
\end{remark}

\vspace{2mm}

\begin{remark}
The main obstruction in extending the above result for underactuated systems (span $\big \lbrace g_i(x): i \in \lbrace 1,2....n \rbrace \big \rbrace \neq T_xM$ for some $x \in M$) with drift, i.e. $g_0\not\equiv 0$, is that usual tests for small-time local controllability of control systems with drift \cite{sussmann1987general} require the initial condition to be an equilibrium point. Hence, starting at a non-equilibrium initial condition one might need to make large excursions (in our case, possibly outside the domain $M$) in order to return to the initial condition. Take for example, the simplest control-affine system with drift, the double integrator: $\ddot{x} = u$. Hence, given an initial and target density, the optimal transport problem on a bounded domain might not admit a solution for a system with drift if $M$ is not taken to be large enough.\end{remark}

\vspace{2mm}

\PGn{In the following, we show that Eq. (\ref{eq:adv_graph_CF}) has certain a controllability property for the case when the underlying system is driftless (i.e., $g_0\equiv 0$).} The proof follows from a more general result proved in Ref. \cite{elamvazhuthi2017mean} where the controllability result was proved for the case when $A_i(t,e)$ is either equal to $0$ or $1$ for each $i \in \lbrace 1,2....n \rbrace$ and each $e \in \mathcal{G}_c$ and $\mathcal{G}_c$ is only required to be strongly connected. Here we give an alternative proof for the case when $\mathcal{G}_c$ is strongly connected and symmetric, to keep the paper self-contained.\\

\begin{theorem}\label{thm:g0}
Consider $\mu_0, \mu_1 \in \KE{int (\mathcal{P}(\mathcal{V}))}$, $\mathcal{G}_c=\mathcal{G}$ strongly connected, and $A_0(t,e) = 0$ for every $e \in \mathcal{E}$ and all $t \in [0,1]$. Then there exist piecewise continuous $U^s_i(t,\cdot) \geq 0$ such that the solution of Eq. (\ref{eq:adv_graph_CF}), $\mu(t, \cdot)$ satisfies $\mu(0,\cdot) = \mu_0$ and $\mu(1,\cdot) = \mu_1$.
\end{theorem}
\begin{proof}
\KE{We can represent the equation \eqref{eq:adv_graph_CF} as a bilinear control system of the form 

\begin{eqnarray}
\frac{d}{dt}\mu(t, \cdot) = \sum_{s \in \lbrace +,- \rbrace} \sum_{i=1}^n \sum_{e = (v \rightarrow w)}A^s_i(e)U^s_i(t,e)B^{si}_e \mu(t,\cdot)
\end{eqnarray}
where $B^{si}_e \in \mathbb{R}^{m \times m}$ is given by
\[
 (B^{si}_{e})_{pq} = 
  \begin{cases} 
   -1 & \text{if } p = q=  v\\
    1 & \text{if } p= w, \hspace{1mm} q = v\\
   0       & \text{otherwise}
  \end{cases}
\]
for each $s \in \lbrace +,-\rbrace$, $ i \in \lbrace 1,2,...n\rbrace$ and $e = (v \rightarrow w) \in \mathcal{E}$.
Here,  $(B^{si}_{e})_{pq}$ denotes the element in the $p^{th}$ row and $q^{th}$ column of the matrix $B^{si}_{e}$. Corresponding to the graph, $\mathcal{G}_c$, we define the adjacency matrix $A_c$ defined by
\[
 (A_c)_{pq} = 
  \begin{cases} 
   1 & \text{if } (p \rightarrow q) \in \mathcal{E}_c\\
    0 & \text{otherwise} \\
  \end{cases}
\]
Additionally, the degree matrix is a diagonal matrix $D_c$ where the diagonal elements $(D_c)_{jj}$ are equal to the total number of edges leaving the vertex $j \in \mathcal{V}$. Then the Laplacian matrix $L_c$ associated with the graph $\mathcal{G}_c$ is given by $L_c = A_c-D_c$. Since the graph $\mathcal{G}_c$ is symmetric, $L_c$ is a symmetric matrix. Alternatively, the Laplacian $L_c$ of the graph $\mathcal{G}_c$ can also be expressed as $L_c= \sum_{s \in \lbrace +,- \rbrace} \sum_{i=1}^n\sum_{e \in \mathcal{E}^{si}_{sub}}B^s_{ie} $ for some subsets $\mathcal{E}^{si}_{sub} $ of the set of edges $\mathcal{E}_c$ such that $A^s_i(e) \neq 0$ for each $s \in \lbrace +,-\rbrace$ and $i \in \lbrace 1, 2,...n\rbrace$ such that $ e \in  \mathcal{E}^{si}_{sub}$. Since, $\mathcal{G}_c$ is strongly connected and symmetric it follows that the rank of the matrix $L_c$ is $m-1$. Let $\mathbf{1}$ and $ \mathbf{0}$ be vectors of dimension $m$ with all elements equal to $1$ and  $0$ respectively. Then we know that that $ (B^{si}_e\mathbf{1})^T\mathbf{1} = 0$ for each $s \in \lbrace +,-\rbrace$ and $i \in \lbrace 1, 2,...n\rbrace$ such that $ e \in  \mathcal{E}^{si}_{sub}$. This implies that span of the set $\mathcal{B} = \cup_{s \in \lbrace +,- \rbrace} \cup_{i=1}^n\cup_{e \in \mathcal{E}^{si}_{sub}}\lbrace  B^s_{ie}  \mathbf{1}\rbrace $ is equal to the tangent space $ T_x(\mathcal{P}(\mathcal{V}))$ of $\mathcal{P}(\mathcal{V})$ at every $x \in int(\mathcal{P}(\mathcal{V}))$. Here, we are identifying the set $ T_x(\mathcal{P}(\mathcal{V}))$ with the set $\lbrace y \in \mathbb{R}^m; \sum_{j=1}^my_j = 0 \rbrace$. Let $\tilde{\mathcal{E}}_{sub}^{si} $  be subsets of $\mathcal{E}^{si}_{sub} $, such that the elements of $\tilde{\mathcal{B}} = \cup_{s \in \lbrace +,- \rbrace} \cup_{i=1}^n\cup_{e \in \tilde{\mathcal{E}}^{si}_{sub}}\lbrace  B^s_{ie}  \mathbf{1}\rbrace $ are linearly independent and span the set $ T_x(\mathcal{P}(\mathcal{V}))$ at every $x \in int (\mathcal{P}(\mathcal{V}))$. 

Next, we note that $int (\mathcal{P}(\mathcal{V}))$ is convex and hence there exists an atleast-once-differentiable path $\gamma  : [0,T] \rightarrow int (\mathcal{P}(\mathcal{V}))$ such that $\gamma(0)=\mu_0$ and $\gamma(1) =\mu_1$. Since span of $\tilde{\mathcal{B}}$ is equal $T_x(\mathcal{P}(\mathcal{V}))$ at every $x \in int(\mathcal{P}(\mathcal{V})$ and $\frac{d}{dt}\gamma(t) \in T_{\gamma(t)}(\mathcal{P}(\mathcal{V}))$ at every $t \in [0,1]$ it follows that there exist parameters $\tilde{U}^s_i(t,e)$ that are continuous with respect to time satisfying
\begin{equation}
\frac{d}{dt}\gamma(t) =  \sum_{s \in \lbrace +,- \rbrace} \sum_{i=1}^n \sum_{e \in \tilde{\mathcal{E}}^{si}_{sub} } \tilde{U}^s_i(t,e)B^{si}_e\mathbf{1}
\end{equation}  
for all $t \in [0,1]$.
Next, for each $s \in \lbrace +,- \rbrace$ and each $i \in \lbrace 1,2....n \rbrace$ we set    
\[
  U^s_i(t,e) = 
  \begin{cases} 
  \frac{\tilde{U}^s_i(t,e)}{A^s_i(e) \mu(t,v)} ~~ \text{if} ~~ U^s_i(t,e) \geq 0\\
   0       & \text{otherwise}
  \end{cases}
\]
and 
\[
  U^s_i(t,\bar{e}) = 
  \begin{cases} 
  -\frac{\tilde{U}^s_i(t,e)}{A^s_i(\bar{e}) \mu(t,w)} ~~ \text{if} ~~ U^s_i(t,e) \leq 0\\
   0       & \text{otherwise}
  \end{cases}
\]
whenever $e \in \tilde{\mathcal{E}}^{si}_{sub}$, for each $t \in [0,1]$.
Additionally, if for a given $s \in \lbrace +,- \rbrace$ and $i \in \lbrace 1,2....n \rbrace$ we have an edge $e \in \mathcal{E}$ such that $e \in \mathcal{E}_c \backslash \tilde{\mathcal{E}}^{si}_{sub}$, then we set 
\begin{equation}
U^s_i(t,e) = 0
\end{equation}
and 
\begin{equation}
U^s_i(t,\bar{e}) = 0
\end{equation}
for each $t \in [0,1]$.  Note that that $B^{si}_e \mu(t,\cdot) = \mu(t,v)B^{si}_e \mathbf{1}$ when $e = (v \rightarrow w)$. From this, the result follows by noting that for the choice of parameters $U^s_i(t,e) $ the solution of \eqref{eq:adv_graph_CF}, given by $\mu(t,\cdot ) = \gamma(t)$ for all $t \in [0,1]$.}
\end{proof}
The above proof can also be extended to the case when $\mathcal{G}_c$ is only strongly connected and not necessarily symmetric. \PG{Note that for the case when either of $(\mu_0, \mu_1)$ lie on the boundary of $\mathcal{P}(\mathcal{V})$, controllability of the system in Eq. (\ref{eq:adv_graph_CF}) might not hold, as pointed out in Ref. \cite{elamvazhuthi2017mean}. This does not affect our numerical results however, and we are able to achieve convergence in all cases we discuss in Section \ref{sec:Examples}.\\
}

\PGn{Theorem \ref{thm:g0} leads to the following result for the case of systems with drift (i.e., $g_0\not\equiv 0$).\\

\begin{theorem}
\label{drfl_ctrb}
Consider $\mu_0, \mu_1 \in \KE{ int(\mathcal{P}(\mathcal{V}))}$. Assume the graph $\mathcal{G}_c=\mathcal{G}$ is strongly connected, and $\mathcal{G}_0 \subseteq \mathcal{G}_c$. Then there exist $U_i(t,\cdot) \geq 0$ such that Eq. (\ref{eq:adv_graph_CF}) satisfies $\mu(0,\cdot) = \mu_0$ and $\mu(1,\cdot) = \mu_1$.
\end{theorem}}

\begin{proof}
The graph $\mathcal{G}_c$ is connected. Since $\mathcal{G}_0 \subseteq \mathcal{G}$, we can choose $\tilde{U}^s_i(t,\cdot)$ such that the right hand side in Eq. (\ref{eq:adv_graph_CF}) is equal to $0$ for all $t \in [0,1]$. Then, from the previous theorem, it follows that there exists a control $U^s_i(t,\cdot)$, of the form $ U^s_i(t,\cdot) = \hat{U}^s_i(t,\cdot)+\tilde{U}^s_i(t,\cdot)$ such that Eq. (\ref{eq:adv_graph_CF}) satisfies $\mu(0,\cdot) = \mu_0$ and $\mu(1,\cdot) = \mu_1$. Here, the parameters $\tilde{U}_i(t,\cdot)$ negate the effect of the drift field $A_0$, and $\hat{U}^s_i(t,\cdot)$ ensure the density $\mu_0$ is transported to $\mu_1$ as in Thm \ref{thm:g0}.
\end{proof}

\subsection{Construction of Approximate Feedback Control Laws}

Given the solution the optimal transport problem on the graph, we reconstruct the corresponding approximate feedback control laws $\lbrace u_i(x,t) \rbrace$ for the underlying dynamical system Eq. (\ref{eq:ctrl_af}). \PGn{Since the optimal transport problem is solved on the graph, the feedback control law is vertex-based. For any vertex $v$ of the graph $\mathcal{G}$, all agents with their state $x$ lying in the sub-volume $B_v$, apply the following feedback law:}

\begin{equation}\label{eq:feedback}
u_i(x,t) = \frac{\sum_{w \in \mathcal{N}^+_i(v)} U^+_i(v \rightarrow w,t)} {|\mathcal{N}^+_i(v)|} - \frac{\sum_{w \in \mathcal{N}^-_i(v)} U^-_i(v \rightarrow w,t)} {|\mathcal{N}^-_i(v)|} \hspace{2mm} \forall x \in B_v.
\end{equation}

Here, $\mathcal{N}^s_i(v)$ refers to the the neighboring vertices of $v$ in the graph $(\mathcal{V},\mathcal{E}^s_i)$ for each $s \in \lbrace +,- \rbrace$ and $i \in \lbrace 1,2....n\rbrace$. \subsection{Numerical Implementation}

\PG{We adapt the numerical scheme used in Refs. \cite{papadakis2014optimal,solomon2016continuous} to our setting, and use a staggered discretization scheme for pseudo-time discretization. }We define 
\begin{align}
\mu_{j}(v)\triangleq \mu(t_j,v),\\
J^{s}_{i,j}(e)\triangleq J_i^s(t_j,e),
\end{align} 
where $t_j=(j/k)t_f, j\in[0,1,2,\dots,k]$ is the time discretization into $k$ intervals. We take $t_0=0$. Here $J^{s}_{i,j}(e)$ represents the $s\in\{+,-\}$ flow due to $g_i(x)$ over edge $e=(v\rightarrow w)$, from vertex $v$ at time $t_{j}$ to vertex $w$ at time $t_{j+1}$. 

Hence, the optimization problem given in Eqs. (\ref{eq:convex_OT_G}) can discretized as,

\begin{align}
\tilde{W}(\mu_0,\mu_1)^2=\inf_{J^{s}_{i,j}\geq 0,\mu_{j}\geq 0}
\sum_{s\in\{+,-\}}\sum_{i=1}^n\sum_{j=1}^k\sum_{\substack{e=1\\e=(v\rightarrow w)}}^{|\mathcal{E}_i^s|}(J^{s}_{i,j}(e))^2 (\frac{1}{\mu_{j}(v)}+\frac{1}{\mu_{j+1}(w)})\label{eq:cost_disc1},
 \end{align}
subject to the following constraints:
\begin{align}
\frac{\mu_{j+1}-\mu_{j}}{\Delta t}= A_0(t_j)\mu_j+\sum_{s\in\{+,-\}}\sum_{i=1}^n (D_i^s)^\intercal J^{i,j}_s,\label{eq:OT_disc1}\\
\mu_{0}=\mu_{t_0}, \mu_{k}=\mu_{t_f}\label{eq:bc_constraint1},
\end{align}
where we have used the vertex-based $m\times m$ transition rate matrix $A_0(t_j)$ as originally defined in Eq. (\ref{eq:gen1}).
Here $\Delta t=\dfrac{t_f}{k}$. The cost function given by Eq. (\ref{eq:cost_disc1}) is again of the form `quadratic over linear', and the advection (Eq. (\ref{eq:OT_disc1})) imposes linear constraints. Hence the discretized problem is convex, and can be solved using many off-the-shelf convex solvers. The optimization problem is solved via CVX \cite{grant2008cvx} modeling platform, an open-source software for converting convex optimization problems into usable format for various solvers. We use the SCS \cite{o2013operator} solver, a first-order solver for large size convex optimization problems. This solver uses the Alternating Direction Method of Multipliers (ADMM) \cite{eckstein2012augmented} to enable quick solution of very large convex optimization problems, with moderate accuracy. 

The variables to be solved for in the optimization problem Eqs. (\ref{eq:cost_disc1}-\ref{eq:bc_constraint1}) are vertex based quantities $\mu_{j}$, and edge based quantities $J^s_{i,j}$. The size of the optimization problem can be quantified in terms of number of time-discretization steps $k$, number of vertices $|\mathcal{V}|=m$, and the number of edges $|\mathcal{E}_c|$. The graph $\mathcal{G}_c$ is always sparse, since a typical vertex is at most connected to $2(n+1)d$ neighbors, and $m\gg n,m\gg d$. Hence, the variables in the optimization problem scale as $O(k(m+|\mathcal{E}|))=O(n\cdot d\cdot k\cdot m)$. \PGn{The computations are performed on a workstation with an Intel Xeon X$5690$ \cite{intel} processor.

In the examples that follow, the graph size $m$ is chosen to be large enough so that the qualitative features of the optimal transport are well resolved, and do not change upon finer grid refinement. The time-discretization parameter $k$ is chosen such that the optimal transport cost $\tilde{W}$ is insensitive to finer discretization.}

\section{Examples}\label{sec:Examples}

\subsection{Optimal Transport in the Grushin Plane}\label{subsec:Gr}
We first apply our framework to a non-holonomic control-affine system in which certain optimal transport solutions can be found analytically. We consider transport of measure in the Grushin system. In Ref. \cite{agrachev2009optimal}, the structure of optimal controls in this problem was analyzed. Using this structure, optimal transport to a delta measure at $(0,0)$ was computed. The system is described by
\begin{subequations}
\begin{align}
\dot{x}_1=u_1,\\
\dot{x}_2=u_2x_1.
\end{align}
\end{subequations}

\PGn{This system is a driftless system with control vector fields $g_1(x_1,x_2)=[1 \; 0]^\intercal$,  $g_2(x_1,x_2)=[0 \;x_1]^\intercal$. These do not span the tangent space $\mathbb{R}^2$, but their Lie-algebra does, i.e. $Lie_x \bigg \lbrace g_i: i \in \lbrace 1,2 \rbrace \bigg \rbrace = \mathbb{R}^2$. This can be seen by noting that the Lie-bracket $[g_1,g_2]=[0 \;1]^{\intercal}$, and hence $span \{[g_1,g_2],g_1\}=\mathbb{R}^2$. Hence, this system satisfies condition $1$ of Theorem \ref{thm:strngcon}. By Theorem \ref{thm:g0}, the corresponding optimal transport problem for this driftless system is well-posed. }

The optimal control cost $c(x,y)$ between initial and final states, $x=(x_1,x_2)^\intercal,y=(y_1,y_2)^\intercal$, is taken to be square of the subriemannian distance $d(x,y)=\inf_{\mathbb{U}_x^y}\int_0^1\sqrt{u_1^2+u_2^2}dt$. Hence, the optimal control solutions are also geodesics in the subriemannian space. The solutions of the optimal control problem are integral curves of the Hamiltonian $H$ given by
\begin{align}
H(x_1,x_2,p_1,p_2)=\frac{1}{2}(p_1^2+x_1^2p_2^2). \label{eq:H_Gr}
\end{align}
Here $p_1,p_2$ are the co-state variables.
Note that since $H$ is independent of $x_2$, $H$ can be reduced to a Hamiltonian in $(x_1,p_1)$, and the integral curves of $H$ can be obtained using quadratures.  The geodesics reaching $(0,\alpha)$ at $t=1$ are of the form
\begin{align}
x_1(t)=\frac{a}{b}sin(b(1-t)),\label{eq:Geo_Gr1}\\
x_2(t)=\frac{a^2}{4b^2}(2b(1-t)-sin(2b(1-t)))+\alpha.\label{eq:Geo_Gr2}
\end{align}

A geodesic between a specified initial point $(\bar{x}_1,\bar{x}_2)$, and $(0,\alpha)$ can be obtained by inverting the Eqs. (\ref{eq:Geo_Gr1}-\ref{eq:Geo_Gr2}) at $t=0$ to solve for $(a,b)$. For $t\leq\frac{\pi}{b}$, these geodesics are also global minimizers of the optimal control problem. Figure \ref{fig:Geo_Gr}(a) shows some geodesics to the origin. 

Now consider the optimal transport problem with $c(x,y)=d^2$ from an initial measure $\mu_0$ to final measure $\mu_{1}=\delta_{(0,0)}$. Clearly, the optimal map $T$ is $x\rightarrow (0,0)$, and the corresponding flow is given by the geodesics between each $x\in \:supp(\mu_0)$ and $(0,0)$. See Fig. \ref{fig:Geo_Gr}(b) for analytically computed transport in the case in which the initial measure is uniform over a disk.
\begin{figure}[h]
\centering
\subfloat[]{\includegraphics[width=2.25in]{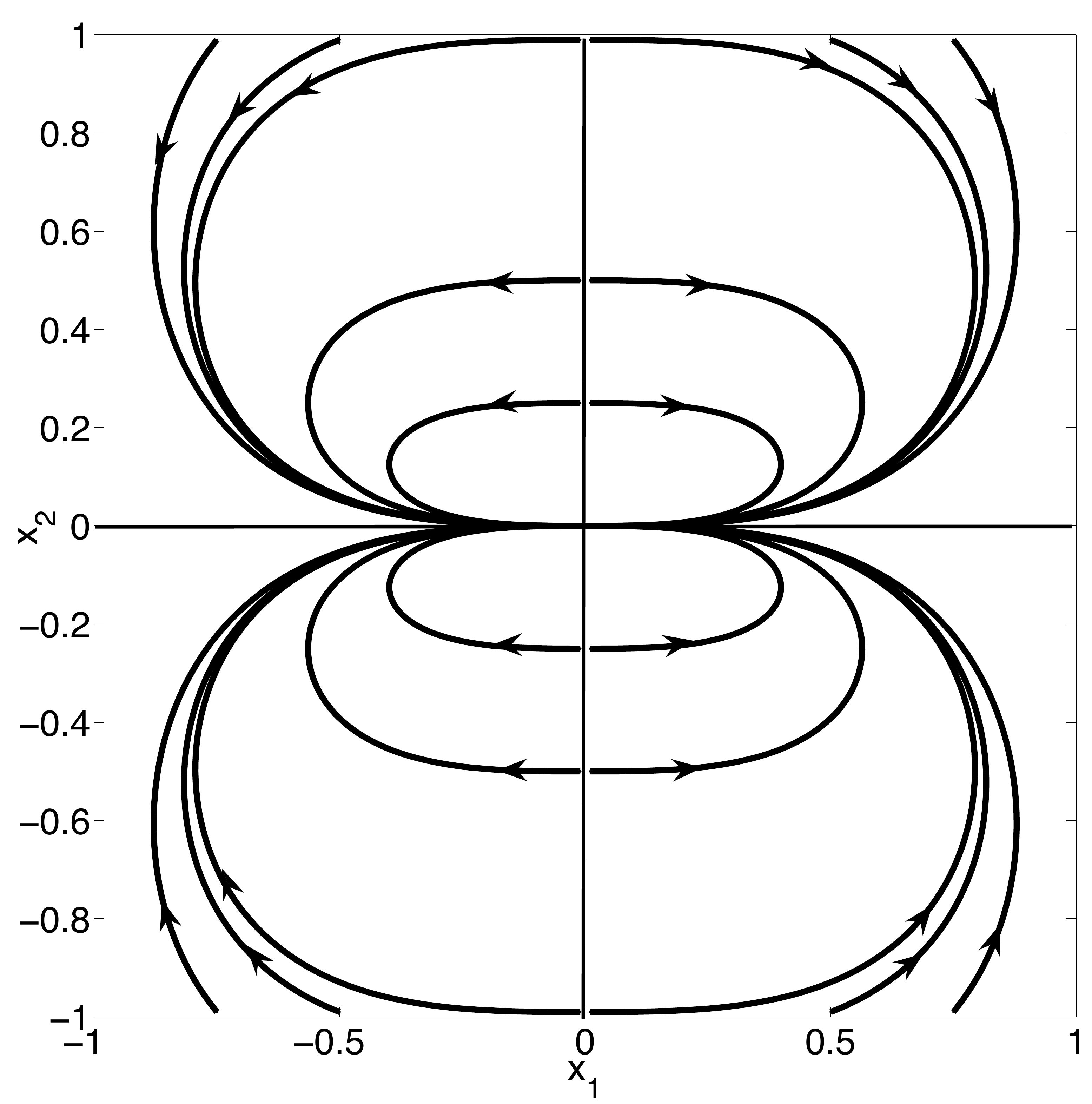}}
\subfloat[]{\includegraphics[width=5in]{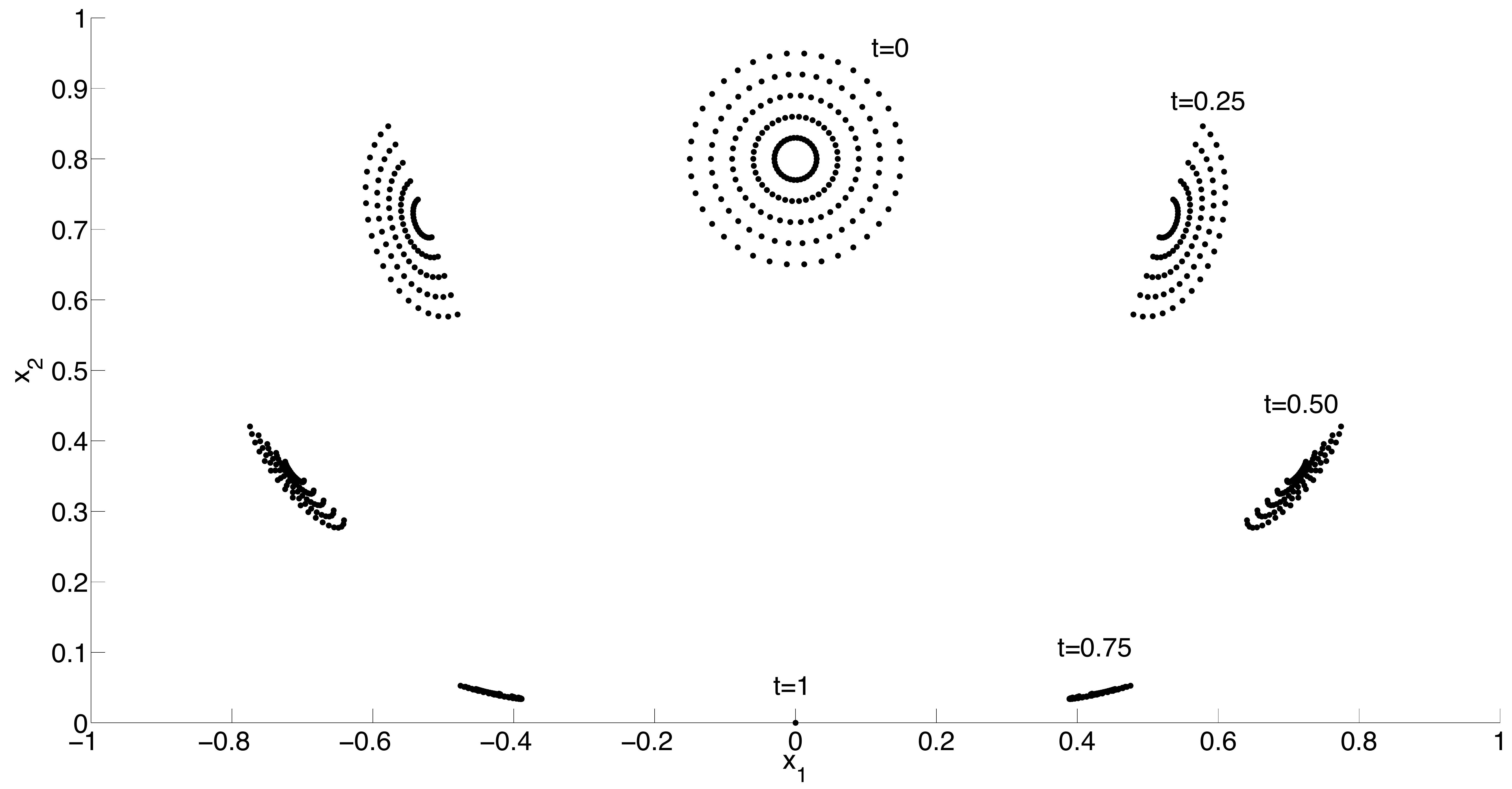}}
\caption{\footnotesize{(a) Some minimizing geodesics to the origin in the Grushin plane. (b) Analytically computed optimal transport solution between a uniform measure whose support is the disk $\Omega=\{(x,y)|x^2+(y-.8)^2<.15^2\}$, and a measure concentrated at the origin.}}
\label{fig:Geo_Gr}
\end{figure}
Using the algorithm developed in Section \ref{sec:Algo}, we compute optimal transport for this same case. We divide the $X=[-1,1]\times[-1, 1]$ into $m=100^2$ boxes, and form the corresponding graph $\mathcal{G}$.  The resulting solution is shown in Figure \ref{fig:OT_Gr}(a)-(e). The convergence of optimal transport cost $\tilde{W}$ with $m$, and  $k$ (i.e., the time-discretization) is shown in Fig \ref{fig:OT_Gr}(f). \PGn{The $k=75$ case takes about $2\times 10^4$ seconds of computation time.} It can be seen that the computed solution closely follows the analytical solution shown in Fig. \ref{fig:Geo_Gr}.

\begin{figure}[h!]
\centering
\subfloat[t=0]{\includegraphics[width=3in]{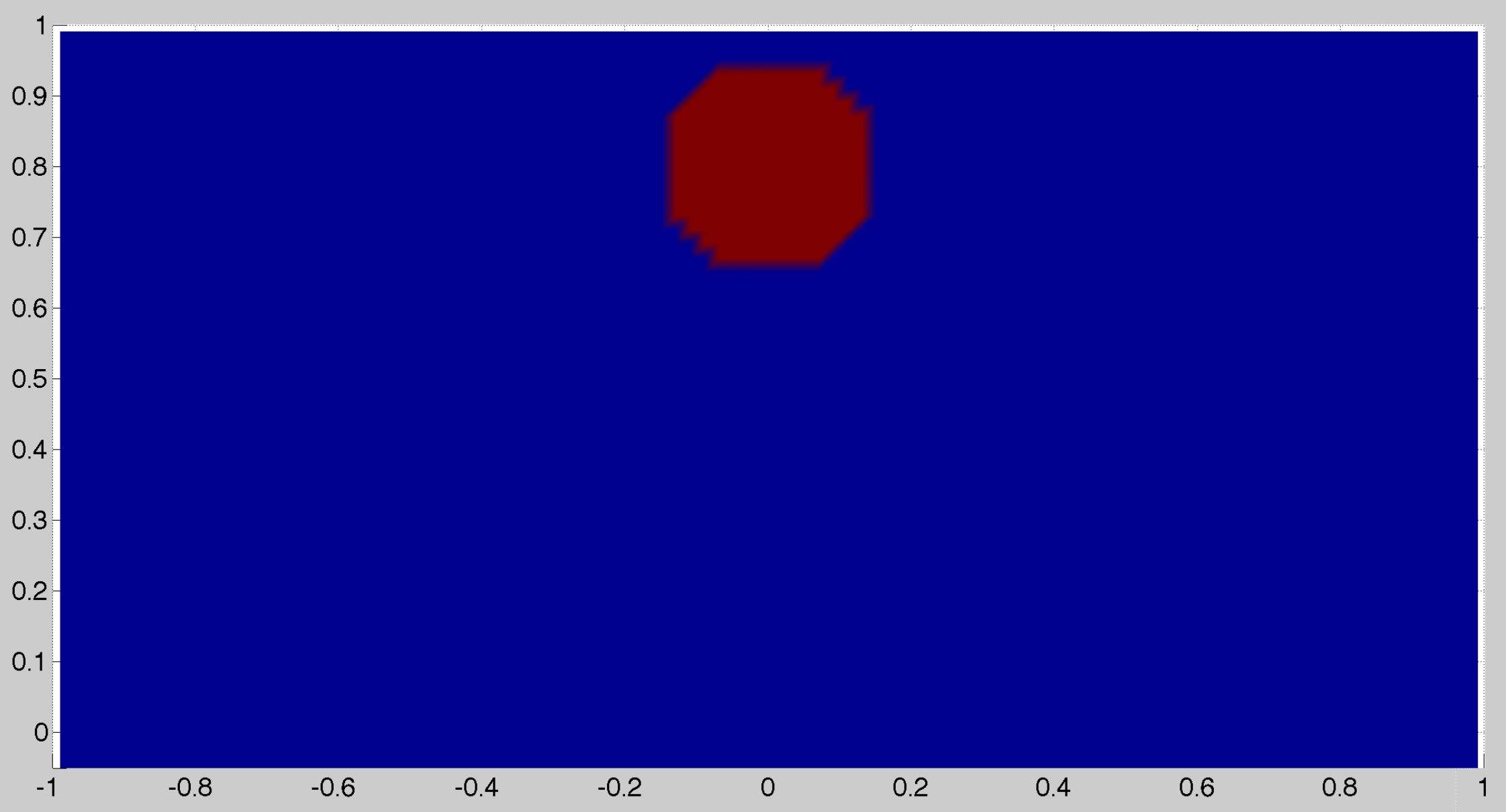}}\qquad
\subfloat[t=0.25]{\includegraphics[width=3in]{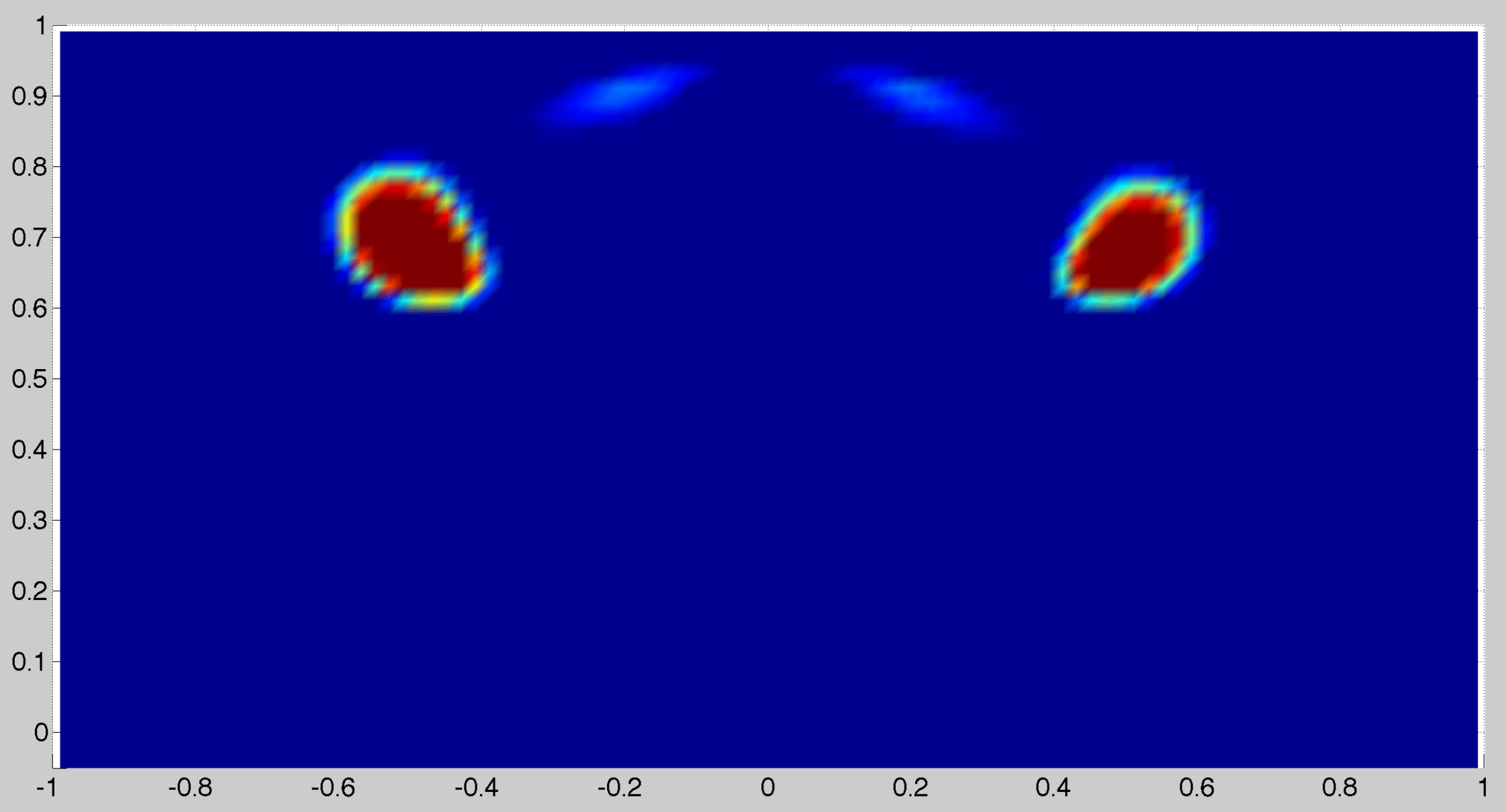}}\qquad
\subfloat[t=0.5]{\includegraphics[width=3in]{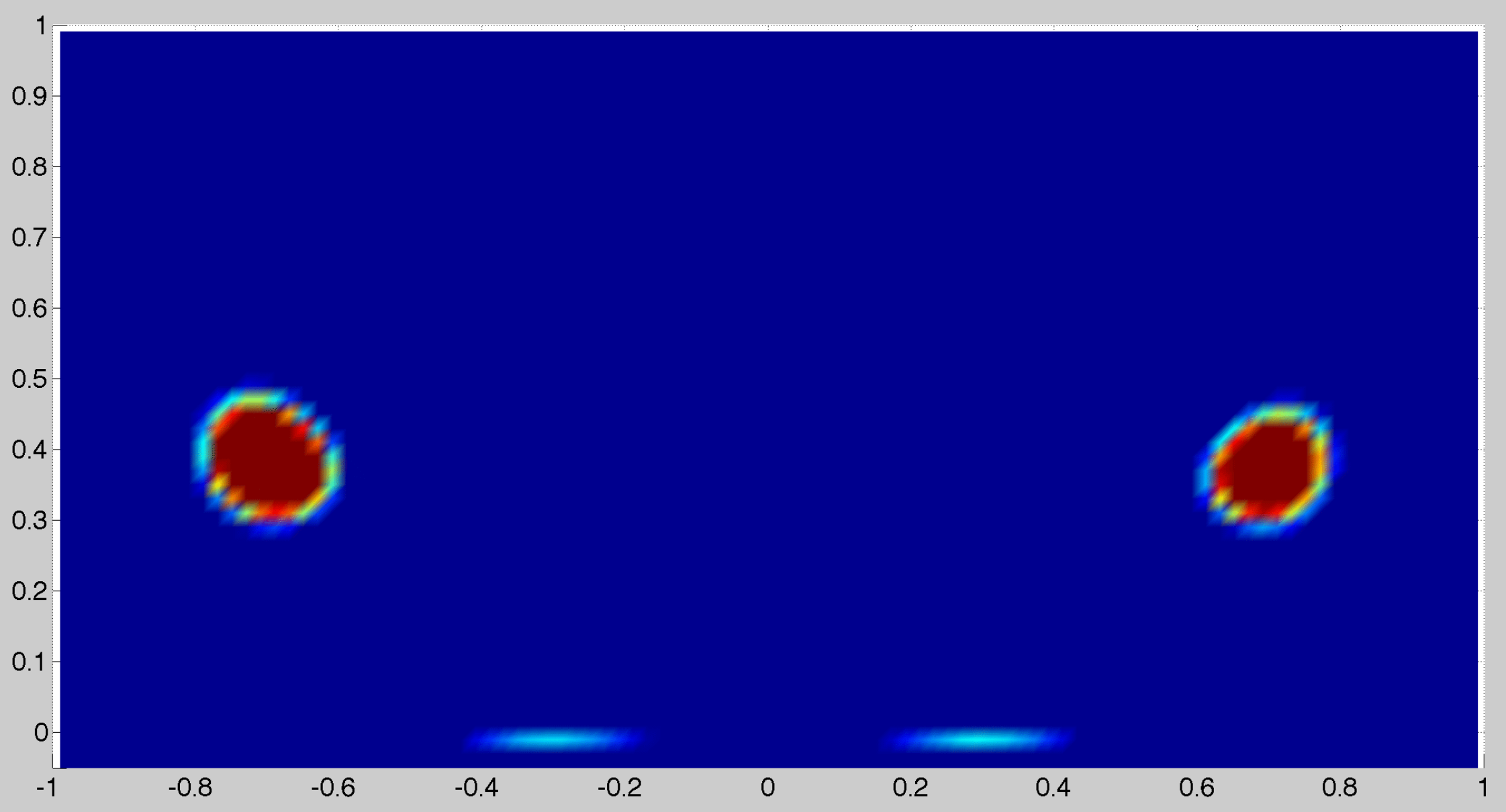}}\qquad
\subfloat[t=0.75]{\includegraphics[width=3in]{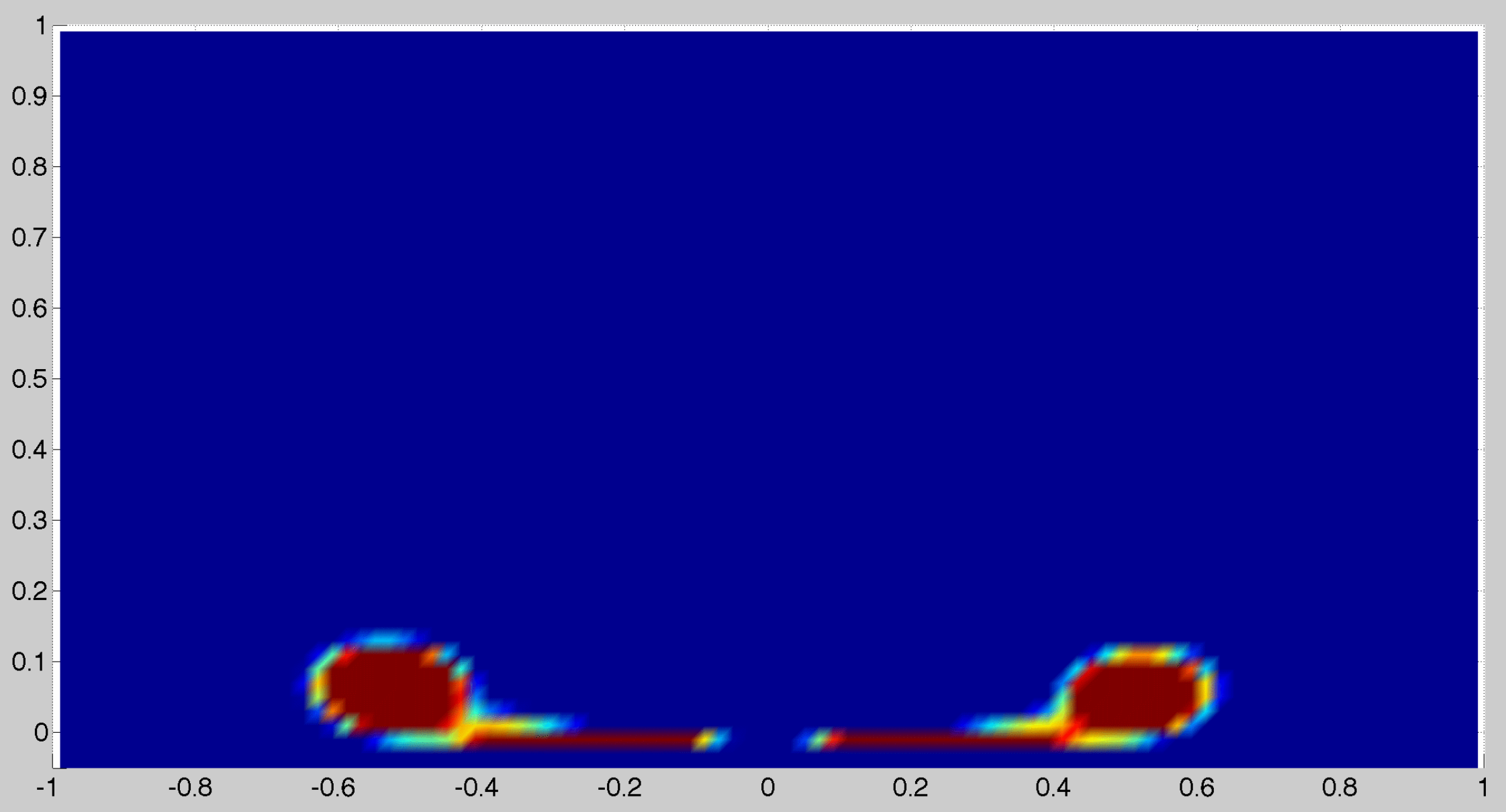}}\qquad
\subfloat[t=0.95]{\includegraphics[width=3in]{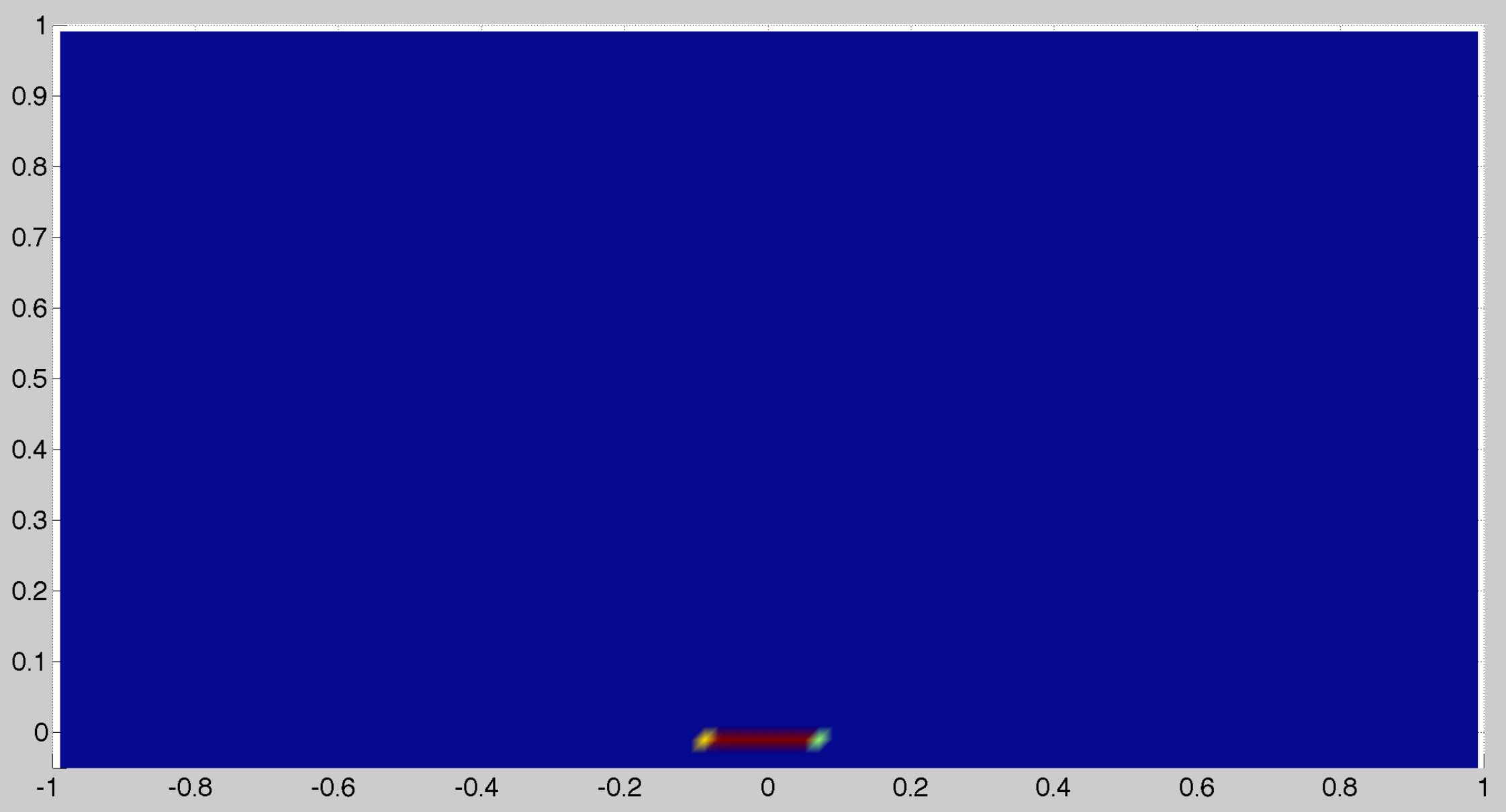}}\qquad
\subfloat[]{\includegraphics[width=3in]{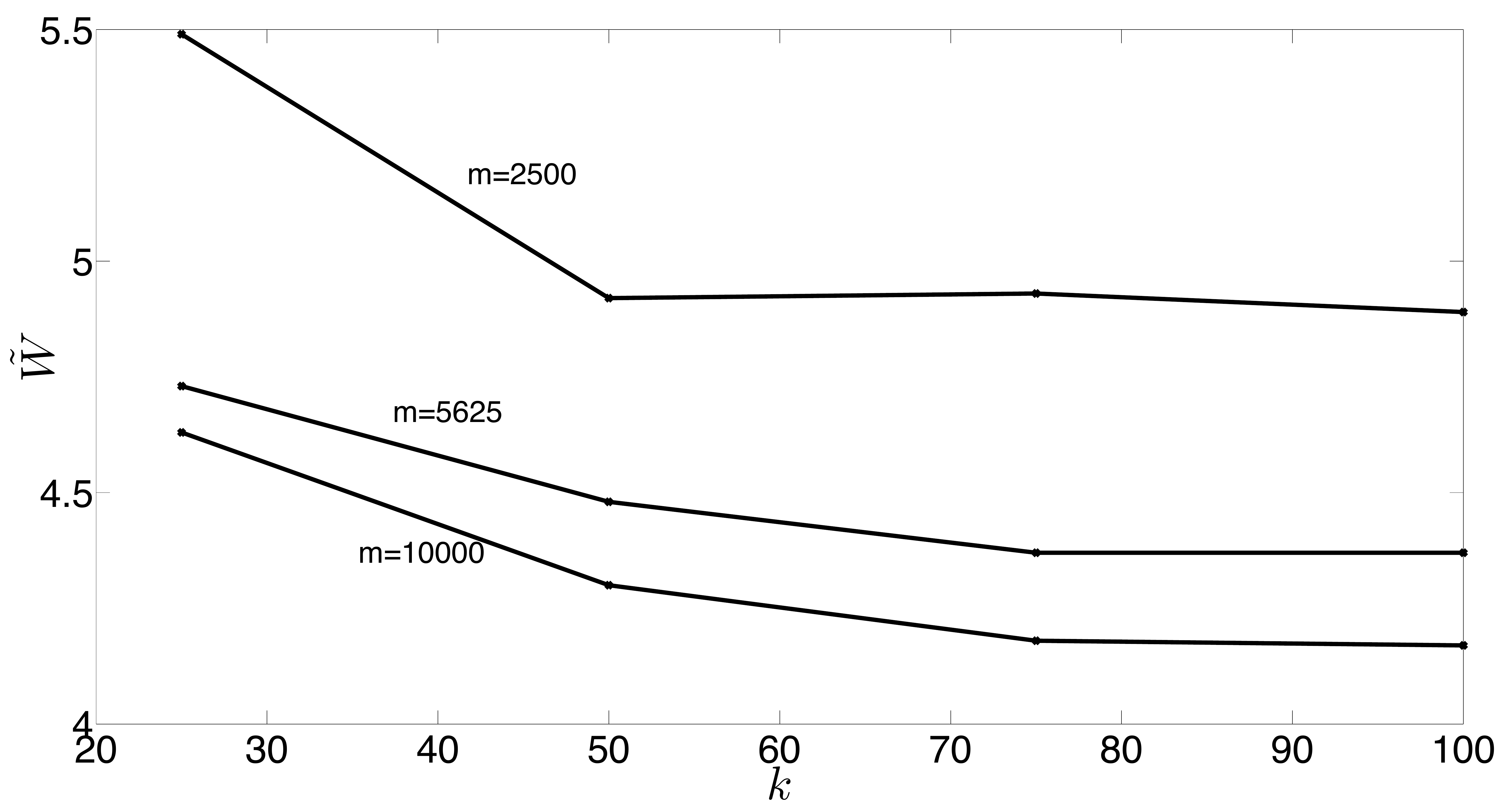}}\qquad
\caption{\footnotesize{ (a)-(e) The optimal transport solution in the Grushin plane using graph based algorithm between a measure whose support is the disk $\Omega=\{(x,y)|x^2+(y-.8)^2<.15^2\}$, and delta measure at the origin. The parameters are $m=10^4,k=75$. (f) Convergence of optimal transport cost with number of time discretization steps $k$ and grid size $m$.}}
\label{fig:OT_Gr}
\end{figure}

\PGn{Next, we perform particle simulations with feedback controls extracted from the optimal transport computation, using Eq. (\ref{eq:feedback}). We take $p=4$ particles per box, and use Eq. (\ref{eq:feedback}) to get state-dependent control commands for each particle.} The results are shown in Figure \ref{fig:particle_Gr}. About $95\%$ of the particles get transported according the optimal transport solution shown in Fig \ref{fig:OT_Gr}, while the rest are dispersed. Note that the control laws obtained from optimal transport solution do not automatically guarantee feedback stabilization of individual particles. 

\begin{figure}[h!]
\centering
\subfloat[t=0]{\includegraphics[width=3in]{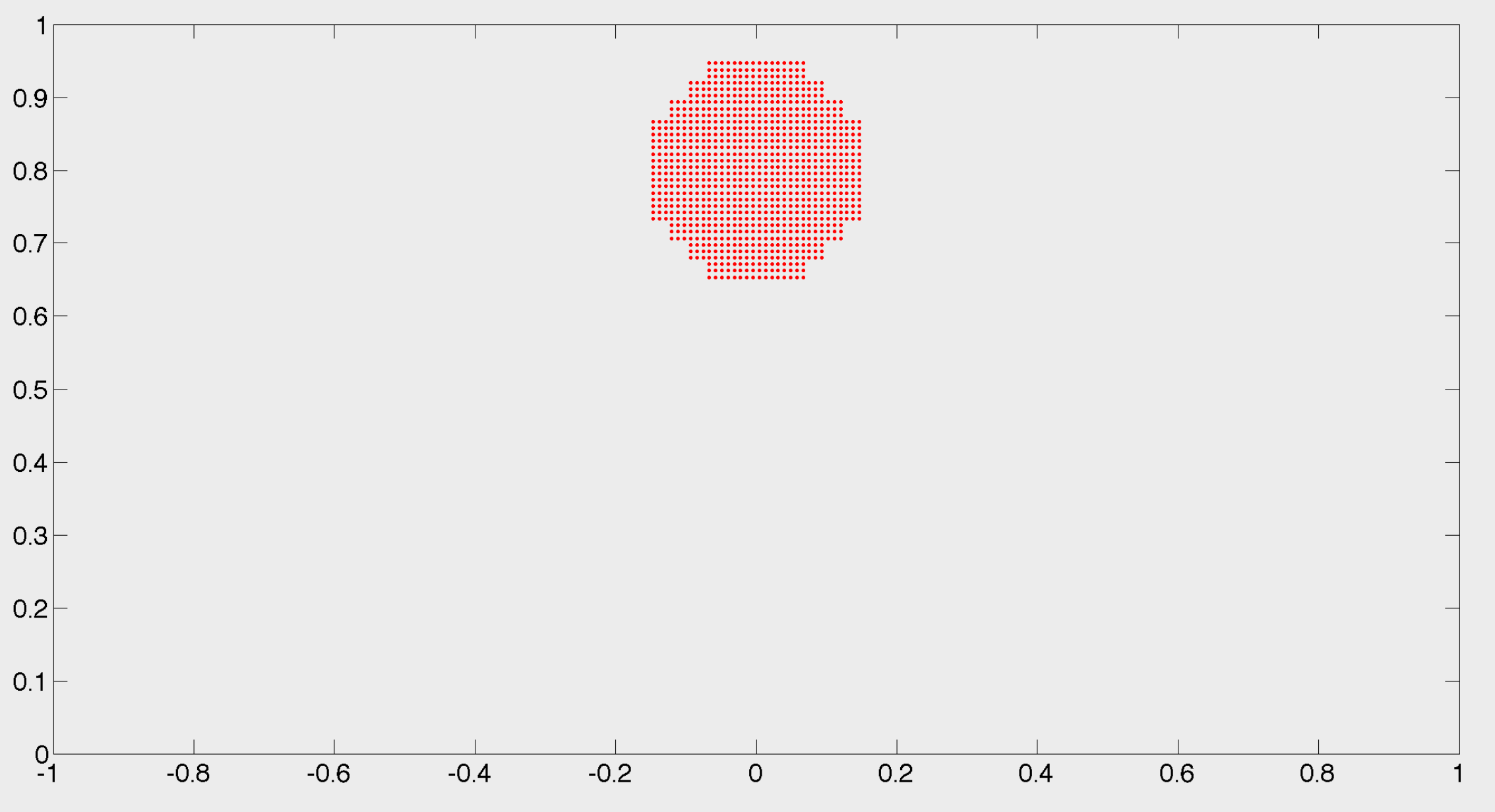}}\qquad
\subfloat[t=0.1]{\includegraphics[width=3in]{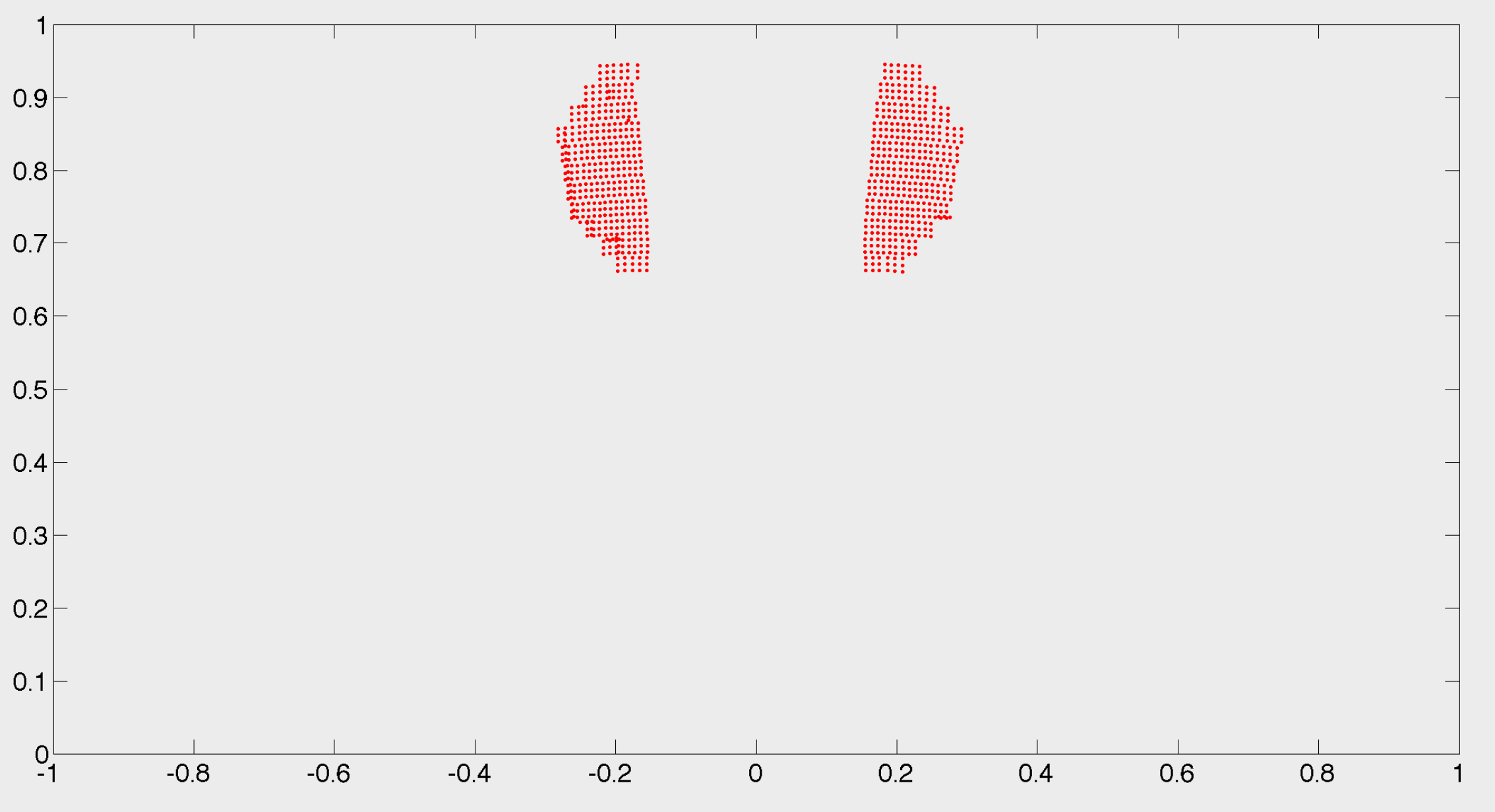}}\qquad
\subfloat[t=0.25]{\includegraphics[width=3in]{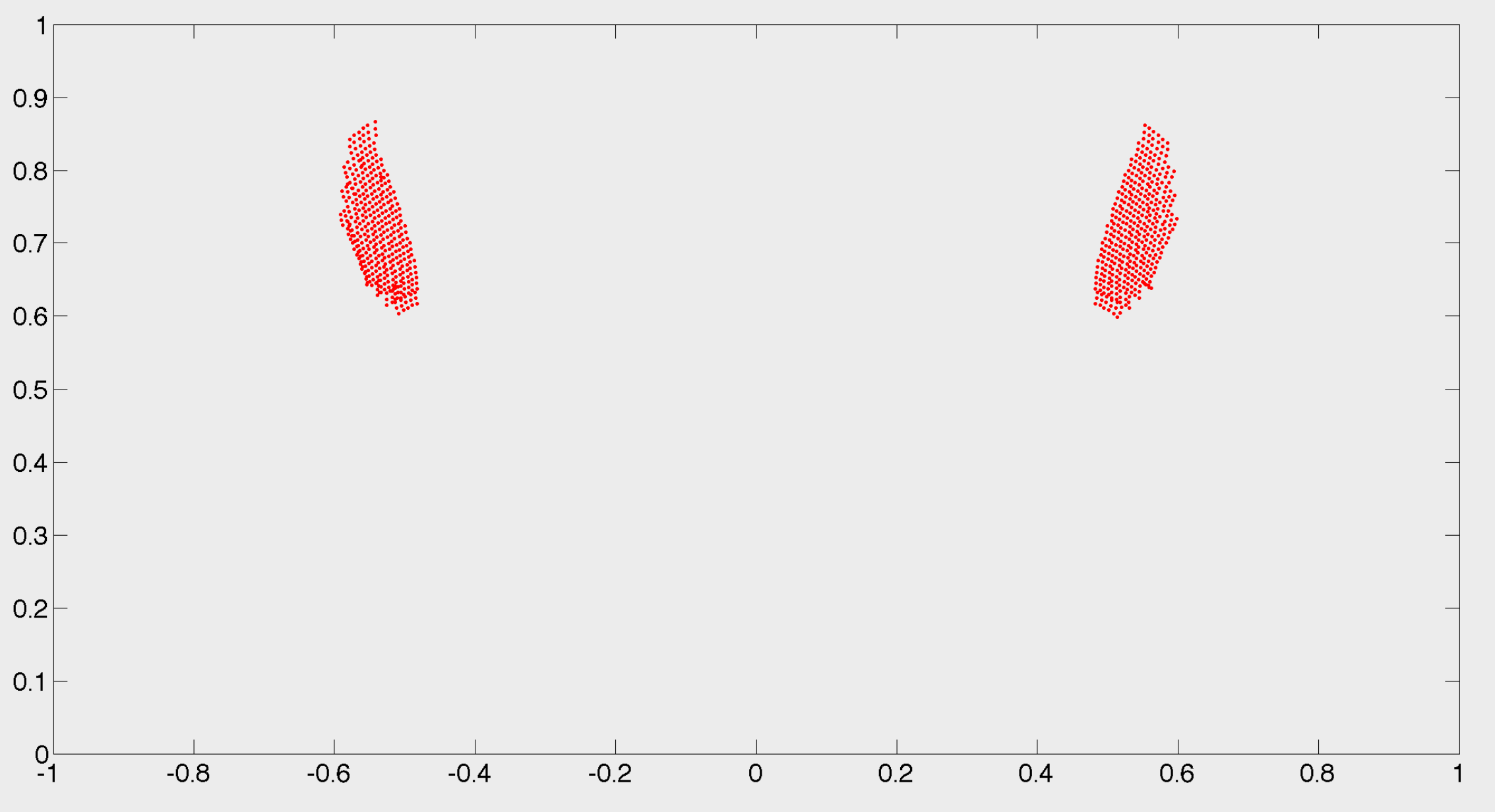}}\qquad
\subfloat[t=0.5]{\includegraphics[width=3in]{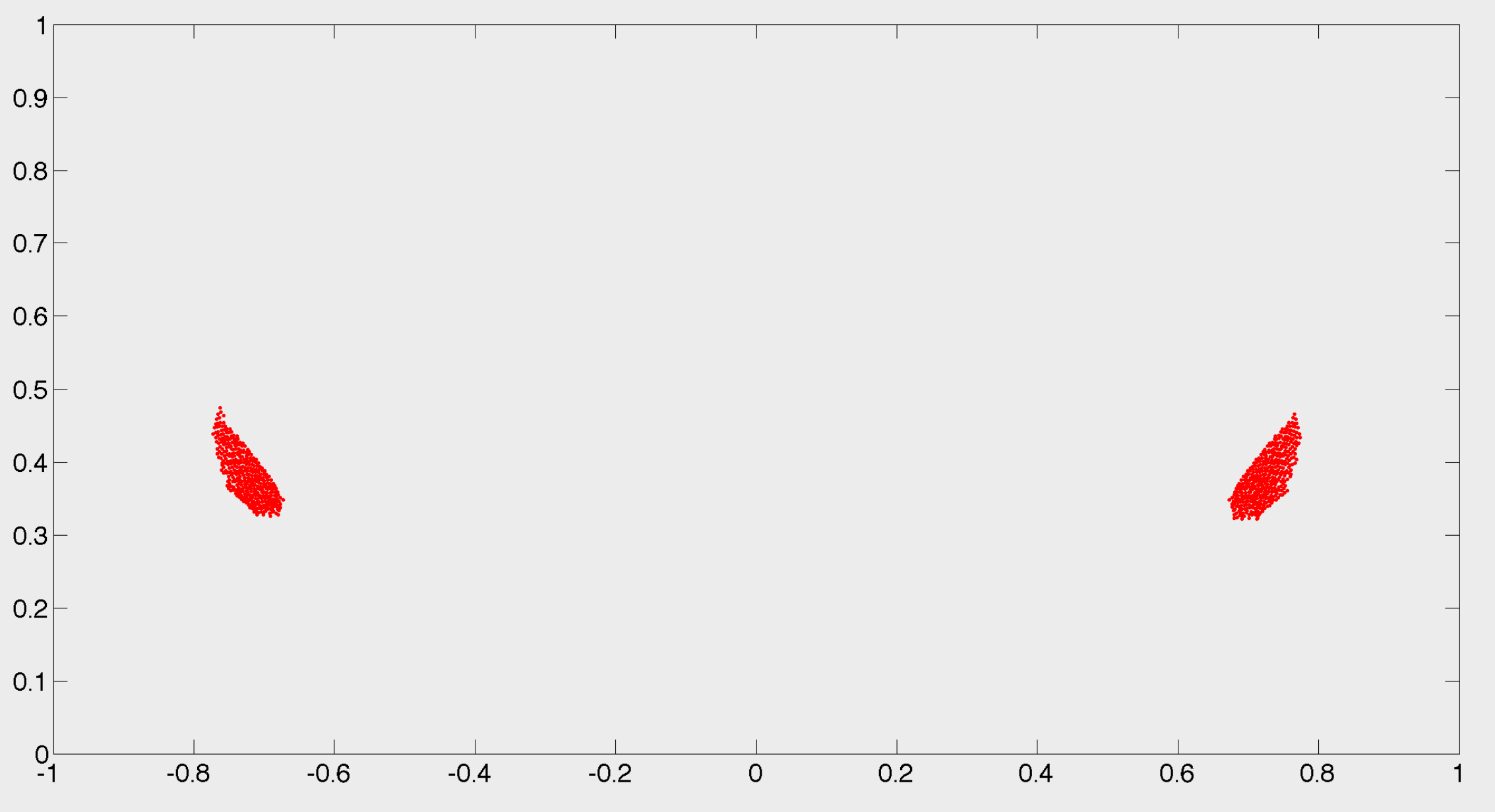}}\qquad
\subfloat[t=0.75]{\includegraphics[width=3in]{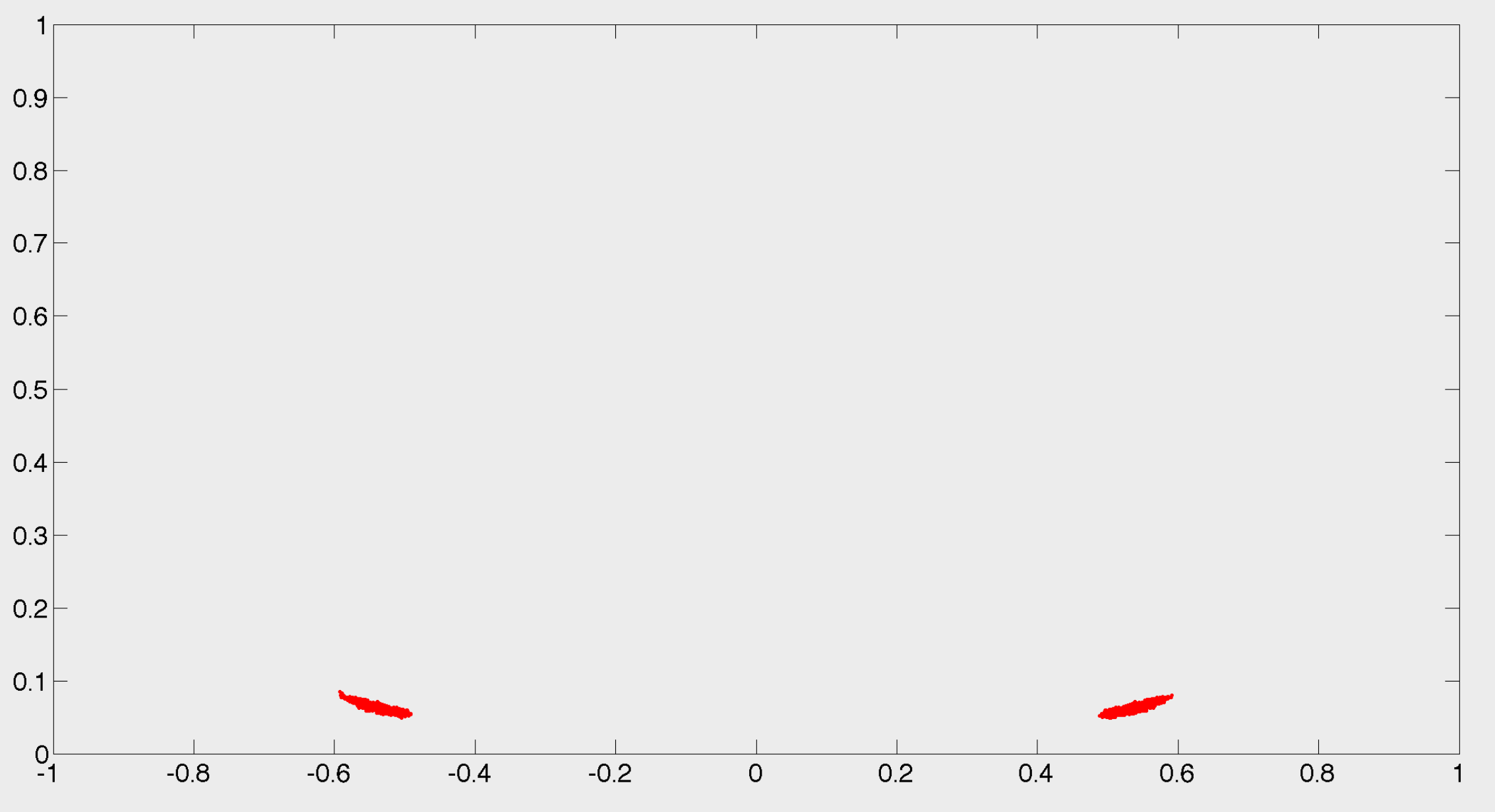}}\qquad
\subfloat[t=0.95]{\includegraphics[width=3in]{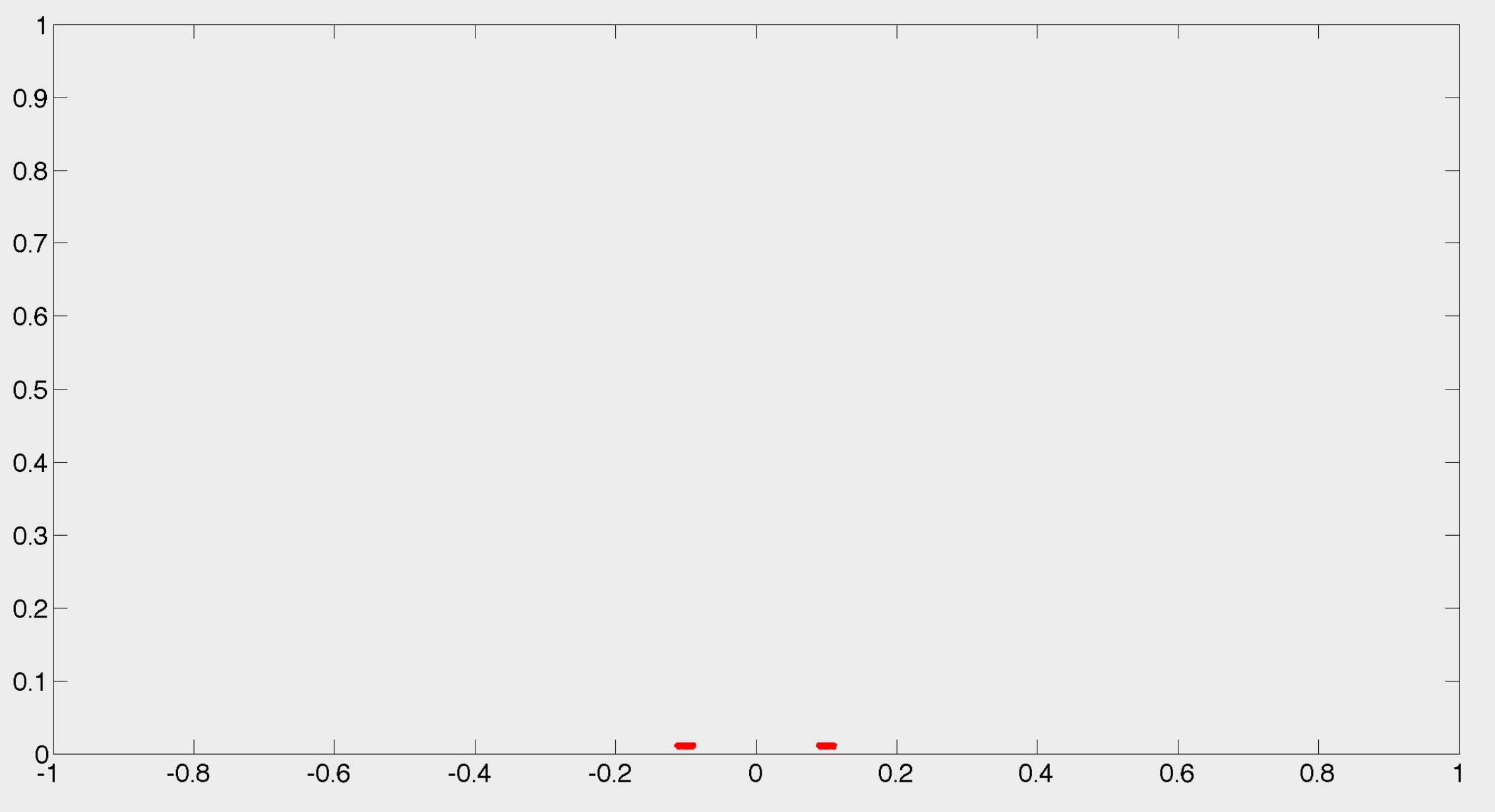}}\qquad
\caption{\footnotesize{Particle trajectories with feedback control computed using Eq. (\ref{eq:feedback}) from the optimal transport solution in the Grushin plane. Each box contained in the support of uniform initial measure $\mu_0$ is initially populated with 4 particles.}}
\label{fig:particle_Gr}
\end{figure}

\subsection{Optimal Transport in Time-Periodic Double-Gyre system}
Now we consider a measure transport problem for the time-periodic double-gyre system \cite{lekien2007chaotic}. This chaotic dynamical system has been analyzed using several computational tools related to transport and mixing \cite{poje1999geometry, wiggins2005dynamical, FrPa09, tallapragada2013set, froyland2015optimal}. 
The controlled equations we consider as follows,
\begin{subequations}
\begin{align}
\dot{x}=-\pi A\sin(\pi f(x,t))\cos(\pi y) + u_1,\label{eq:DG1}\\
\dot{y}=\pi A\cos(\pi f(x,t))\sin(\pi y)\dfrac{df(x,t)}{dx}+u_2\label{eq:DG2},
\end{align}
\end{subequations}

where $f(x,t)=\beta\sin(\omega t)x^2+(1-2\beta\sin(\omega t))x$ is the time-periodic forcing in the system. \PGn{Since the above system with drift has control vector fields $g_1(x,y)=[1 \:\:0]^\intercal$, and $g_2(x,y)=[0 \:\:1]^\intercal$, that span the tangent space $\mathbb{R}^2$, it satisfies condition $2$ of Theorem \ref{thm:strngcon}. By Theorem \ref{drfl_ctrb}, the corresponding optimal transport problem for this system with drift is well-posed, since there exists a transport between any pair of measures.}

The phase space is $X=[0\: \: 2]\times[0\: \: 1]$. We first describe the dynamics of the uncontrolled ($u_1=u_2=0$) system. For the trivial case of $\beta=0$ (i.e. no time-dependent forcing), the phase space is divided into two invariant sets, i.e., the left and right halves of the rectangular phase space (`gyres'), by a heteroclinic connection between fixed points $x_1=(1,1)$ and $x_2=(0,1)$. For non-zero $\beta$, the Poincare map $F$ of the system, obtained by integrating the dynamics over one time period $\tau$ of $f$, describes an autonomous discrete-time system. The heteroclinic connection is broken in this case, and results in a heteroclinic tangle. This heteroclinic tangle leads to transport between left and right sides via lobe-dynamics. 

We choose parameters $A=0.25, \beta=0.25, \omega=2\pi$, such that the time-period of the flow is $\tau=1$. To get insight into the phase space transport due to heteroclinic tangles, the theory of lobe dynamics \cite{lekien2007chaotic} is useful. Lobe dynamics techniques allows one to quantify the transport between sets separated by invariant manifolds, and their transversal intersections. In figure \ref{fig:DG_manif}, the unstable manifold of $x_1\approx(0.919,1)$, $U_{x_1}$, and the stable manifold of $x_2\approx(1.081,0)$, $S_{x_2}$ are shown in green and white respectively. The lobe  labeled 'A', its pre-image $F^{-1}(A)$ and image $F(A)$ are also shown. Consider the segment $L=  U_{x_1}(x_1\rightarrow F(P_1)) \cup S_{x_2}(F(P_1)\rightarrow x_2)$. Then $L$ divides phase space $X$ into two regions. The points that get mapped from left to right region in one iteration of $F$ are precisely those in set $A$. Hence, the amount of mass transport from left to right side of $L$ is $\bar{m}(A)$.
\begin{figure}[h]
\includegraphics[width=6in]{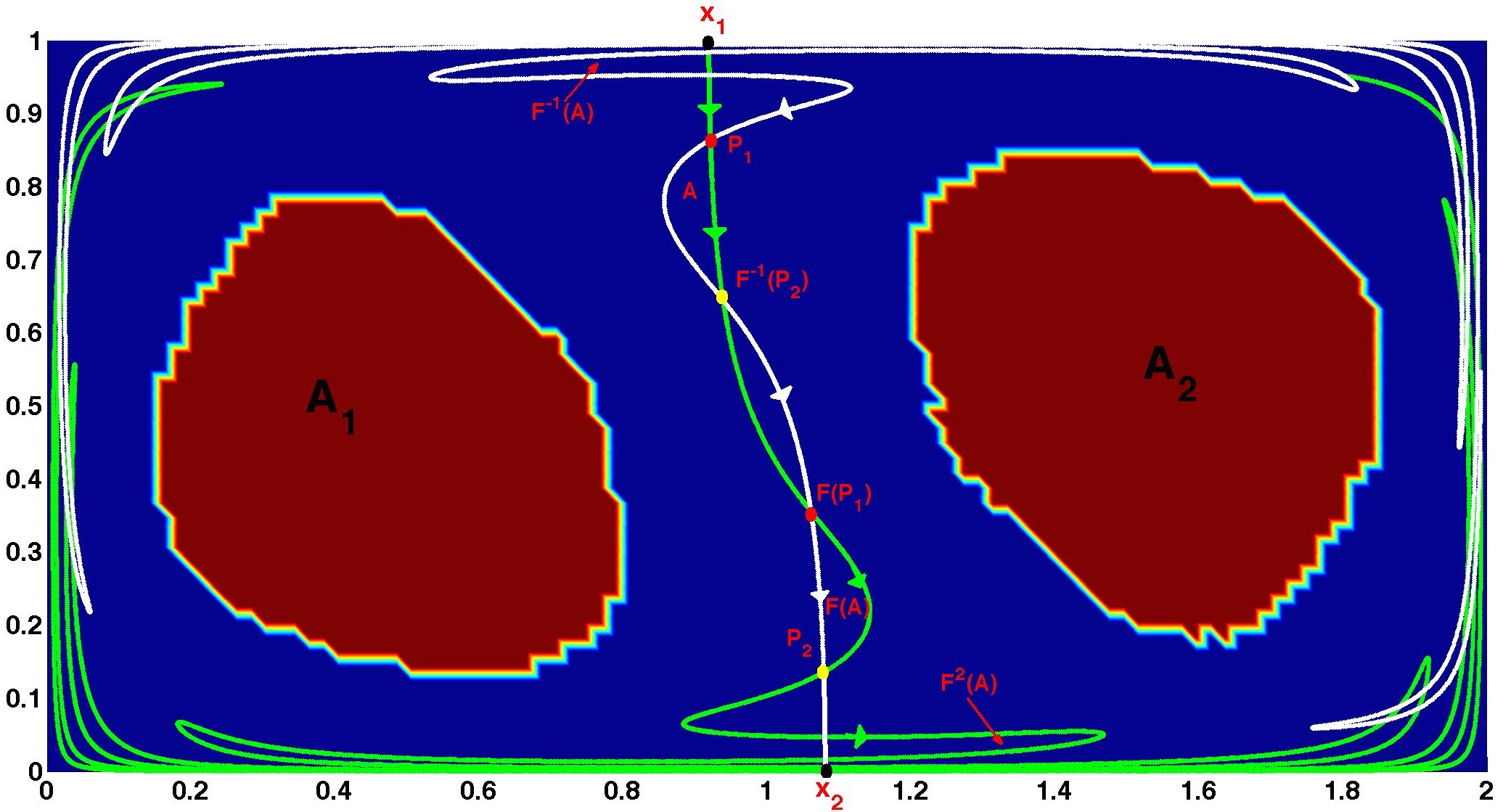}
\caption{\footnotesize{Invariant manifolds and lobe-dynamics in the double-gyre system (reproduced from Ref. \cite{grover2018optimal}). }}
\label{fig:DG_manif}
\end{figure}
While our algorithm for optimal transport can be applied between any arbitrary pair of measures, it is instructive to choose a pair of measures which are `distinguished' for the given system. Here, we choose almost-invariant sets as in Ref. \cite{grover2018optimal}.

In Figure \ref{fig:DG_manif}, two almost-invariant sets, $A_1$ and $A_2$ are also shown. We choose the initial and final measures to be uniform measures supported on $A_1$ and $A_2$, respectively. Both measures are normalized to sum to unity. \PGn{We solve the the optimal transport problem by discretizing $X$ into $m=60\times 30$ boxes for different time horizons $t_f$. For each $t_f$, we choose $k$ such that $\Delta t=\dfrac{t_f}{k}=\dfrac{1}{40}$. We use a piecewise-constant approximation of the time-dependent drift vector field, i.e., $g_0(x,t_j+\delta t)\approx g_0(x,t_j)\; \forall\; \delta t\leq \Delta t, j\leq k$, for computing the corresponding generator using Eq. (\ref{eq:gen1}). The computation time for obtaining solution in CVX for the $t_f=1$ case is $9\times 10^3$ seconds, while the $t_f=10$ case takes $7.5\times 10^4$ seconds, and the $t_f=15$ takes about $10^5$ seconds. Our simulations are repeated with a finer grid-size $m=100\times 50$ to verify that our results are nearly independent of $m$. }

In Fig. \ref{fig:DG_OT_n1}, we show the optimal transport sequence for $t_f=\tau=1$. In other words, the transport is constrained to be completed in one time-period of the uncontrolled flow. Due to the short time-horizon, all of the mass is pushed across the invariant manifolds separating the the two gyres. In Fig. \ref{fig:DG_OT_n10}, the optimal transport sequence for the more interesting case with $t_f=10$ is shown. The transport process in this case is completely different than the $t_f=1$ case. We observe that the transport occurs in a `quantized' manner, i.e. packets of mass are transported, one at a time, via lobe-dynamics from left gyre to the right gyre. The number of such packets \emph{exactly} equal the number of time-periods in the time-horizon of the transport problem, i.e. $10$ in this case. Hence, while the global transport is being done by the natural dynamics via lobes, the role of control is to gather the mass in pre-images of the those lobes. For instance, in Figs \ref{fig:DG_OT_n10}((b)-(e)), the transport of one such packet via the sequence $F^{-1}(A)\rightarrow A\rightarrow F(A) \rightarrow F^2(A)$ is shown. The mechanism is essentially the same for other time-horizons that we analyzed, $t_f=2,5,8,12.5 \:\&\:15$. The increasing use of `natural' chaotic lobe-dynamics of the uncontrolled system during optimal transport should reflect in the optimal transport cost. This cost, $\tilde{W}$, decreases rapidly as the time-horizon $t_f$ is increased, as shown in Fig. \ref{fig:tVsCost_dg_gen}. Hence, the optimal transport solution discovers efficient paths in this chaotic system, where `going with the flow' is the best option. \PGn{This result provides a continuous-time control interpretation of the results of the discrete-time switching algorithm in Ref. \cite{grover2018optimal}. }

\begin{figure}[h!]
\centering
\subfloat[$t=0$]{\includegraphics[width=3in]{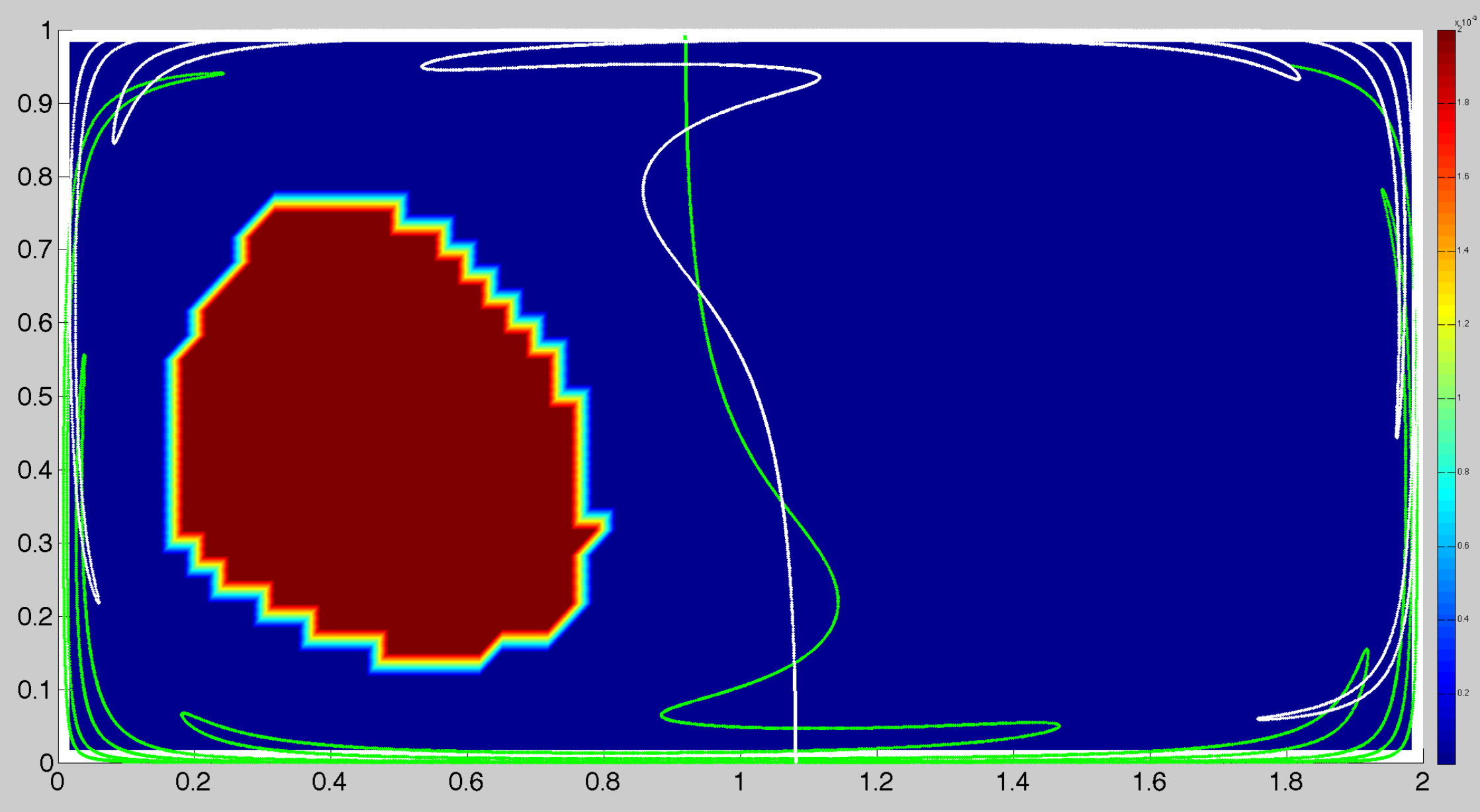}}\qquad
\subfloat[$t=0.25$]{\includegraphics[width=3in]{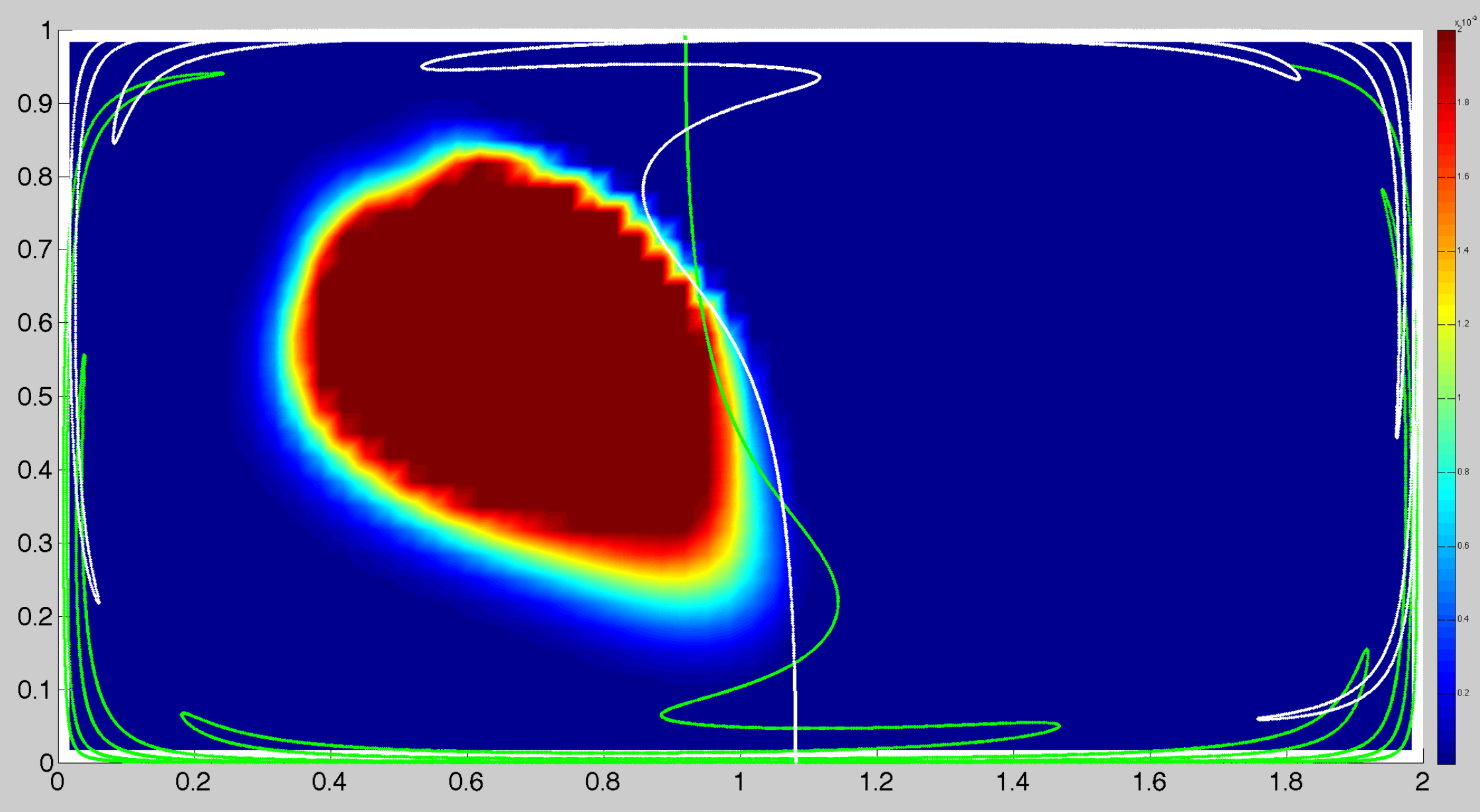}}\qquad
\subfloat[$t=0.5$]{\includegraphics[width=3in]{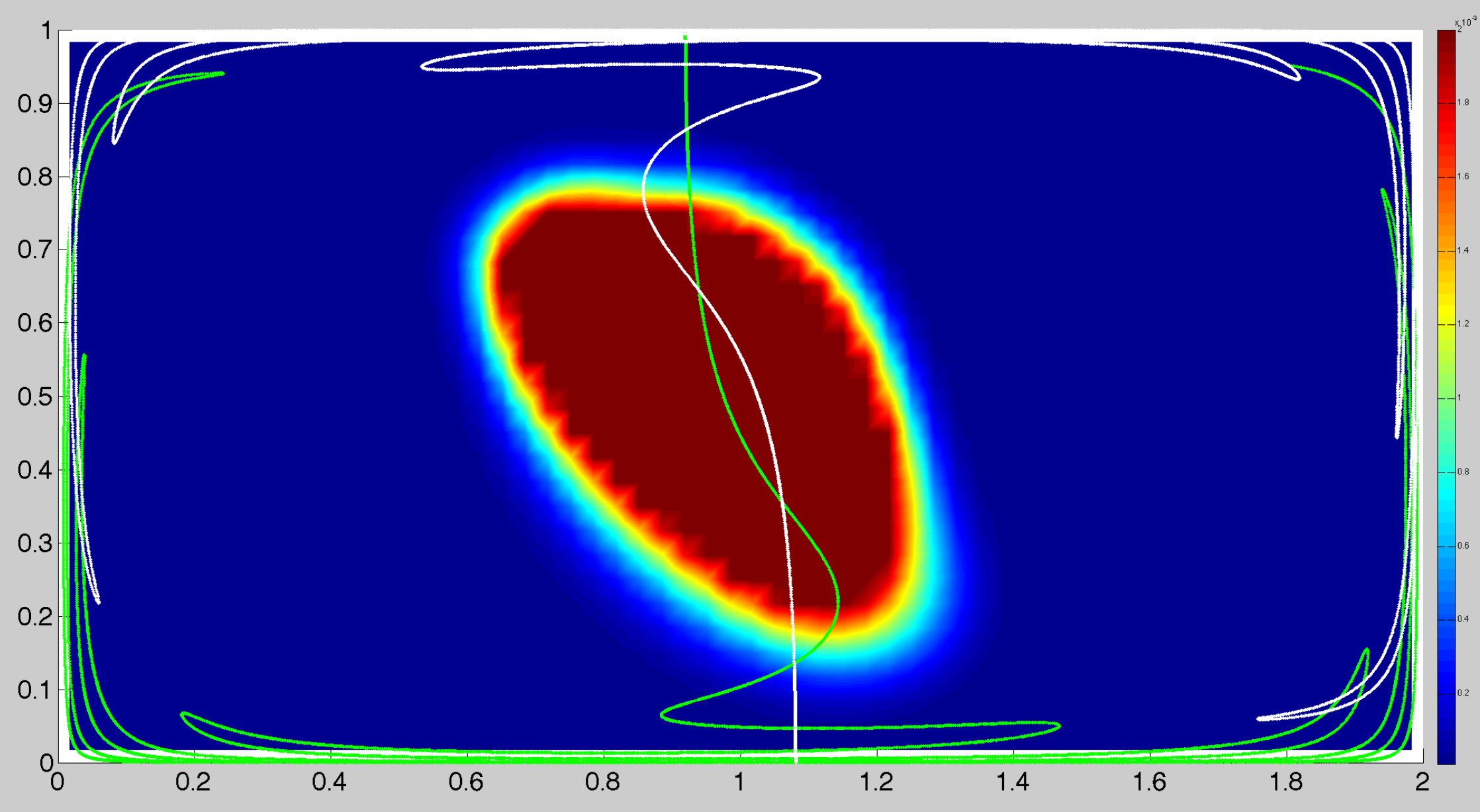}}\qquad
\subfloat[$t=0.75$]{\includegraphics[width=3in]{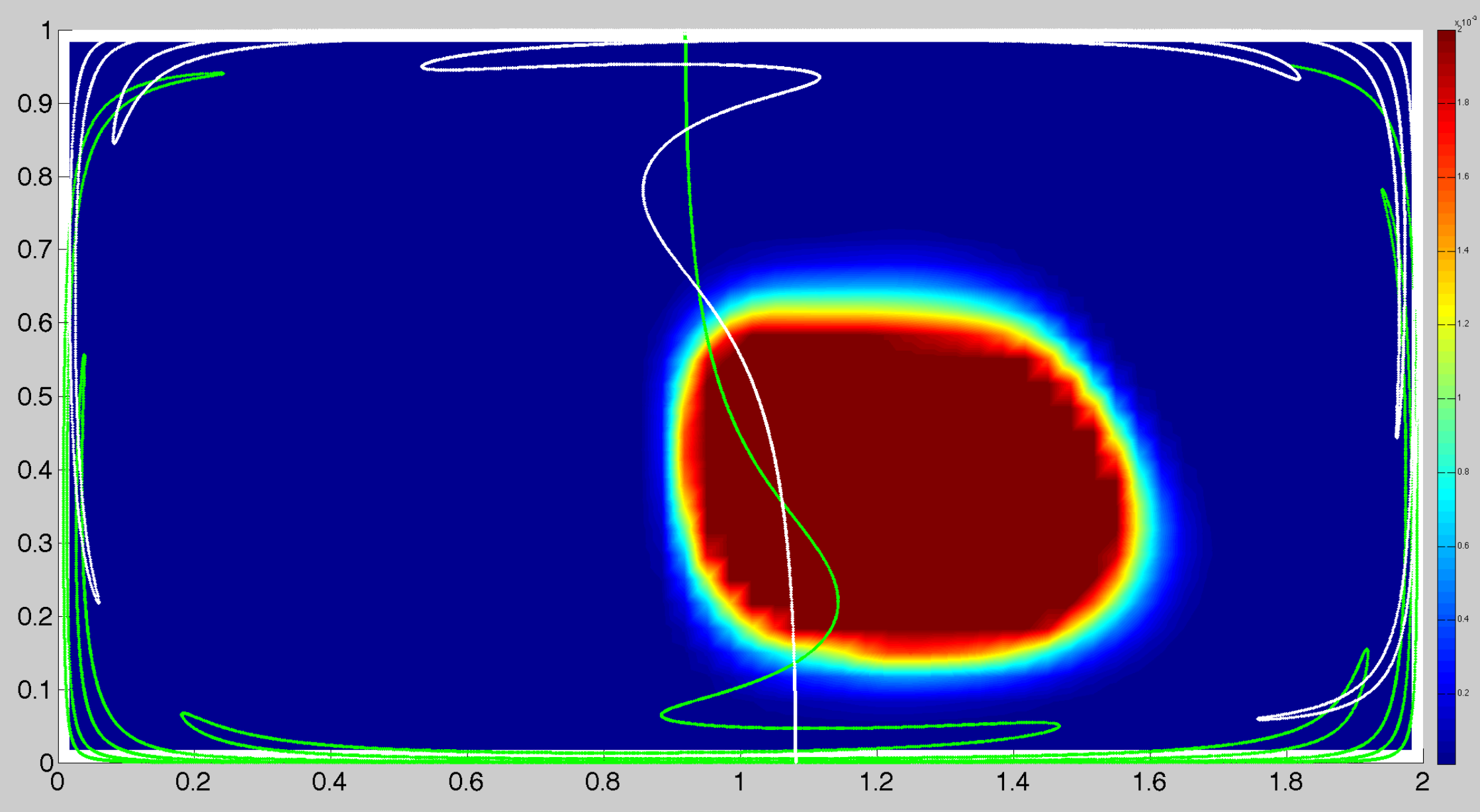}}\qquad
\subfloat[$t=1$]{\includegraphics[width=3in]{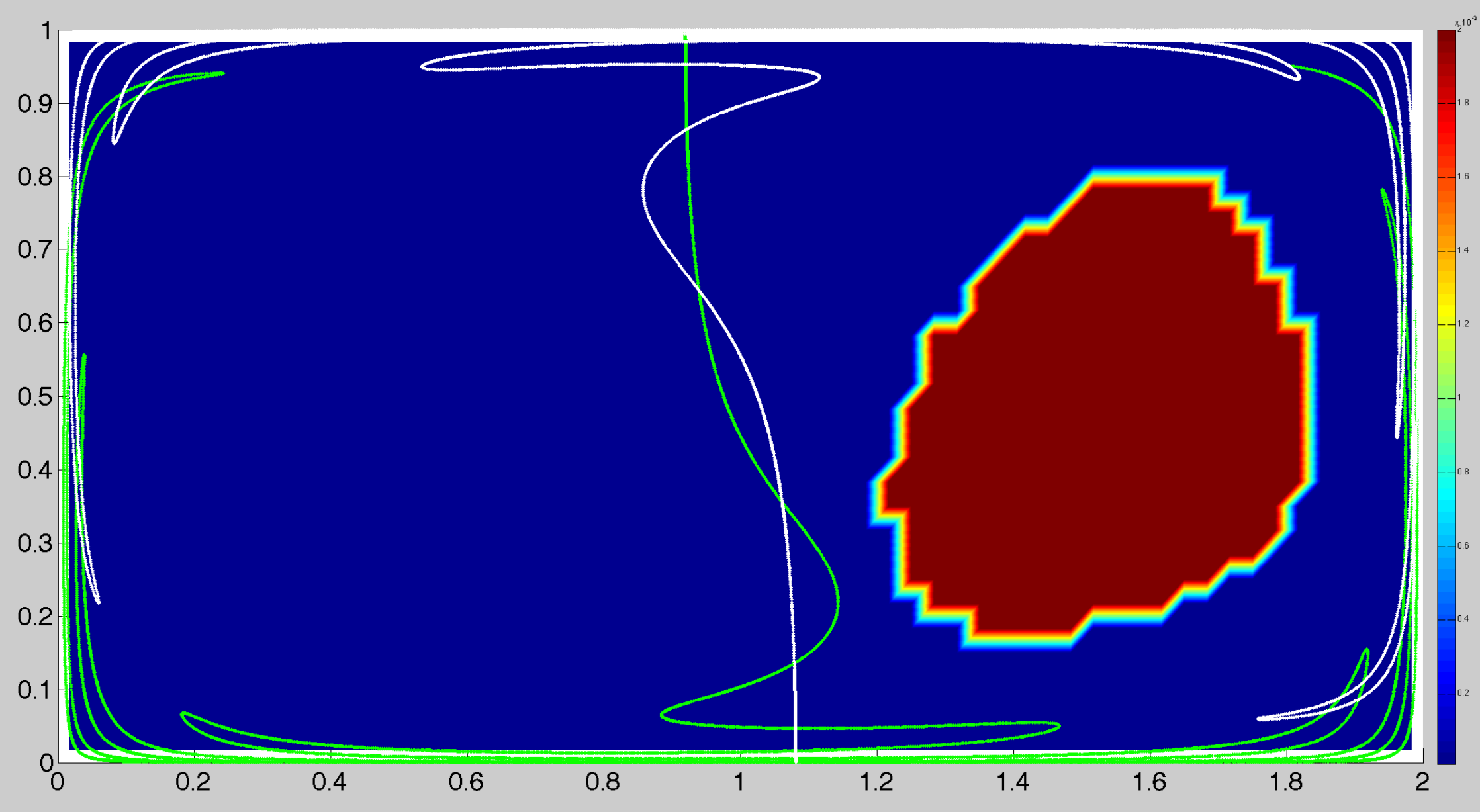}}
\caption{\footnotesize{Optimal transport in the periodic double gyre system (Eqs.(\ref{eq:DG1}-\ref{eq:DG2})) between measures shown in (a) and (e) for $t_f=1$. The parameters are $m=1800, \Delta t=\dfrac{1}{40}$. }}
\label{fig:DG_OT_n1}
\end{figure}

\begin{figure}[h!]
\centering
\subfloat[$t=0$]{\includegraphics[width=3in]{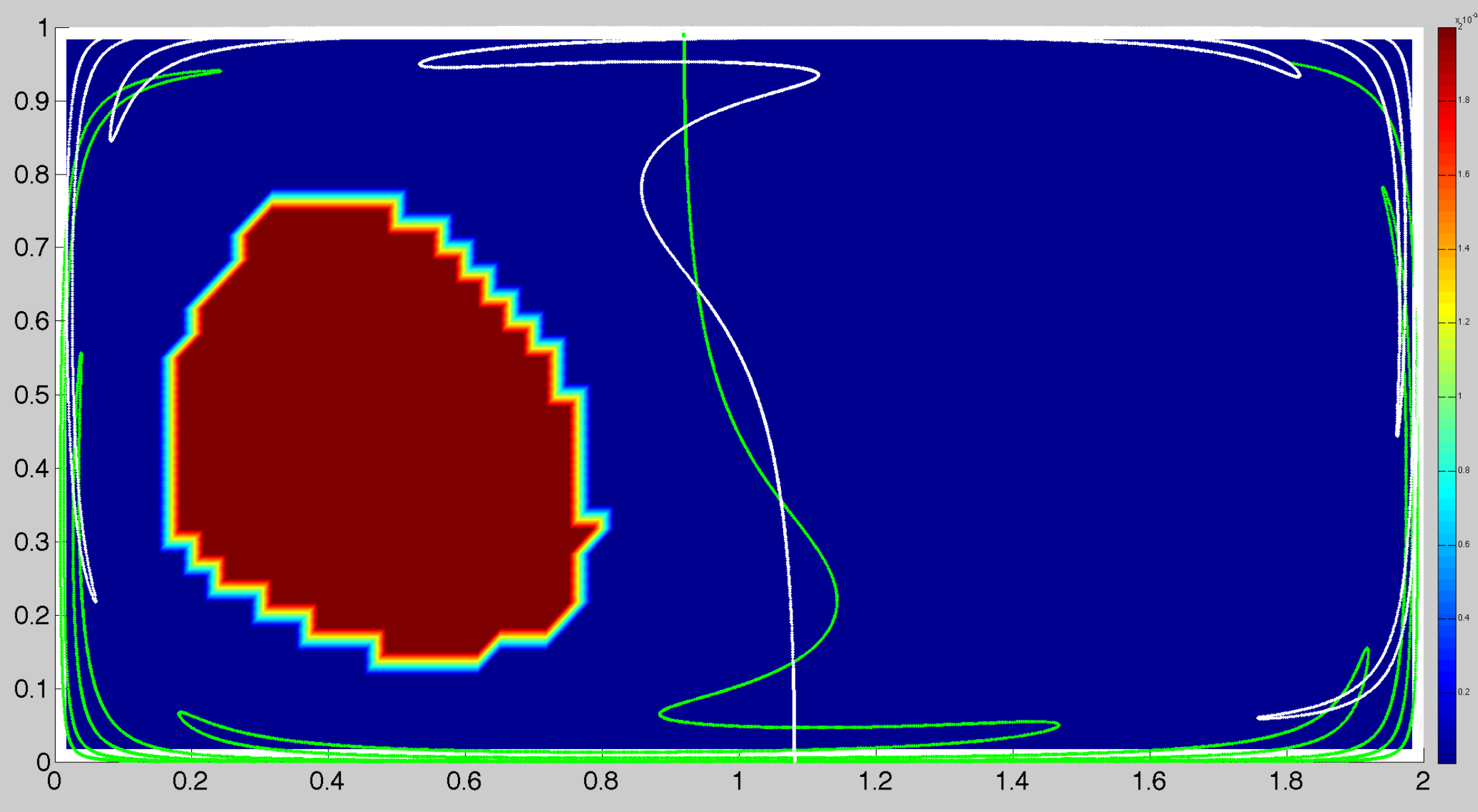}}\qquad
\subfloat[$t=2$]{\includegraphics[width=3in]{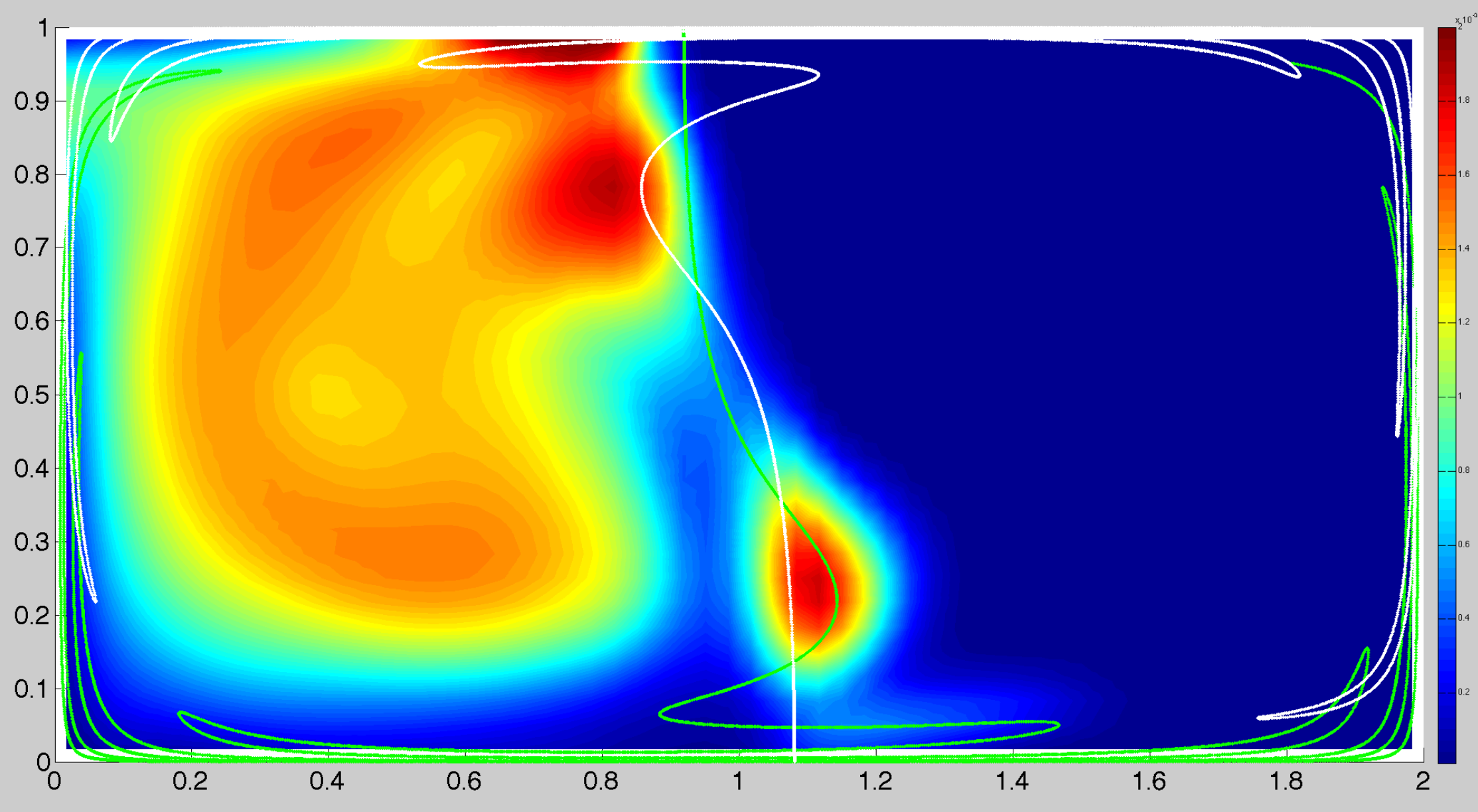}}\qquad
\subfloat[$t=3$]{\includegraphics[width=3in]{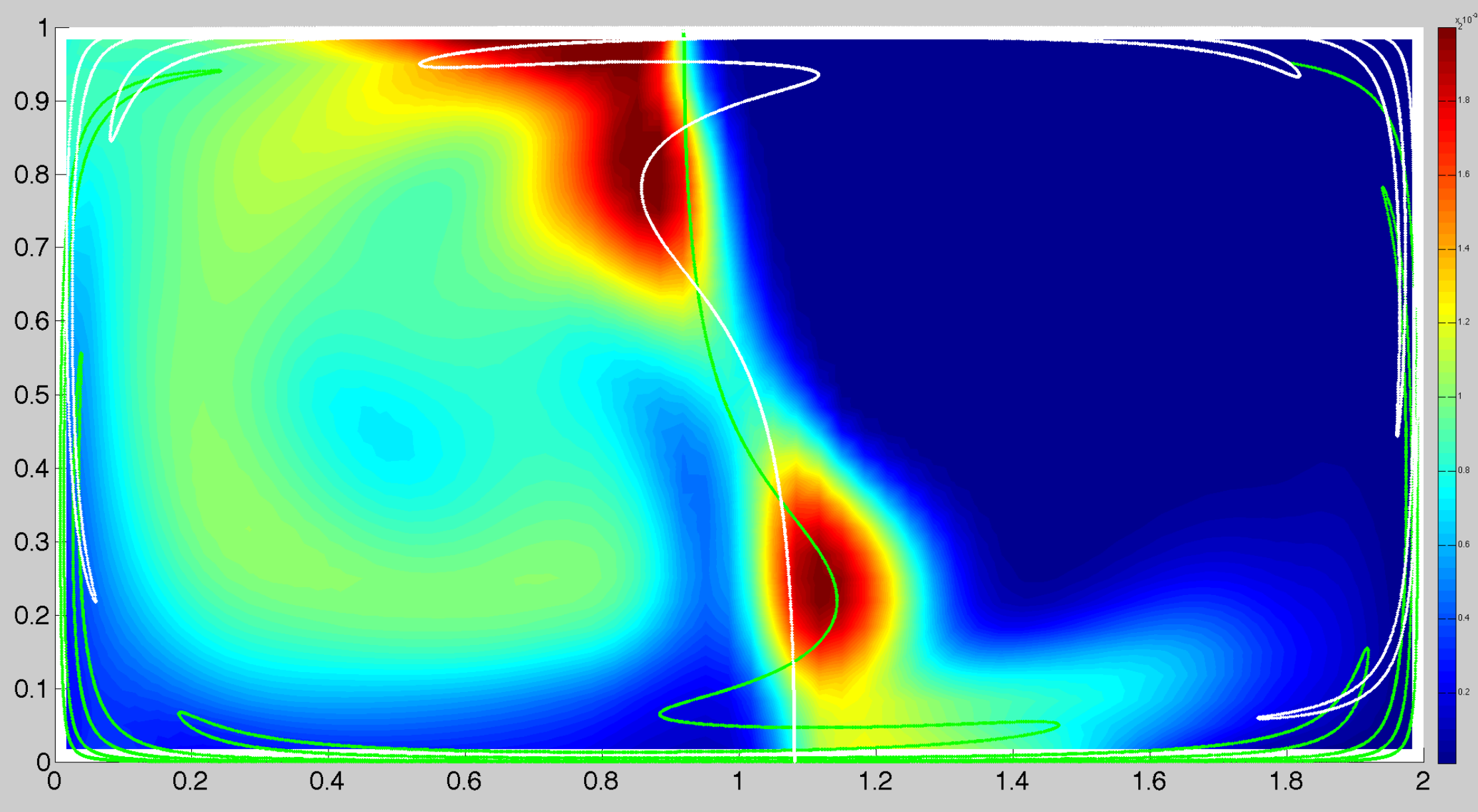}}\qquad
\subfloat[$t=4$]{\includegraphics[width=3in]{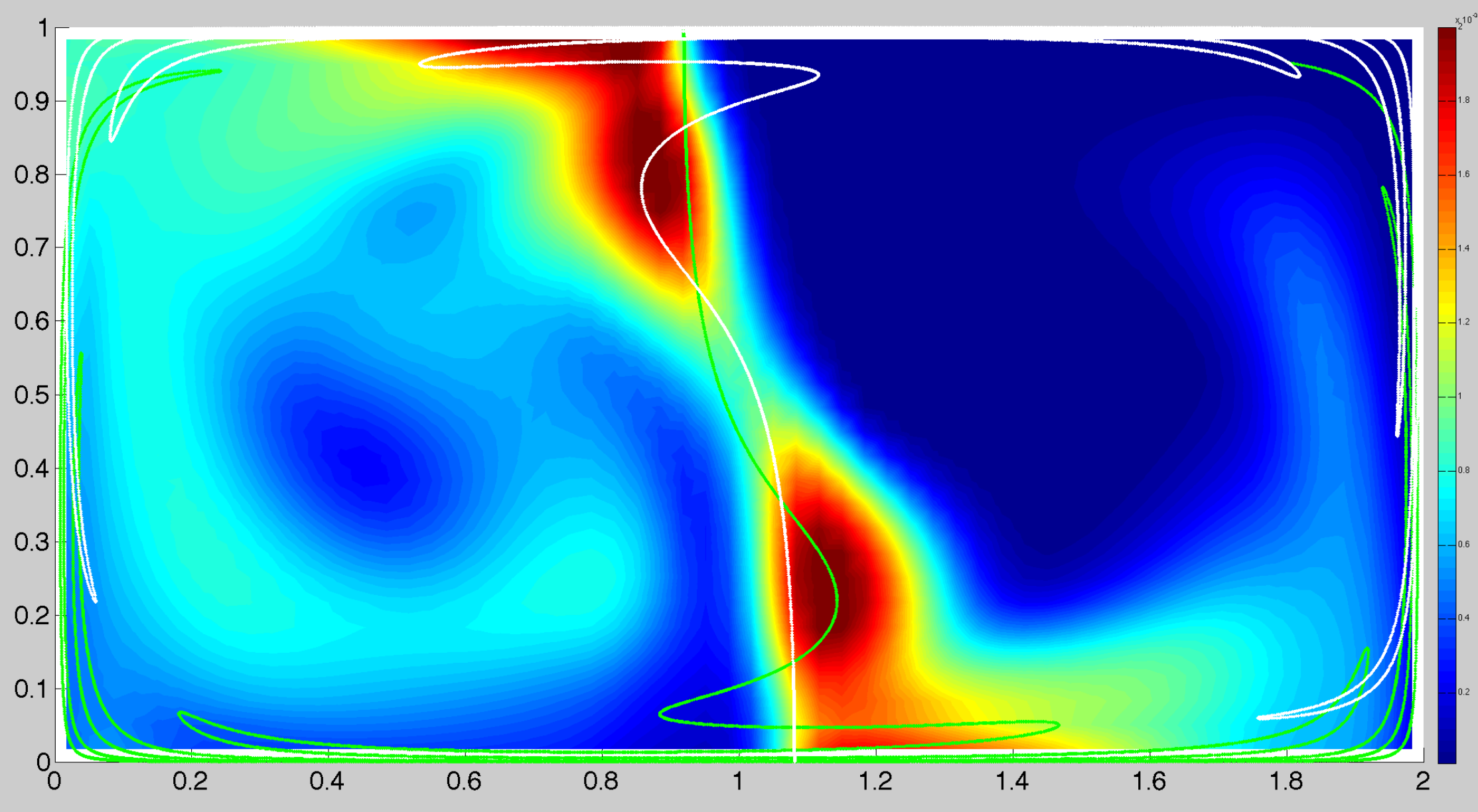}}\qquad
\subfloat[$t=5$]{\includegraphics[width=3in]{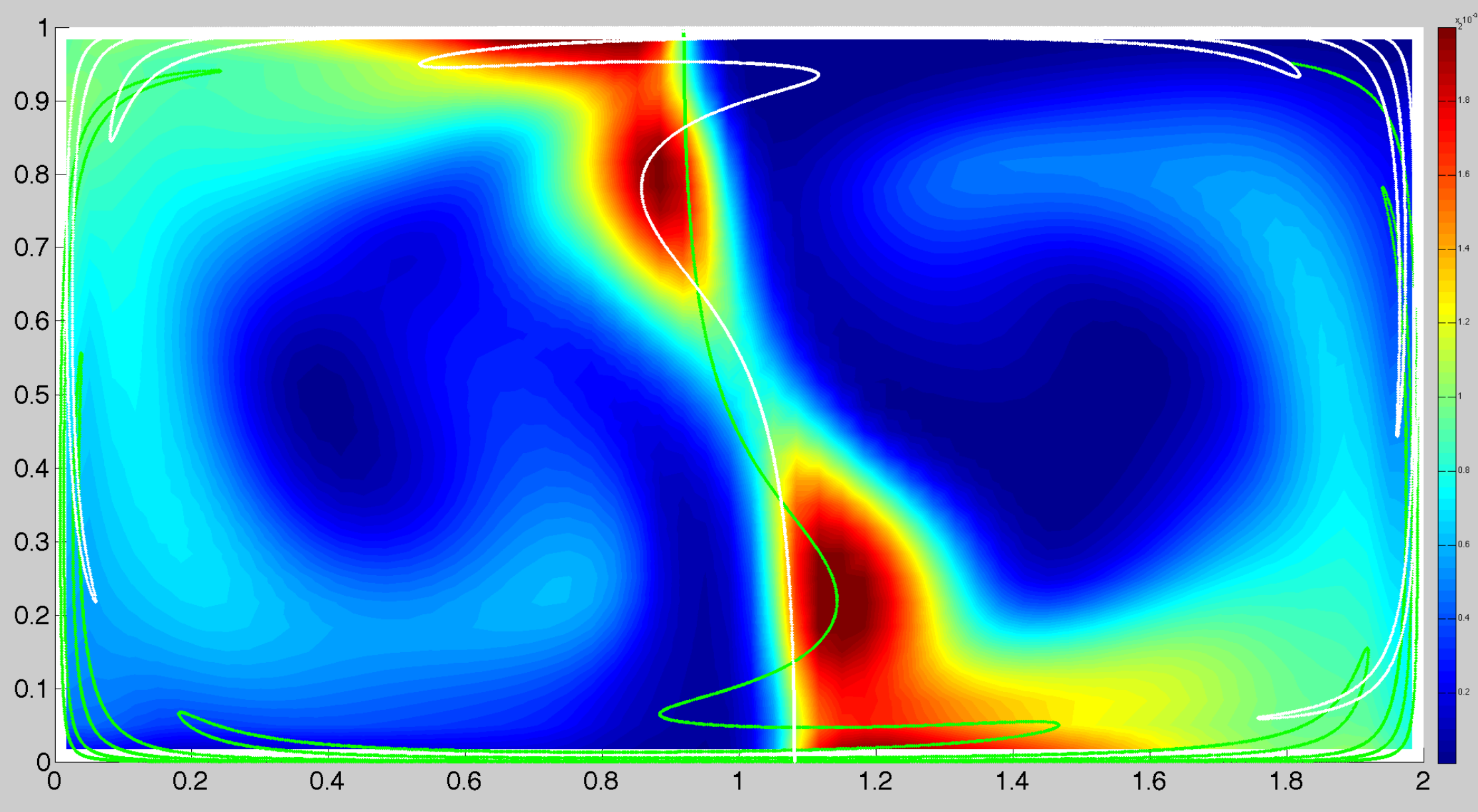}}\qquad
\subfloat[$t=8.5$]{\includegraphics[width=3in]{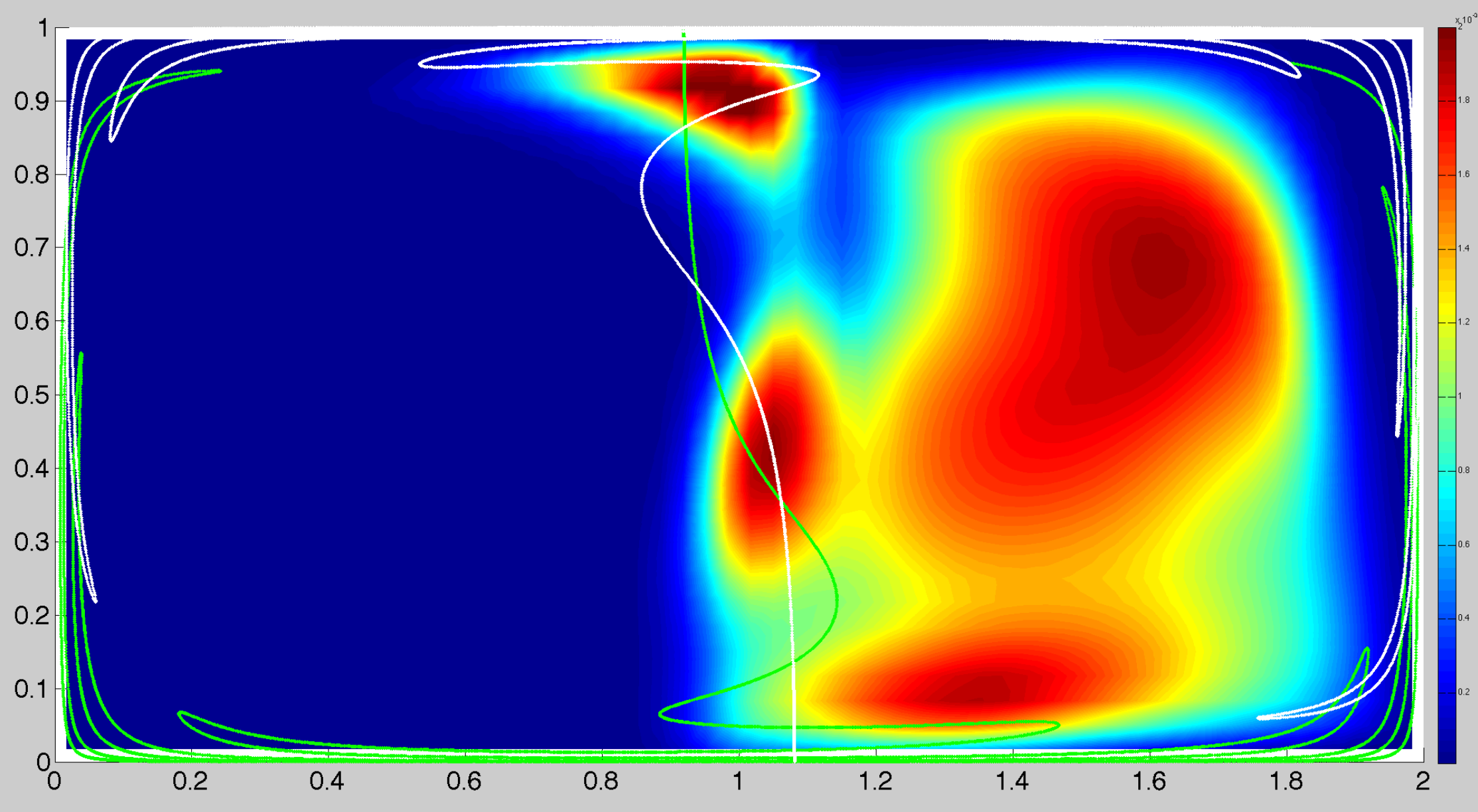}}\qquad
\subfloat[$t=9$]{\includegraphics[width=3in]{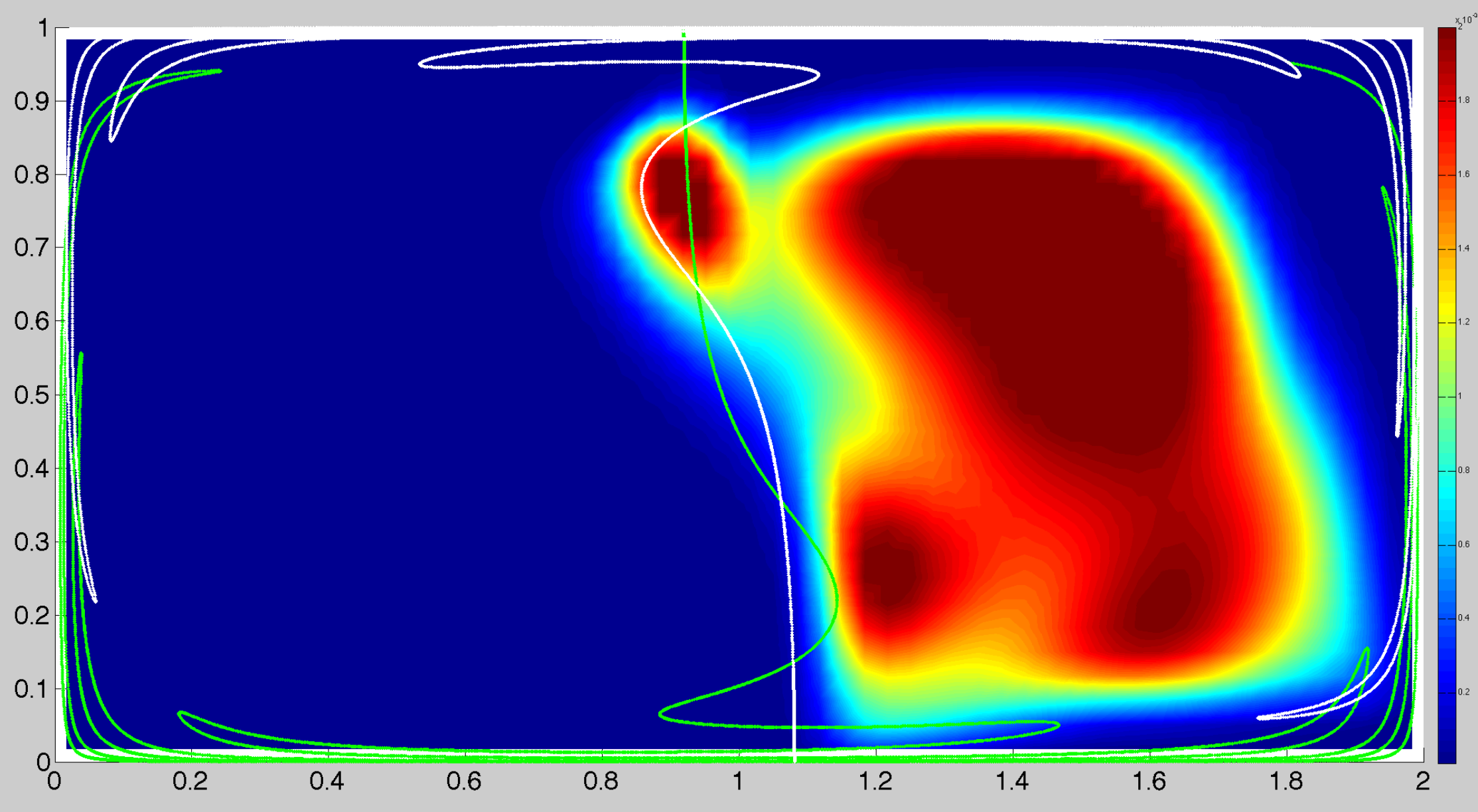}}\qquad
\subfloat[$t=10$]{\includegraphics[width=3in]{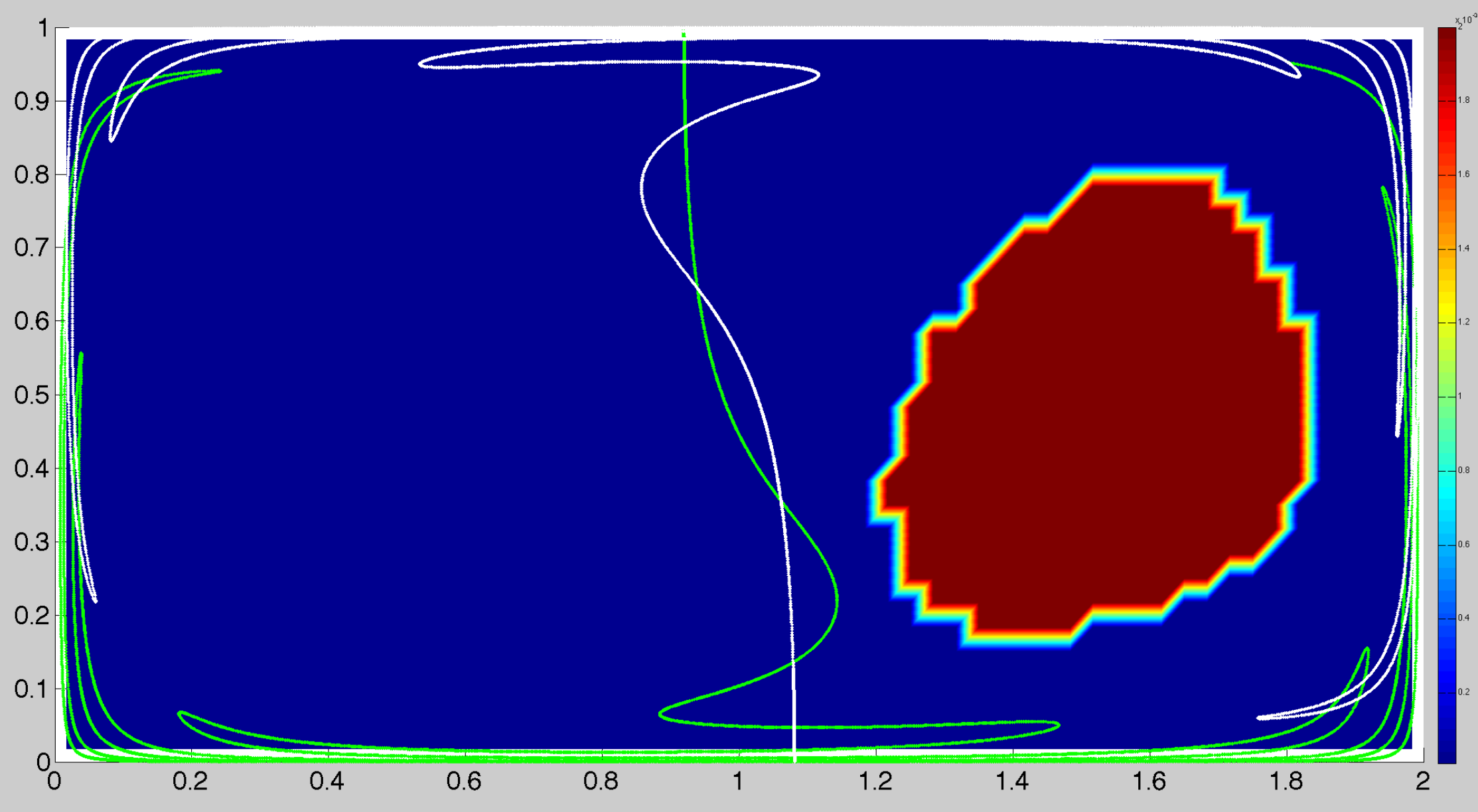}}\qquad
\caption{\footnotesize{Optimal transport in the periodic double gyre system between measures shown in (a) and (h) for $t_f=10, \Delta t=\dfrac{1}{40}$. The optimal transport solution shows a quantization phenomenon. Ten `packets' are transported via lobe-dynamics from the left side to the right side of the domain. (b)-(e) The transport of third packet to right side via the sequence $F^{-1}(A)\rightarrow A\rightarrow F(A) \rightarrow F^2(A)$.(f)-(g) The last packet gets transported to the right side. Animation available at: \url{https://www.youtube.com/watch?v=Pu7sCkpm4RY}}}
\label{fig:DG_OT_n10}
\end{figure}

\begin{figure}[h!]
\includegraphics[width=5in]{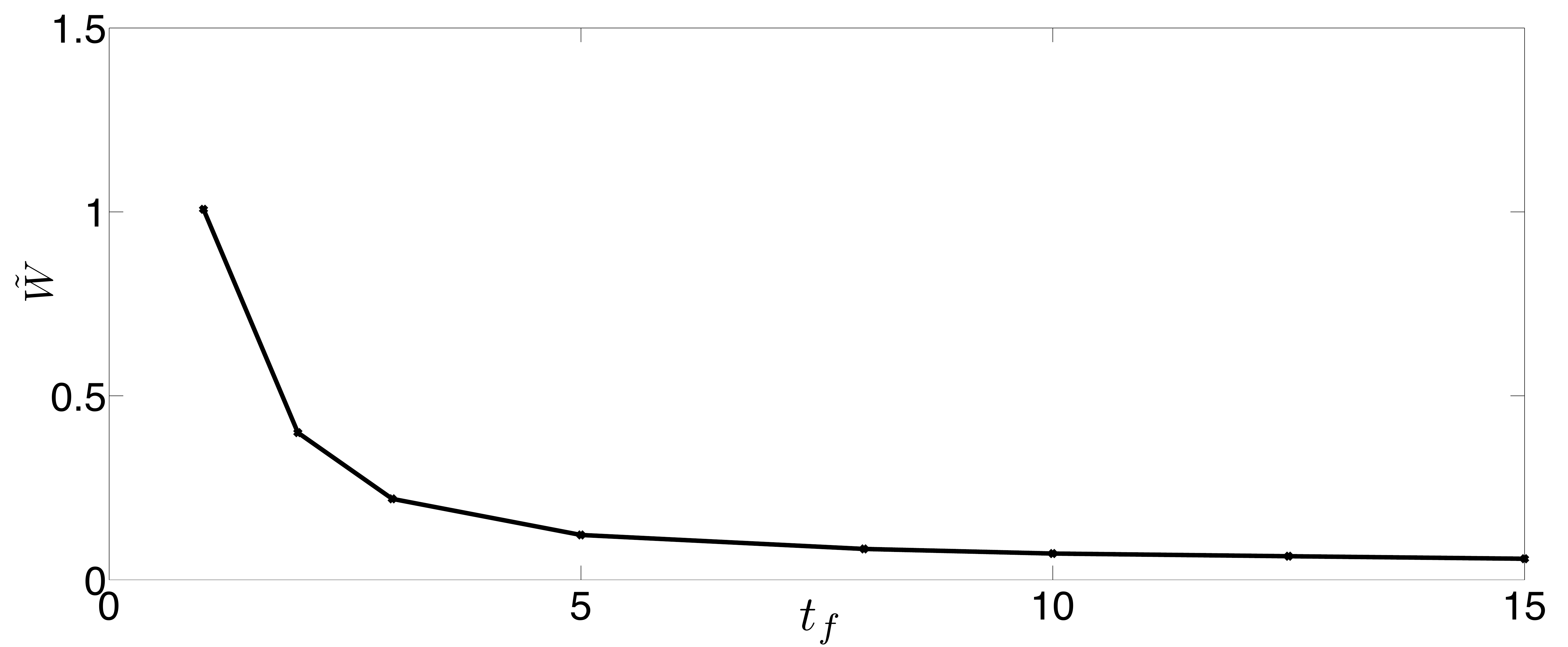}
\caption{\footnotesize{The cost of optimal transport between two measures supported on two AIS for the double-gyre system, as a function of time-horizon of the problem. }}
\label{fig:tVsCost_dg_gen}
\end{figure}
\subsection{Optimal Transport for Unicycle Model}
Finally, we consider optimal transport in a three-dimensional non-holonomic system, called the `unicycle' model. This system is a toy model for vehicle kinematics, and is used extensively in vehicle path planning and control \cite{murray1993nonholonomic,aicardi1995closed}. The states are cartesian coordinates $(x,y)\in\mathbf{R}^2$, and orientation $\theta\in \mathbf{S}^1$ of the unicycle. \PGn{The system equations on $M= \mathbf{S}^1\times \mathbf{R}^2$ are given by
\begin{align*}
\dot{\theta}=u_1,\\
\dot{x}=u_2\cos{\theta},\\
\dot{y}=u_2\sin{\theta},
\end{align*}

where $u_1$ is the steering speed, and $u_2$ is the translation speed. The above system is a driftless system with control vector fields $g_1(\theta,x,y)=[1 \; 0 \; 0]^\intercal$,  $g_2(\theta,x,y)=[0 \;\cos{\theta} \; \sin{\theta}]^\intercal$. These do not span the tangent space $T_xM$, but their Lie-algebra does, i.e. $Lie_x \bigg \lbrace g_i: i \in \lbrace 1,2 \rbrace \bigg \rbrace = T_xM$. This can be seen by noting that the Lie-bracket $[g_1,g_2]=[0 \;-\sin{\theta}\;\cos{\theta}]^{\intercal}$ does not lie in  $span \{g_1,g_2\}$. Hence, this system satisfies condition $1$ of Theorem \ref{thm:strngcon}. By Theorem \ref{thm:g0}, the corresponding optimal transport problem for this driftless system is well-posed.}
The optimal control problem has been studied for various cost functions, and endpoint conditions \cite{krishnaprasad1993optimal,Kirillova2004,justh2015optimality}. The techniques from geometric mechanics, specifically Lie-Poisson reduction \cite{bloch2003nonholonomic}, have been successfully used to reduce the optimal control problem to a three-dimensional non-canonical Hamiltonian system. For this three-dimensional system, two conserved quantities can be found, and hence, the optimal controls $(u_1(t),u_2(t))$ can be solved explicitly in terms of Jacobi elliptic functions.

 To study the optimal transport problem for the unicycle model, take the control cost to be quadratic, i.e. $d(z_1,z_2)=\inf_{\mathbb{U}_{z_1}^{z_2}}\int_0^1\sqrt{u_1^2+u_2^2}dt$. We compute optimal transport solutions for two scenarios. In the first case, $\mu_0$ is chosen to be a measure supported on box containing $(0,0.5,0)$, and $\mu_1$ is chosen to be uniform measure supported on the union of boxes containing $(1,0,0)$ and $(1,1,0)$. In the second case, $\mu_0$ is chosen to be a measure supported on box containing $(0,0.5,0)$, and $\mu_1$ is chosen to be a uniform measure supported on the union of boxes containing $(1,1,\frac{\pi}{2})$ and $(1,0,\frac{3\pi}{2})$. \PGn{We use $m=25^3$ boxes to discretize the 3D phase space $M$, and $t_f=1$ with $k=20$ equally-spaced time steps, for both cases. The computation in CVX takes about $6\times 10^4$ seconds. }The initial and final measures for the two cases are depicted in Fig \ref{fig:Un_case12}. The optimal transport solution for the first case is shown in Fig. \ref{fig:OT_Un1}. Since the final orientation is prescribed to be along the $x-$axis, this leads to a splitting of the measure half-way in the transport, and steering of the two halves horizontally to their final positions. The optimal transport solution for the second case is shown in Fig. \ref{fig:OT_Un2}. The solution in this case is qualitatively different from the first case. The two halves split and then move vertically towards final positions.
\begin{figure}[h!]
\centering
\subfloat[Case 1]{\includegraphics[width=2.5in]{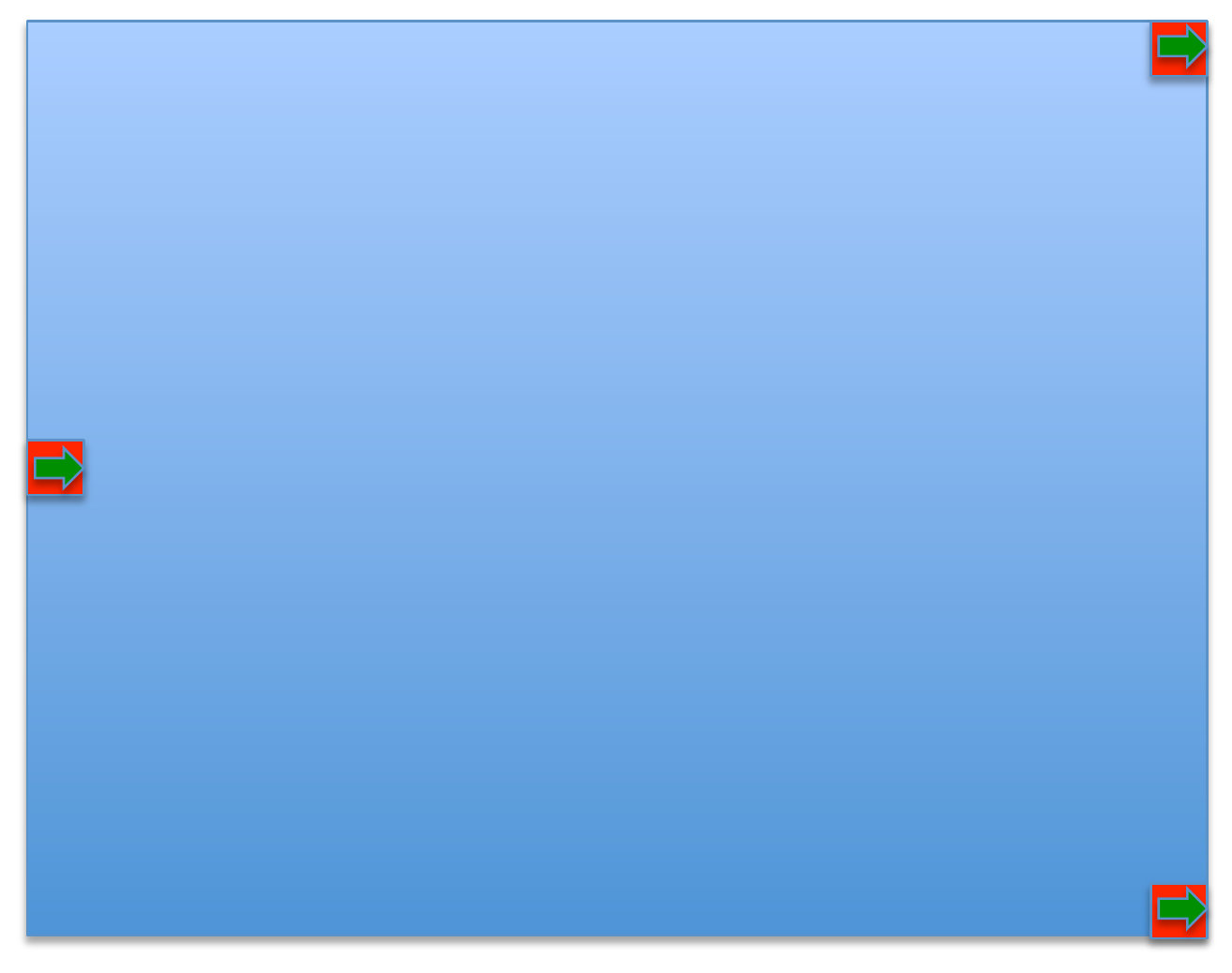}}\qquad
\subfloat[Case 2]{\includegraphics[width=2.5in]{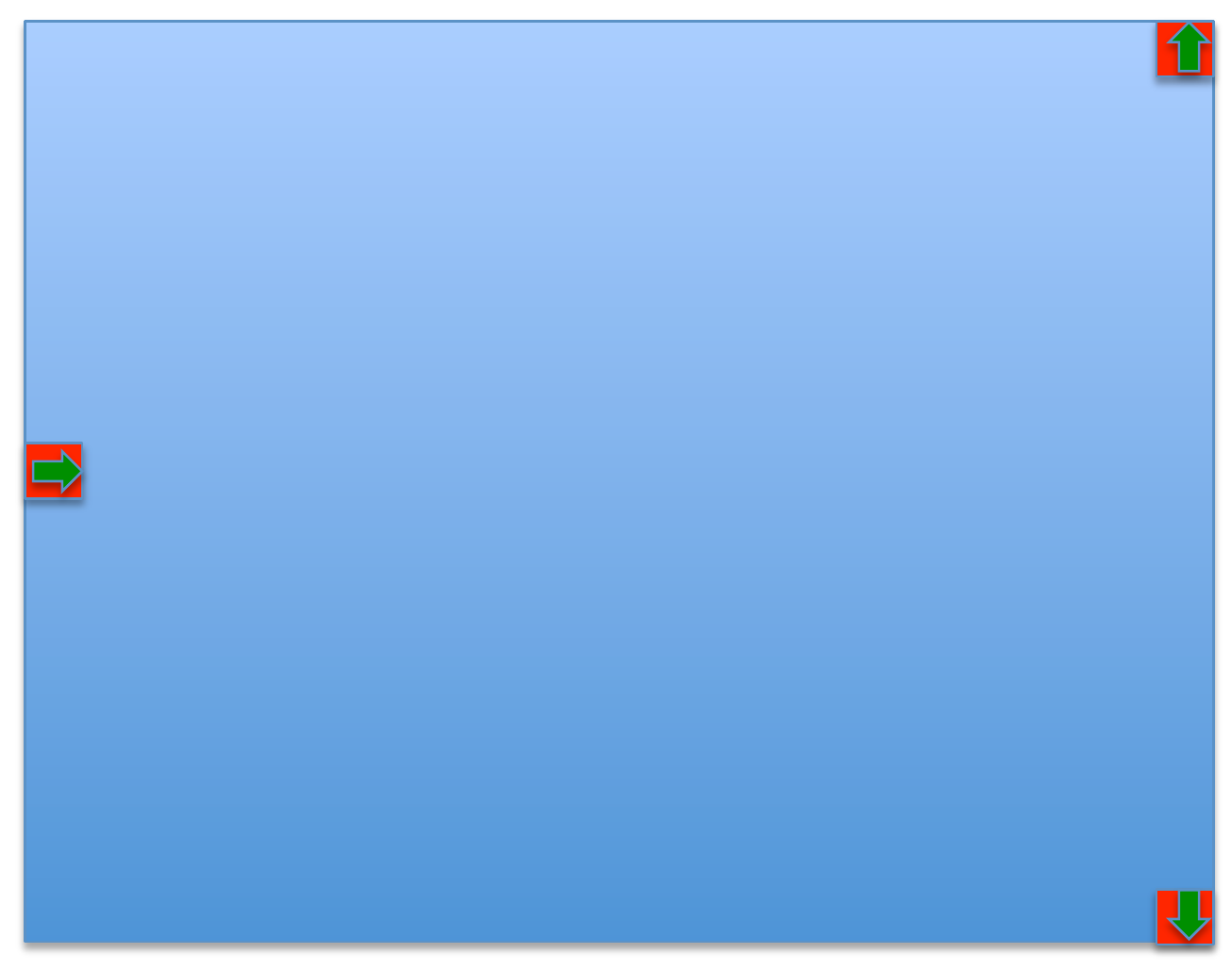}}\qquad
\caption{\footnotesize{ Initial and final measures shown on $(x,y)$ plane for two cases of optimal transport in the unicycle model. The green arrows indicate the third coordinate $\theta$. (a) $\mu_0$ is supported on $(0,0.5,0)$, $\mu_1$ is supported on $(1,0,0)$ and $(1,1,0)$. (b) $\mu_0$ is supported on $(0,0.5,0)$, $\mu_1$ is supported on $(1,0,\frac{3\pi}{2})$ and $(1,1,\frac{\pi}{2}).$}}
\label{fig:Un_case12}
\end{figure}

\begin{figure}[h!]
\centering
\subfloat[t=0]{\includegraphics[width=2.5in]{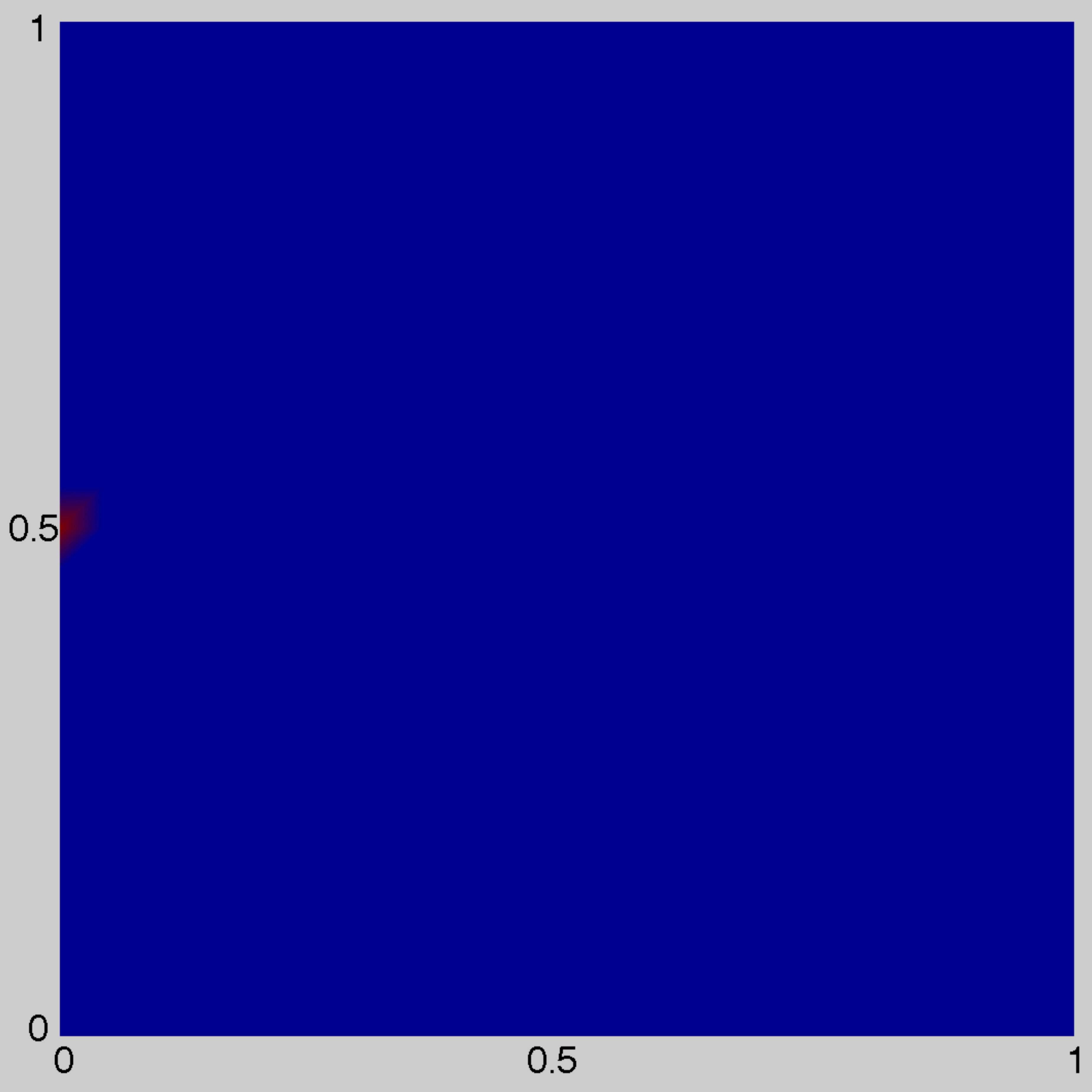}}\qquad
\subfloat[t=0.2]{\includegraphics[width=2.5in]{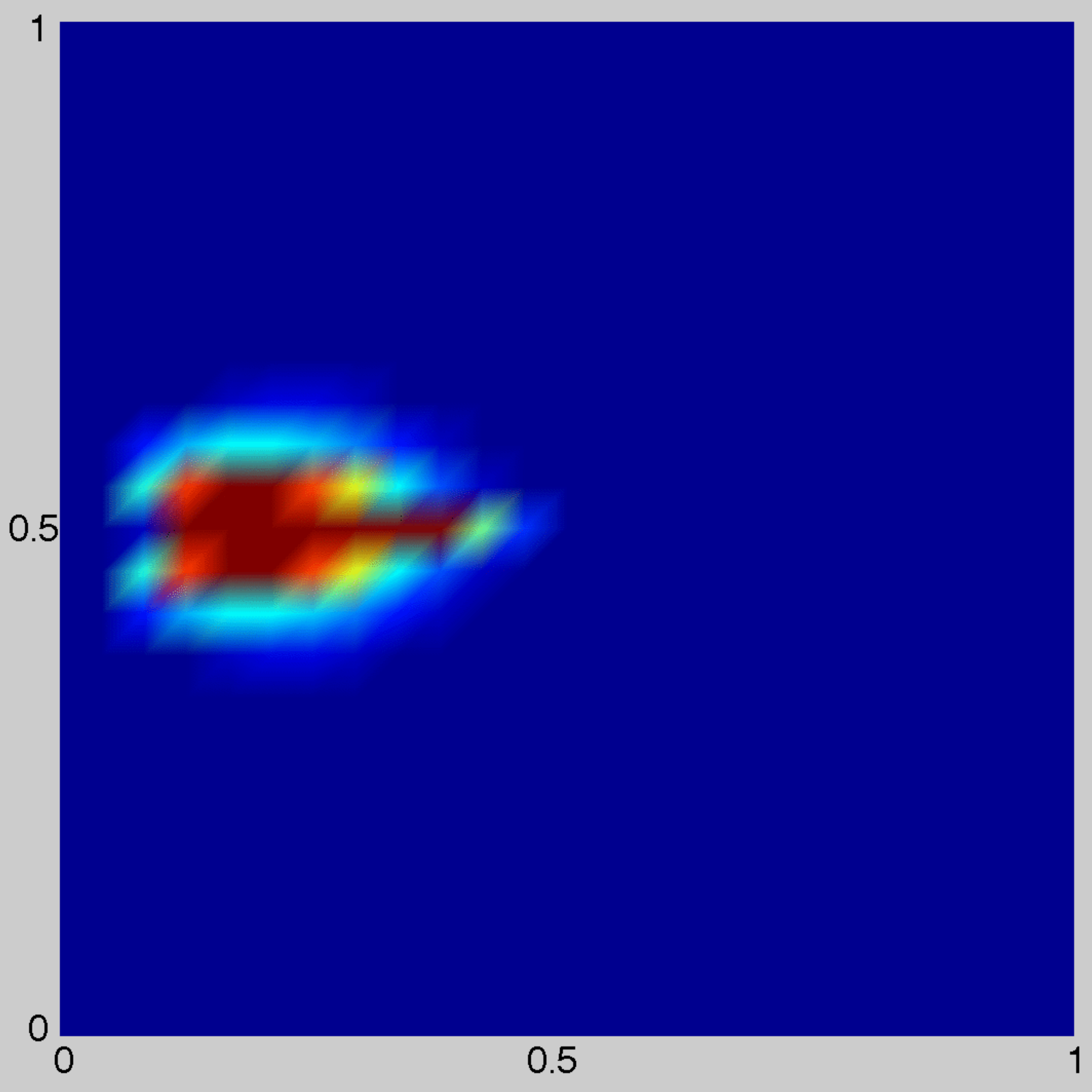}}\qquad
\subfloat[t=0.5]{\includegraphics[width=2.5in]{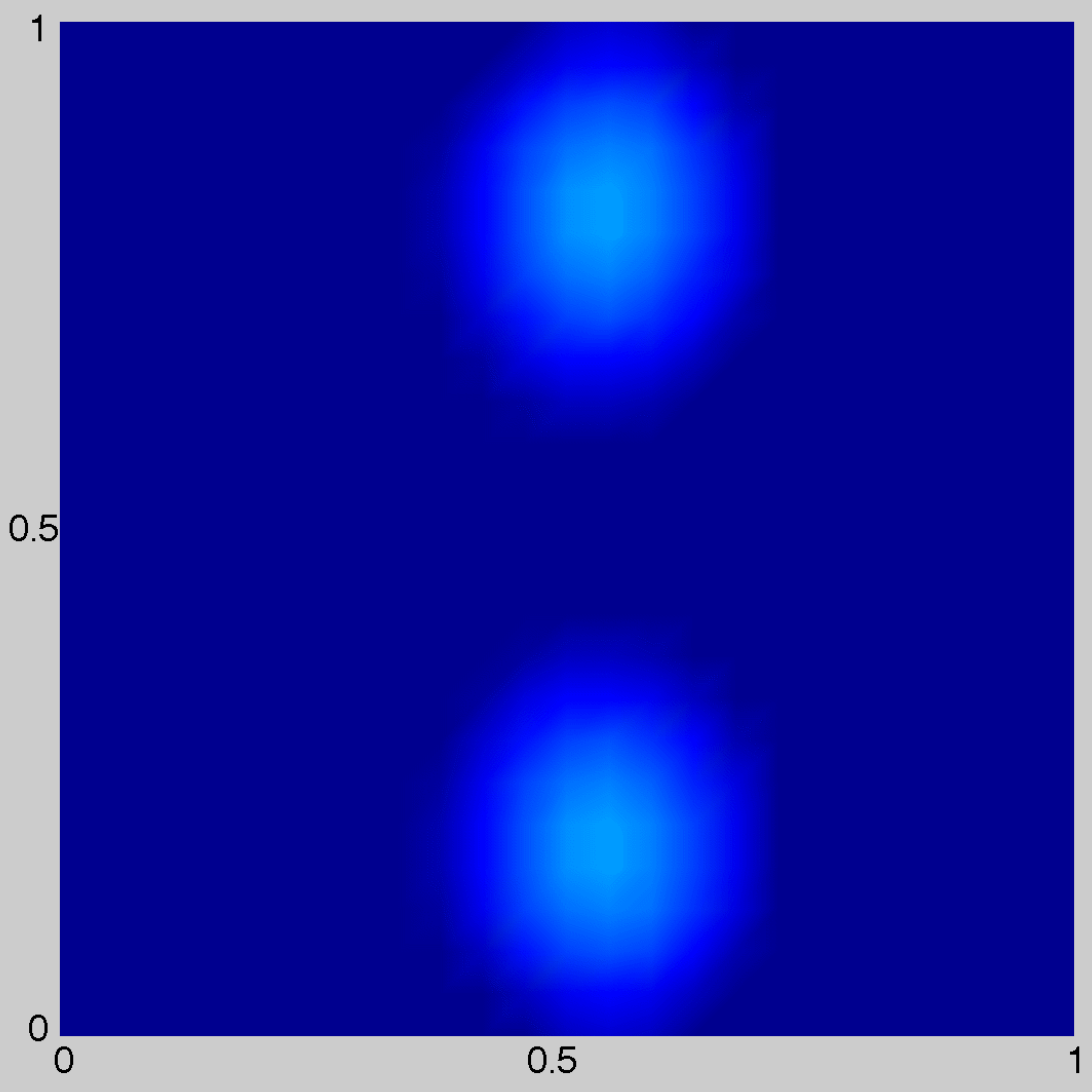}}\qquad
\subfloat[t=0.7]{\includegraphics[width=2.5in]{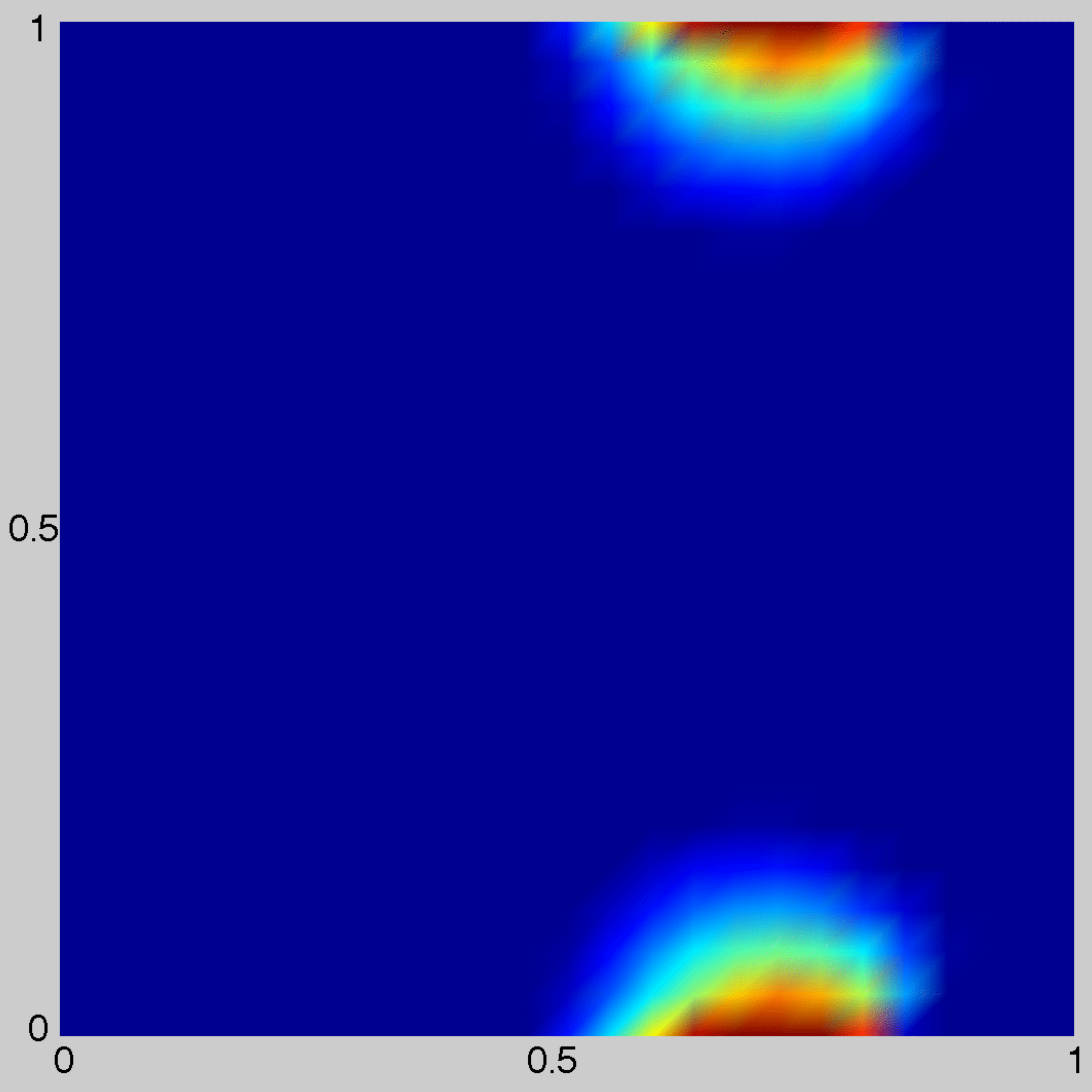}}\qquad
\subfloat[t=0.8]{\includegraphics[width=2.5in]{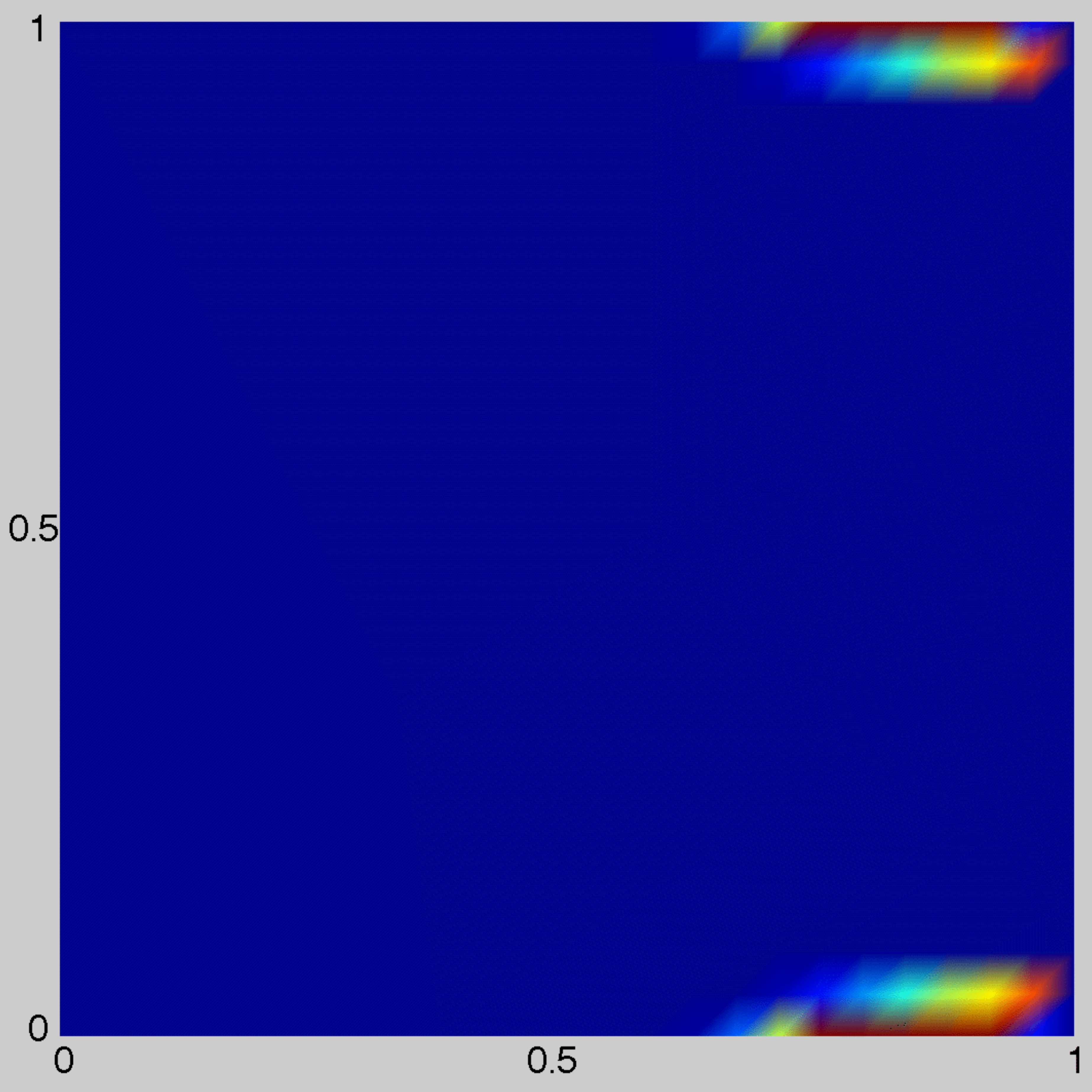}}\qquad
\subfloat[t=1]{\includegraphics[width=2.5in]{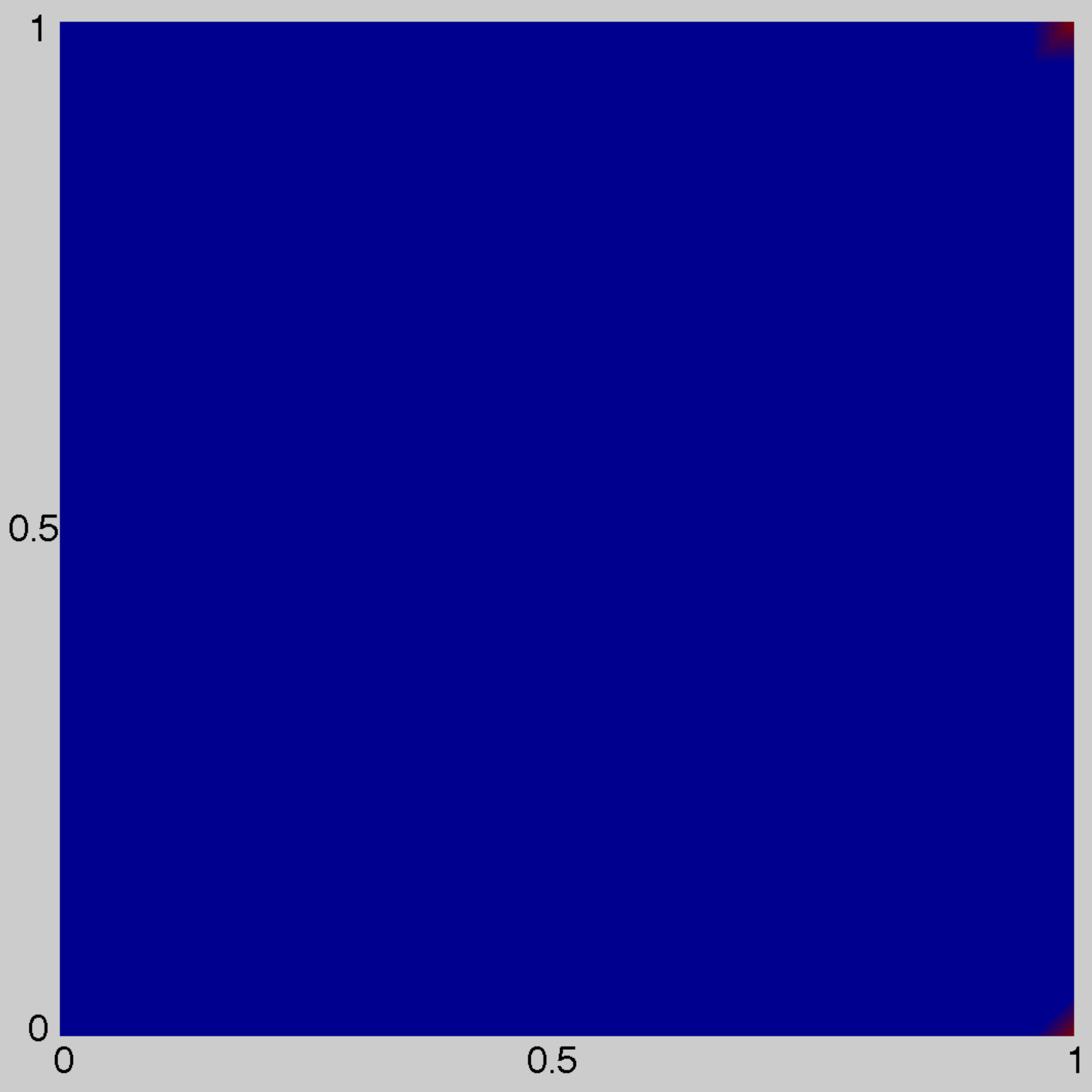}}\qquad
\caption{\footnotesize{The optimal transport solution of unicycle model shown in the $x-y$ plane for  case 1. The grid size is $m=25^3$, and $k=20$. }}
\label{fig:OT_Un1}
\end{figure}

\begin{figure}[h!]
\centering
\subfloat[t=0]{\includegraphics[width=2.5in]{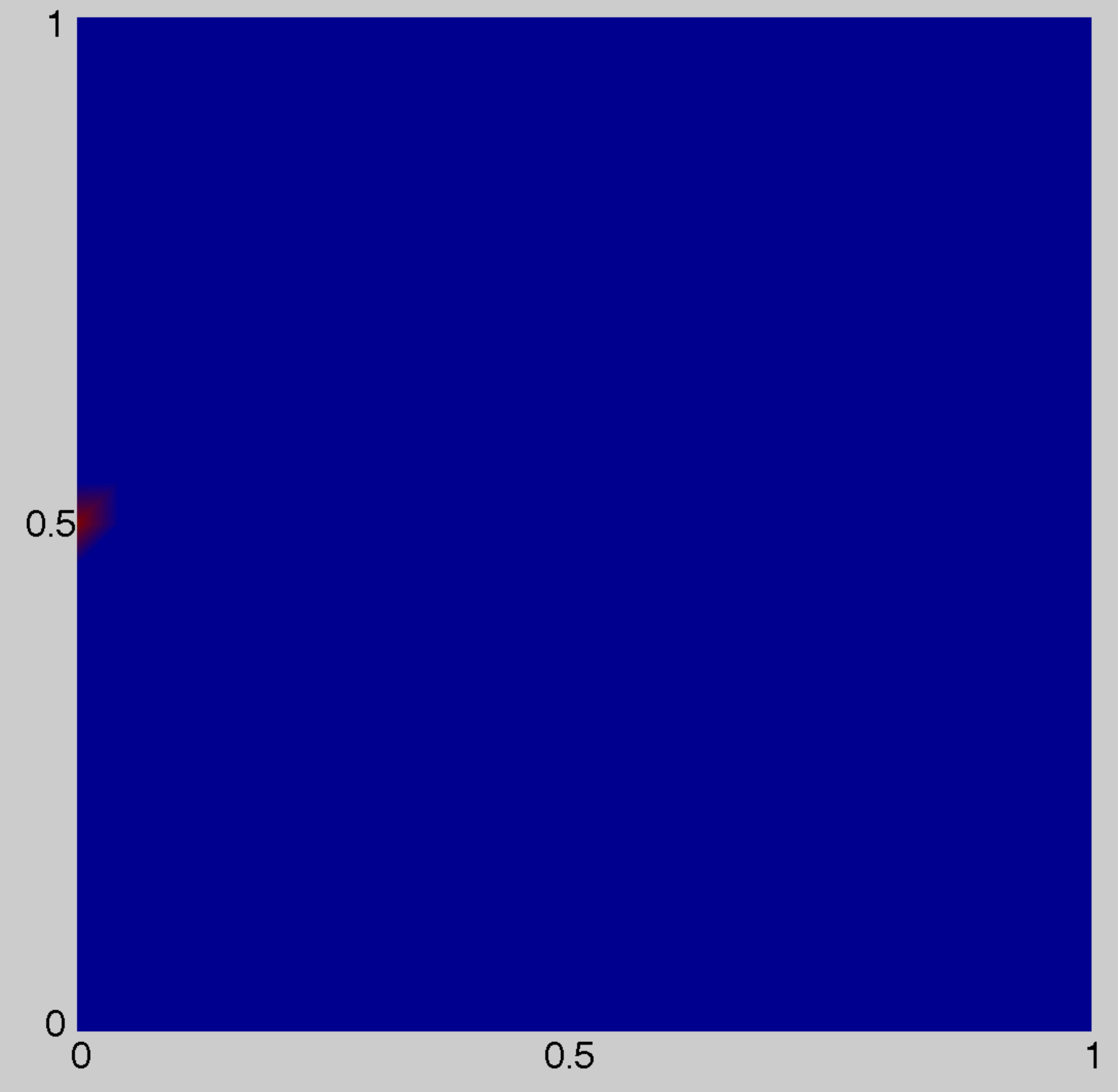}}\qquad
\subfloat[t=0.2]{\includegraphics[width=2.5in]{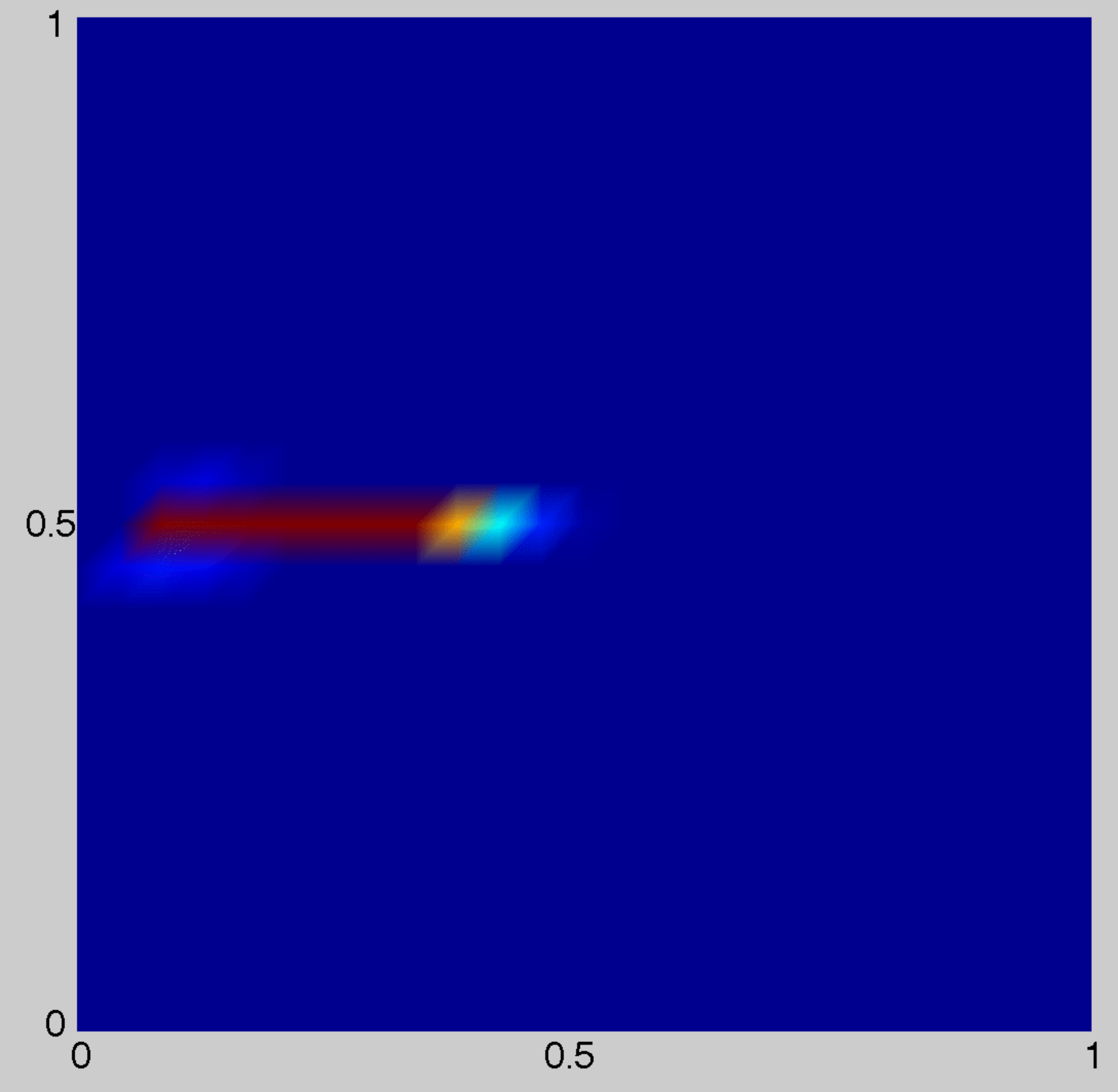}}\qquad
\subfloat[t=0.5]{\includegraphics[width=2.5in]{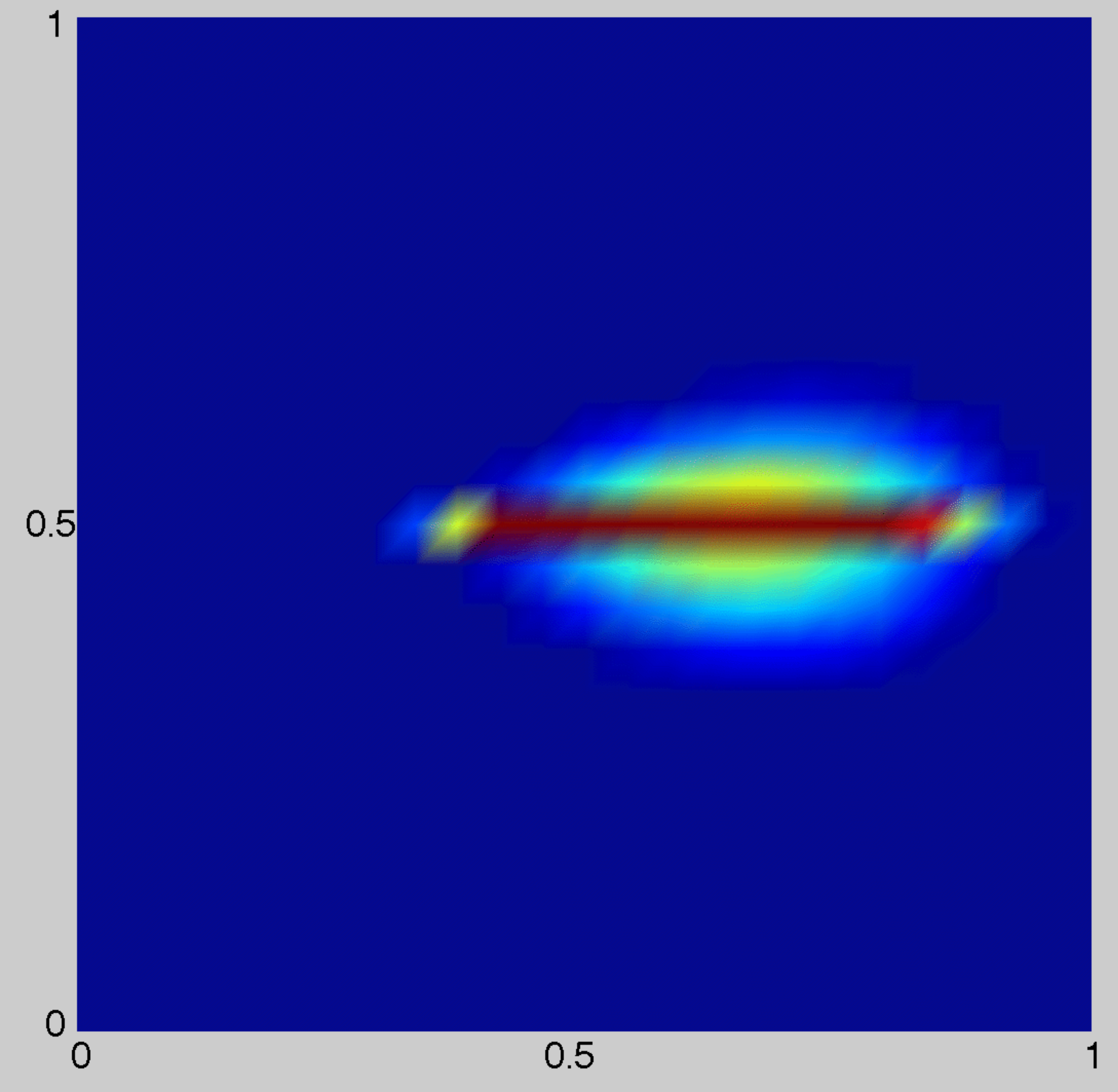}}\qquad
\subfloat[t=0.7]{\includegraphics[width=2.5in]{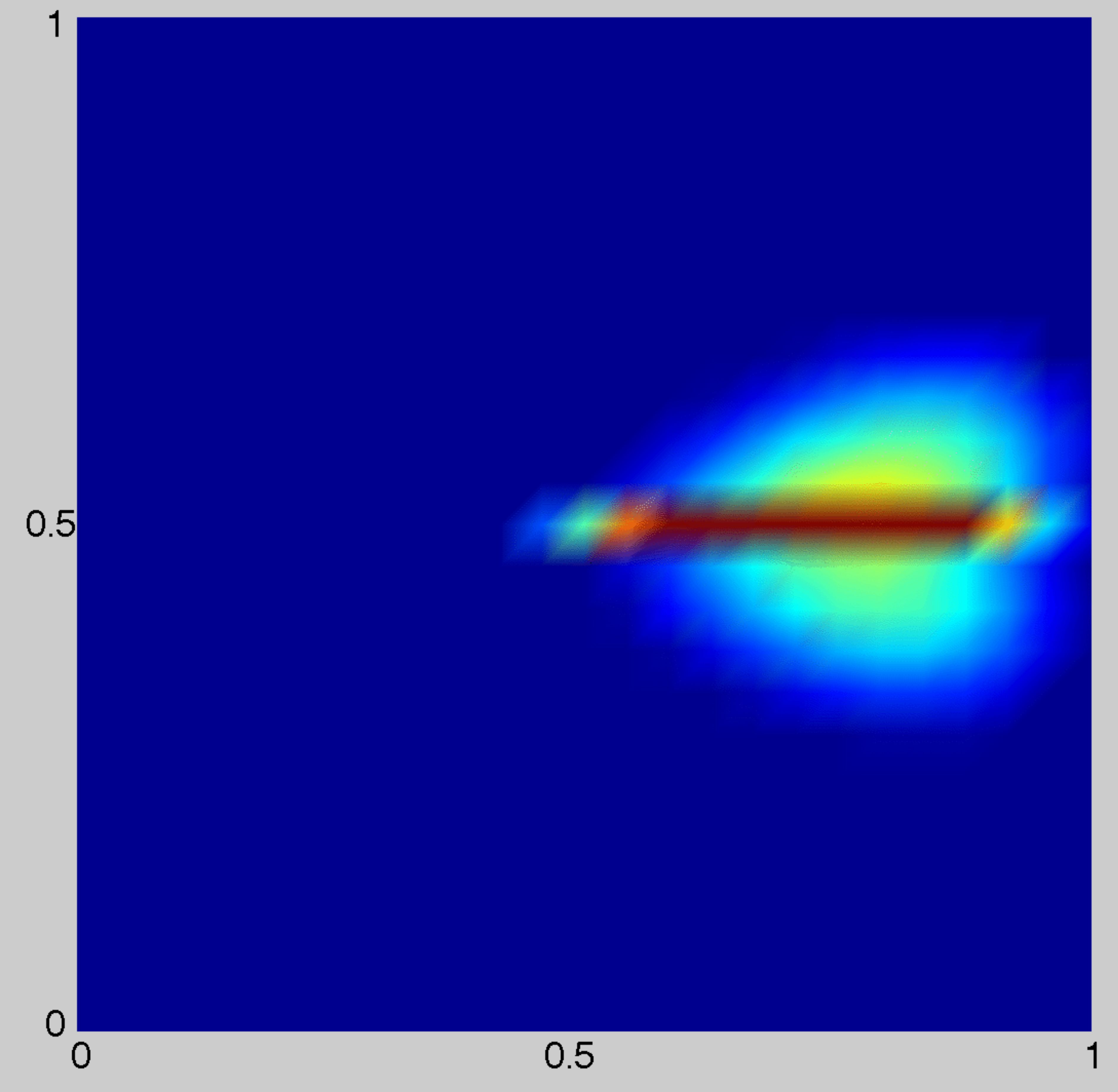}}\qquad
\subfloat[t=0.8]{\includegraphics[width=2.5in]{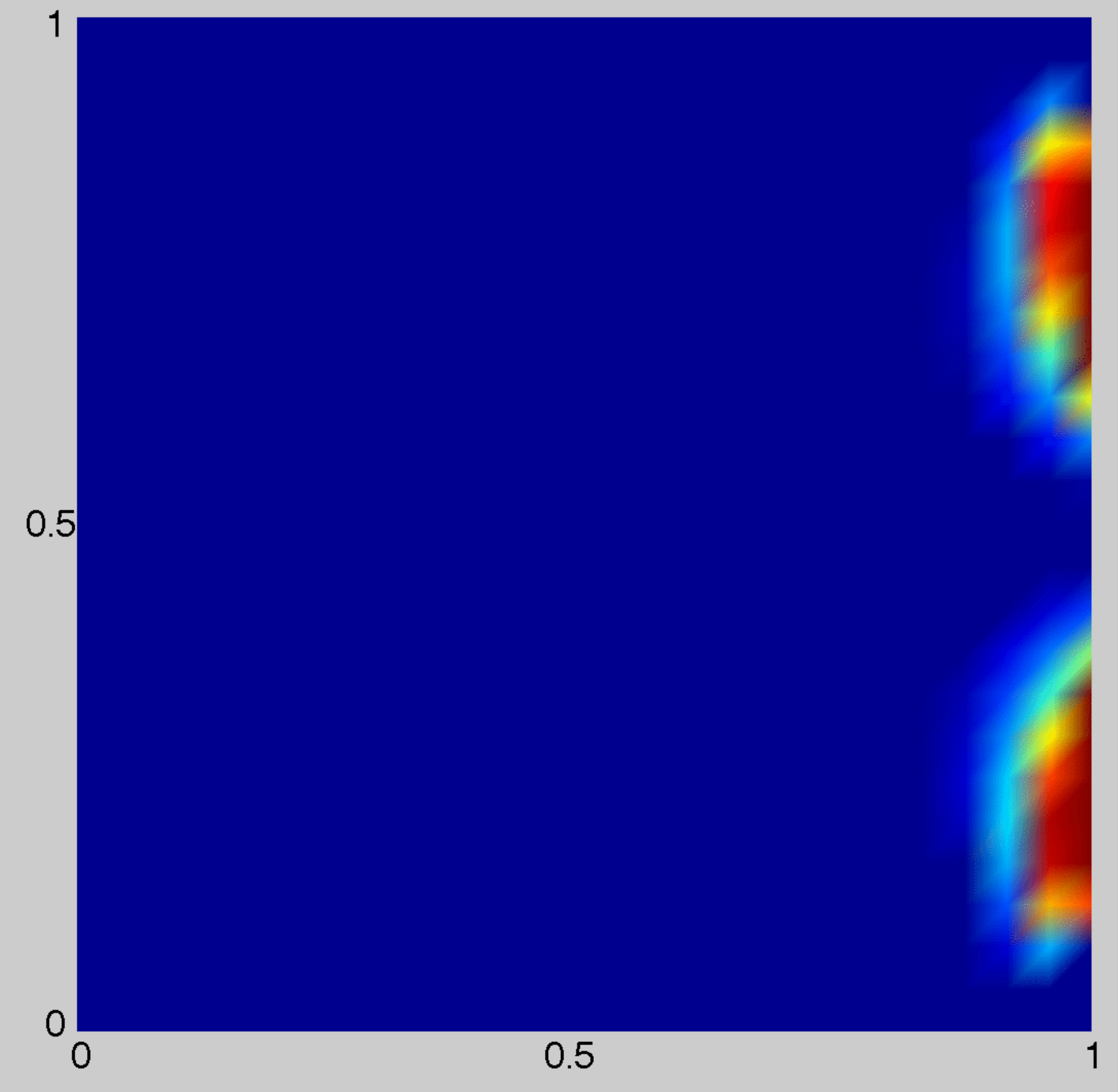}}\qquad
\subfloat[t=1]{\includegraphics[width=2.5in]{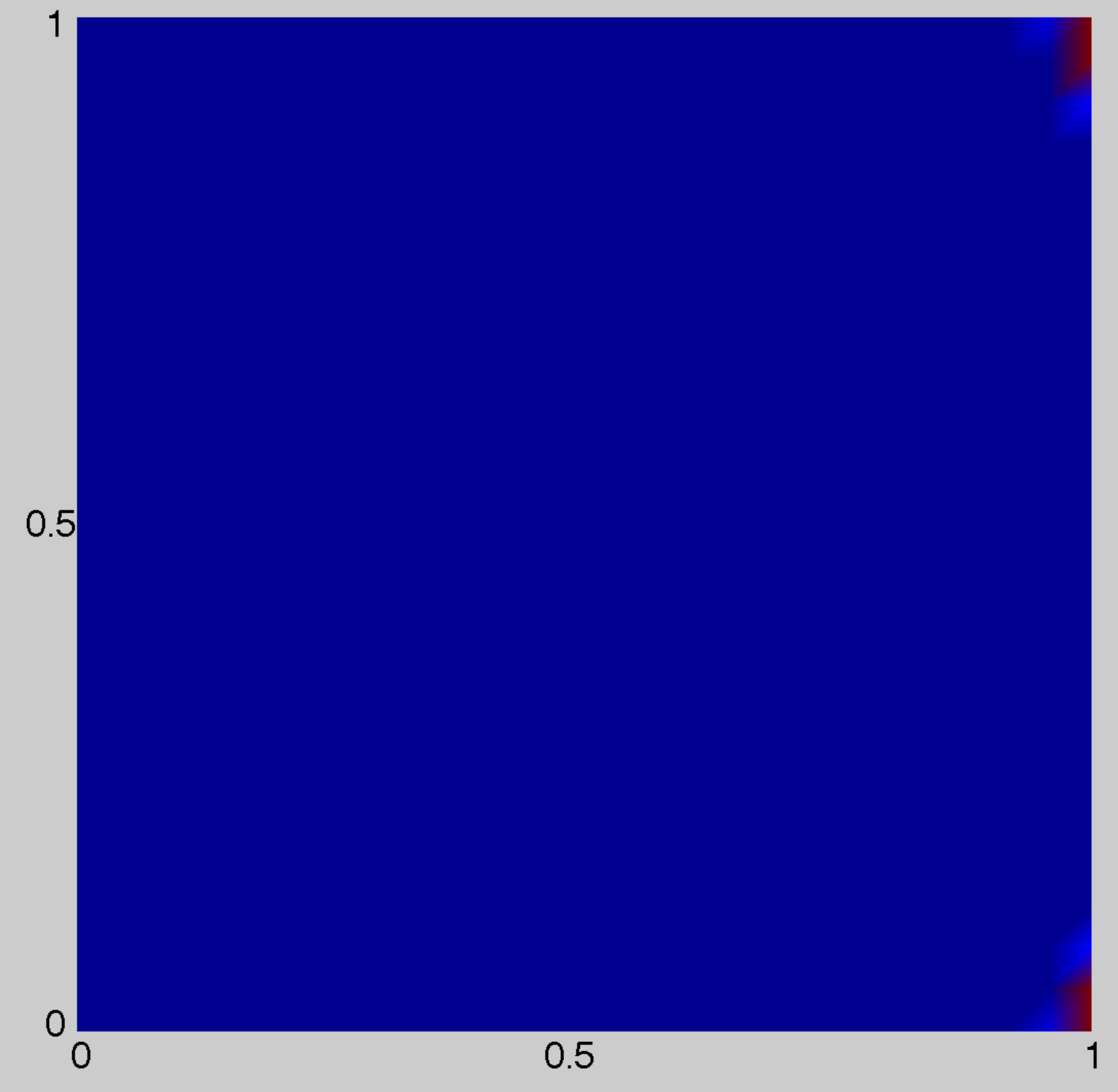}}\qquad
\caption{\footnotesize{ The optimal transport solution of unicycle model shown in the $x-y$ plane for  case 2. The grid size is $m=25^3$, and $k=20.$}}
\label{fig:OT_Un2}
\end{figure}

\section{Conclusions and Future Directions}\label{sec:Conclusions}
\PG{A set-oriented graph-based computational framework for continuous-time optimal transport over nonlinear dynamical systems has been developed. In the control systems setting, this framework generalizes the concept of optimal transport on graphs from that of a single integrator to control-affine nonlinear systems. This is accomplished by exploiting recent work on set-oriented infinitesimal generator approach for nonlinear dynamical systems. The controllability of measures over graphs is related to the connectivity of the `controlled' graph, and is proved to be a consequence of controllability of the underlying control system. This work connects set-oriented operator-theoretic methods in dynamical systems with optimal mass transportation theory, and opens up new directions in design of efficient feedback control strategies for nonlinear multi-agent and swarm systems operating in nonlinear ambient flow fields. }

Application of our set-oriented framework to larger domains, longer time-horizons and/or higher dimensional systems will require improvement in computational efficiency. Using efficient phase space discretization techniques, such as those employed in GAIO \cite{DeFrJu2001}, one can hope to improve the efficiency of the resulting optimal transport algorithms, and apply the framework to higher dimensional dynamical systems. Graph pruning algorithms can be employed to remove edges which are not likely to be used \cite{harabor2011online}.

Solutions to the optimal transport problem in the double-gyre system elucidate the role played by invariant manifolds, lobe-dynamics and almost-invariant sets in efficient transport of phase-space distributions. While it is known that invariant manifolds and lobes act as low-energy `channels' in the phase-space, our results give new qualitative and quantitative information about their role in problems of transport of distributions or swarms of agents. Application of this framework to more complicated arbitrary time-varying flows should provide similar insights into the role of Lagrangian coherent structures and coherent sets. This can lead to development of efficient swarm planning and control strategies for realistic applications in ocean and air-borne systems. Moreover, using our framework, the relative importance of such objects can be studied for different types of controls. 

Furthermore, recent methods in obtaining Lagrangian coherent structures and coherent sets in finite-time non-autonomous systems have used variational formulations of transport under nonlinear dynamics \cite{haller2015lagrangian} or  dynamic versions of the Laplacian \cite{froyland2015dynamic}. It would be fruitful to develop connections of these formulations with optimal mass transportation theory, extending the connections already identified in the Hamiltonian dynamics case \cite{bernard2004optimal}. For instance, one could define a controlled version of almost-invariant sets or coherent sets, by defining a control dynamic Laplacian, analogous to the control infinitesimal-generators as developed in the current work, or control Lyapunov measures developed in Ref. \cite{vaidya2008lyapunov}.

Connections with work in the closely related area of occupation measures \cite{lasserre2008nonlinear} and Lyapunov measures \cite{vaidya2008lyapunov} also need to be explored, especially in context of obtaining feedback control laws from the control laws obtained from optimal transport solutions. \KE{The feedback control laws constructed as solutions to the optimal transport problem guide the measure along shortest paths corresponding to solutions of the corresponding sub-Riemannian problem. Hence, it needs to be explored in what sense these laws can be used for feedback stabilization of an individual control-affine system. Moreover, the numerical approach in the current paper could also be adapted to construct time-independent feedback control laws for such systems.}

\clearpage

\section{Acknowledgements}
We thank Matthias Kawski, Uro\v{s} Kalabi\`{c} and Udit Halder for helpful discussions on geometric control theory. We also thank the reviewers for helpful suggestions.
\PGn{\appendix
\section{Proof of Theorem \ref{thm:strngcon}}
\begin{figure}[h]
\centering
\includegraphics[width=4in]{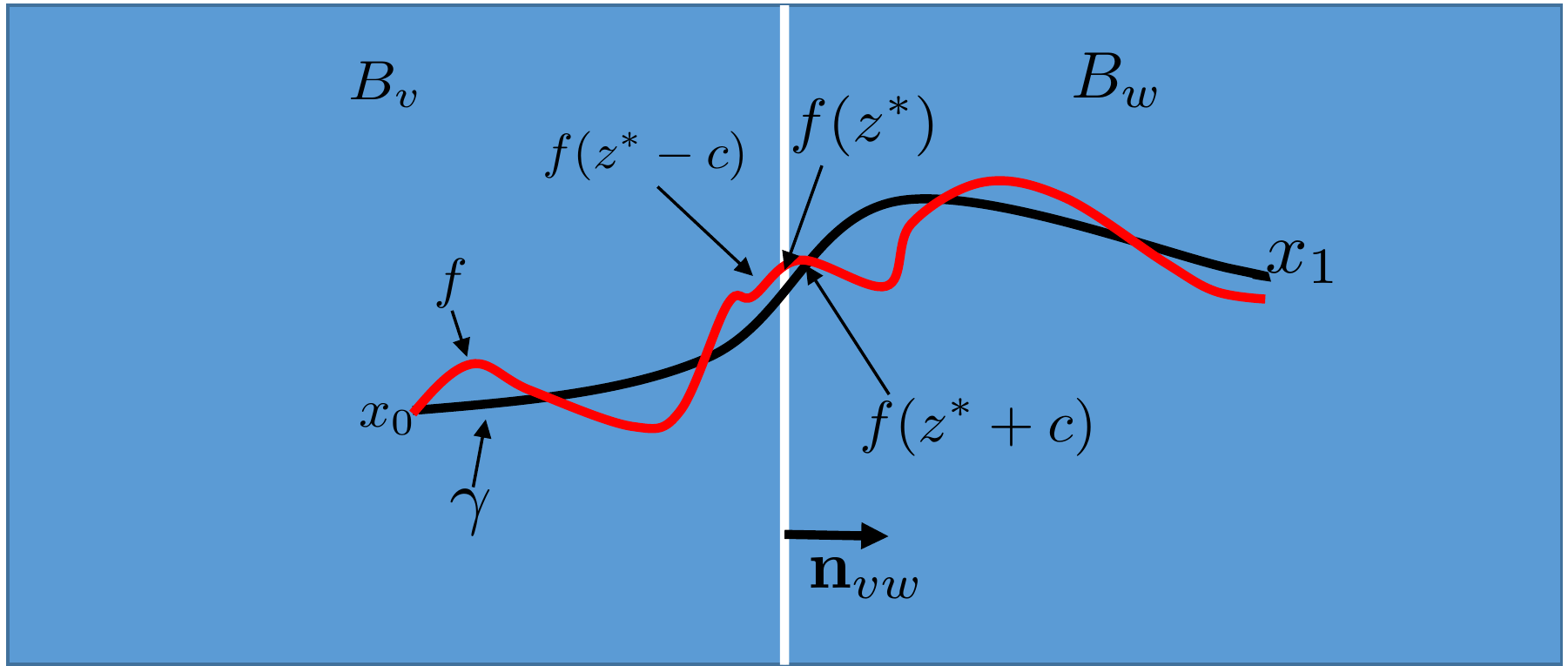}
\caption{\footnotesize{\PGn{Illustration of the proof of Theorem \ref{thm:strngcon}. The existence of an trajectory $f$ approximating a curve $\gamma$ connecting $B_v$ to $B_w$, obtained by piecewise constant control, is guaranteed by the small time local controllability. By continuity, this leads to non-zero transition rates, and hence strong connectivity of the control graph $\mathcal{G}_c$.}}}
\label{fig:con}
\end{figure}

\begin{proof}
Let $v,w \in \lbrace 1,2....m \rbrace$ be such that $v \neq w$ and $\bar{B_v} \cap \bar{B_w}$ has non-zero $(d-1)-$ dimensional (Hausdorff) measure. Consider points $x_0 \in int(B_v)$ and $x_1 \in int(B_w)$. Due to connectedness of M  there exists a continuous path $\gamma: [0,1] \rightarrow M$ such that $\gamma(0) = x_0$, $\gamma(1) = x_1$ and $\gamma(t) \in B_v \cup B_w$ $\forall t \in [0,1]$. From the Lie bracket condition of the vector fields, it follows that the system is small-time locally controllable at every $x \in int(M)$. Then, we can approximate the path $\gamma$ using a trajectory of the control system, using a sequence of piecewise-constant control inputs. 

To construct such a sequence, let us denote the flow map for time period $t$ under an autonomous vector field $X$ by $e^{tX}$. Then, for each $\epsilon >0$ there exists $k \in \mathbb{N}$ large enough, a sequence of time intervals $t_1, t_2......t_k $ satisfying $\sum_{i=1}^k t_i = 1$, constant control inputs $u^{1}, u^{2},.....u^{k} \in \mathbb{R} $, a set of indices $\eta_i \in \lbrace 1,2....n\rbrace$ selecting the corresponding control vector field $g_{\eta_i}$, and an approximating path,  $f:[0,1] \rightarrow M$ satisfying $\| \gamma(z)-f(z) \|^2_2 \leq \epsilon$ for all $z \in [0,1]$. The path $f(z)$ for $z \in [0,1]$ can be written using concatenation of flow under the action of chosen sequence of control vector fields :

$f(z=\sum_{i=1}^j t_i+\tau) =e^{\tau u^{j+1}g_{\eta_{j+1} } }  \circ ......\circ e^{t_j u^{j}g_{\eta_j}} \circ  e^{t_1 u^{1} g_{\eta_1}} x_0$, for each  $j \in \{ 0,1,...k\}$ and  $\tau \in [0,t_{j+1}]$. Here, the case $j=0$ means $f(\tau)  = e^{\tau u^1 g_{\eta_1}}{x_0}$ for all $\tau \in [0,t_1]$.

 Let $z^* \in (0,1)$ be such that  $f(z^*) \in \partial B_v$ and there exists $c \in (0,z^*)$ small enough such that $f(z^*-c) \in int (B_v) $ and $f(z^*+c) \in int(B_w)$. Then, clearly $\mathbf{n}_{vw} \cdot g_r(x) \neq 0 $ for some $ r \in  \lbrace 1,2....n\rbrace$ and some $x$ in an open neighborhood of $f(z^*)$ that is completely contained in $B_v \cup B_w$, assuming $\gamma$ and $\epsilon$ are chosen appropriately (i.e. avoiding crossings of $\gamma$ and $f$ over corners of $B_v$ and $B_w$). If not, $f(z^*+\delta) \in \partial B_i$ for all $\delta \in (0,c]$ since the non-existence of such a point $c$ with the desired property in the neighborhood of $f(z^*)$ implies one cannot use a concatenation of flows associated with the control vector fields to leave the set $\partial B_v$, which leads to a contradiction to the assumed property of small time local controllability. From continuity of the vector field $g_r$, there exists a small enough neighborhood, $N_x$ of $x$ such that $\mathbf{n}_{vw} \cdot g_r(y) \neq 0 $ for all $y \in N_x$. Hence, this implies $A^s_r(e) \neq 0$ for $e = v \rightarrow w$ for some $s \in \lbrace +,-\rbrace$. Due to continuity of the vector field $g_r$ at x, it also follows that $A^s_r(e) = A^s_r(\bar{e})$. Hence, the connectivity of the graph $\mathcal{G}_c$ follows. 
Case 2 just follows from the assumption that $span \bigg \lbrace g_i(x): i \in \lbrace 1,2....n \rbrace \bigg \rbrace = T_xM$ at each $x \in int(M)$.
\end{proof}}

\clearpage
\bibliographystyle{amsplain}
\bibliography{refs}

\end{document}